\documentclass[a4paper,onecolumn,oneside,10pt]{article}

\usepackage{geometry}
\usepackage[x11names,usenames,dvipsnames]{xcolor} 
\usepackage{tikz}
\usepackage[T1]{fontenc}
\usepackage{comment} 
\usepackage{lineno, amsthm}    
\usepackage{mathtools} 
\usepackage{mathdots}
\usepackage{enumerate}
\usepackage{textcomp}
\usepackage{amssymb,amsfonts}
\usepackage{amsmath,paralist,enumitem,mathrsfs,graphicx}
\usepackage{epstopdf}
\usepackage{pgfplots}
\newcommand\norm[1]{\left\lVert#1\right\rVert}

\usetikzlibrary{pgfplots.dateplot} 
\renewcommand{\qedsymbol}{$\blacksquare$}
\newcommand\restr[2]{{
		\left.\kern-\nulldelimiterspace 
		#1 
		\vphantom{\big|} 
		\right|_{#2} 
}}
\makeatletter 
\newcommand{\vas}{\bBigg@{3}}
\newcommand{\vast}{\bBigg@{4}}
\newcommand{\Vast}{\bBigg@{5}} 
\makeatother
\newcommand{\upperRomannumeral}[1]{\uppercase\expandafter{\romannumeral#1}}
\newcommand{\lowerRomannumeral}[1]{\lowercase\expandafter{\romannumeral#1}}
\usepackage{indentfirst}
\usepackage{hyperref}
\usepackage{appendix}

\usepackage{epstopdf}
\usepackage[caption=false]{subfig}
\theoremstyle{plain}
\newtheorem{theorem}{Theorem}
\newtheorem{hyp}{Hypothesis}
\newtheorem{lemma}[theorem]{Lemma}
\newtheorem{corollary}[theorem]{Corollary}
\newtheorem{proposition}[theorem]{Proposition}

\theoremstyle{definition}
\newtheorem{rem}{Remark}
\newtheorem{defini}{Definition}

\modulolinenumbers[2]
\newenvironment{myproof}[1] {{
		
		\noindent\it{Proof of {#1}}. }}{\hfill\qedsymbol}
\makeatletter
\def\ps@pprintTitle{ 
	\let\@oddhead\@empty
	\let\@evenhead\@empty
	\def\@oddfoot{\footnotesize\itshape
		\ifx\@empty\@empty
		\else\@journal\fi\hfill\today}%
	\let\@evenfoot\@oddfoot
}
\makeatother

\def\nor{\color{black}}
\def\R{\mathbb R}
\makeatletter
\let\save@mathaccent\mathaccent
\newcommand*\if@single[3]{%
	\setbox0\hbox{${\mathaccent"0362{#1}}^H$}%
	\setbox2\hbox{${\mathaccent"0362{\kern0pt#1}}^H$}%
	\ifdim\ht0=\ht2 #3\else #2\fi
}
\newcommand*\rel@kern[1]{\kern#1\dimexpr\macc@kerna}
\newcommand*\widebar[1]{\@ifnextchar^{{\wide@bar{#1}{0}}}{\wide@bar{#1}{1}}}
\newcommand*\wide@bar[2]{\if@single{#1}{\wide@bar@{#1}{#2}{1}}{\wide@bar@{#1}{#2}{2}}}
\newcommand*\wide@bar@[3]{%
	\begingroup
	\def\mathaccent##1##2{%
		\let\mathaccent\save@mathaccent
		\if#32 \let\macc@nucleus\first@char \fi
		\setbox\z@\hbox{$\macc@style{\macc@nucleus}_{}$}%
		\setbox\tw@\hbox{$\macc@style{\macc@nucleus}{}_{}$}%
		\dimen@\wd\tw@
		\advance\dimen@-\wd\z@
		\divide\dimen@ 3
		\@tempdima\wd\tw@
		\advance\@tempdima-\scriptspace
		\divide\@tempdima 10
		\advance\dimen@-\@tempdima
		\ifdim\dimen@>\z@ \dimen@0pt\fi
		\rel@kern{0.6}\kern-\dimen@
		\if#31
		\overline{\rel@kern{-0.6}\kern\dimen@\macc@nucleus\rel@kern{0.4}\kern\dimen@}%
		\advance\dimen@0.4\dimexpr\macc@kerna
		\let\final@kern#2%
		\ifdim\dimen@<\z@ \let\final@kern1\fi
		\if\final@kern1 \kern-\dimen@\fi
		\else
		\overline{\rel@kern{-0.6}\kern\dimen@#1}%
		\fi
	}%
	\macc@depth\@ne
	\let\math@bgroup\@empty \let\math@egroup\macc@set@skewchar
	\mathsurround\z@ \frozen@everymath{\mathgroup\macc@group\relax}%
	\macc@set@skewchar\relax
	\let\mathaccentV\macc@nested@a
	\if#31
	\macc@nested@a\relax111{#1}%
	\else
	\def\gobble@till@marker##1\endmarker{}%
	\futurelet\first@char\gobble@till@marker#1\endmarker
	\ifcat\noexpand\first@char A\else
	\def\first@char{}%
	\fi
	\macc@nested@a\relax111{\first@char}%
	\fi
	\endgroup
}

\def\vv {\vskip 1mm}

\def\bb{{\bf C}}
\def\qq{{\bf Q}}
\def\pj{{\bf  \text{proj}_{\bb} \, }}

\title{ Regular stochastic flow and Dynamic Programming Principle for jump diffusions\footnote{The research of Alessandro Bondi benefited from the financial support of the chair ``Statistiques et Modèles pour le Régulation'' of École polytechnique. The study for this paper began during the Ph.D. of Alessandro Bondi at Scuola Normale Superiore di Pisa, which the authors thank. Enrico Priola  is a  member of GNAMPA  of the Istituto Nazionale di Alta Matematica (INdAM). }}  
\author{Alessandro Bondi\thanks{Centre de Mathématiques Appliquées (CMAP), CNRS, École polytechnique, Institut Polytechnique de Paris. \textbf{Email:} alessandro.bondi@polytechnique.edu} 
	\and {Enrico Priola\thanks{Dipartimento di Matematica, Università di Pavia. \textbf{Email:} enrico.priola@unipv.it}}}

\begin{document}

\maketitle

	\begin{abstract} 
Given a  Brownian motion $W$ and a stationary Poisson point process $p$ with values in ${\mathbb R}^d$, we prove a  Dynamic Programming Principle (DPP)  in a strong formulation  for  a stochastic control problem involving 
controlled  SDEs of the form 
\begin{align} \label{ci1} 
\notag dX_{t}=&\,b(t, X_{t}, a_t) dt + 
\alpha \left(t, X_{t}, a_t \right) dW_t+ \!
\! 
\int_{  |z| \le 1}  g\left(X_{t-},t,z, a_t \right)\widetilde{N}_p\left(dt,dz\right)  
\\  & +  
\int_{  |z| >1  } f\left(X_{t-},t,z, a_t \right){N}_p\left(dt,dz\right), \quad \;  X_s=x\in\mathbb{R}^d,\,0\le s \le t \le T.\vspace{-.5em} 
\end{align}     
Here  $N_p$  [resp., $\widetilde{N}_p$] is the Poisson [resp., compensated Poisson] random measure associated with $p$.  We consider arbitrary  predictable controls $a  \in {\mathcal P}_T$ with values in a closed convex set $C \subset \R^{l}$. 
The coefficients  {$b$, $\alpha$, and $g$}    satisfy  linear growth and Lipschitz--type  conditions in the $x-$variable,  and are continuous in the control variable. 
 To prove  the DPP for the value function $ v(s,x)=\sup_{a \in {\mathcal P}_T} \, \mathbb{E}\big[\int_{s}^{T}h\left(r,X_r^{s,x,a}, a_r\right)dr + 
 j\left(X_T^{s,x,a}\right)\big]     
 $, 
  assuming that $h$ and $j$ are bounded and continuous,
  we establish the existence of a regular stochastic flow for \eqref{ci1} when the coefficients    are independent of the control $a$. Notably, this regularity result is new even
  when there is no 
  large--jumps component, i.e., $f\equiv0$ (cf.  Kunita's recent book on stochastic flows).   
The proof of the DPP is completed by introducing an approach that relies  on a suitable subclass of finitely generated step controls in $\mathcal{P}_T$.
 These controls allow us  to  apply a basic measurable selection theorem by L. D. Brown and R. Purves. We believe that this novel method is of independent interest and could be adapted  to prove  DPPs arising in other stochastic control problems.
\\[1ex] 
\end{abstract}
   
{\small\tableofcontents}

\section{Introduction}\label{Intro}
In this paper we prove a new dynamic programming principle (DPP)  in a strong form.   The DPP, also known in the literature as \emph{Bellman's principle of optimality}, is a fundamental concept in the theory of stochastic control.  It can be seen as an extension of the tower property of Markov processes in the context of optimization. This principle also allows to characterize the optimal control  process, to show that the value function is a viscosity solution of the associated Hamilton--Jacobi--Bellman (HJB) equation and to derive numerical resolution schemes; 
 see, for instance, \cite{archi, BS,touzi, cosso, CP, ISHI, karoui, Kr, Pham, PW, praga, YZ}.  For  applications in physics and mathematical finance
 we refer to \cite{archi, cosso, ISHI, Pham,  YZ} and references therein. \\
  The DPP obtained in this work concerns controlled jump diffusions,
  for which, as already pointed out in the introduction of \cite{ISHI}, the literature is limited and not sufficiently detailed,  even in the case   when  there is no  large--jumps component, i.e., when $f\equiv 0$ in \eqref{anchelei0} (cf. Remark \ref{compare}).  
  To prove it, we also  need to establish original  results on the existence of  a regular stochastic flow. This   is a delicate issue,  mainly because  one cannot use the well--known Kolmogorov--Chentsov test  in the jump case under consideration (see  the comments around \eqref{kol}). 
   
We start with  a Brownian motion $W$ and a stationary Poisson point process $p$ with values on  a Polish space $(U,\mathcal{U})$, where $\mathcal{U}$ is the  Borel $\sigma-$algebra, and characteristic measure $\nu(dz)$. We denote by $N_p$ [resp., $\widetilde{N}_p$] the Poisson [resp., compensated Poisson] random measure associated with $p$.  
Given a set $U_0\in\mathcal{U}$ such that $\nu(U\setminus U_0)<\infty$, we consider 
   controlled stochastic differential equations (SDEs) of the form 
 \begin{multline}\label{anchelei0} 
 X^{s,x, a}_{t} = x+ \int_{s}^{t}b\left(r, X_{r}^{s,x, a},a_r \right) dr + 
 \int_{s}^{t}\alpha\left(r,X_{r}^{s,x, a} ,a_r  \right) dW_r\\+
 \int_{s}^{t}\!\int_{U_0}g\left(X_{r-}^{s,x, a} ,r,z,a_r\right)\widetilde{N}_p\left(dr,dz\right) 
 +  
 \int_{s}^{t}\!\int_{U\setminus U_0 } f\left(X_{r-}^{s,x, a} ,r,z,a_r \right){N}_p\left(dr,dz\right), 
 \end{multline}
where the control $a $ is an arbitrary predictable process with values in a fixed nonempty closed convex set $\bb \subset \R^l$, either bounded or unbounded (see Section \ref{sec_DPP} for the precise setting). We  write $a \in {\mathcal P}_T$.
  We  require  the coefficients $b$, $\alpha$, $g$ and $f$ to be measurable. In addition, we assume Lipschitz--type and  linear growth conditions on the $x-$variable of these  coefficients, which ensure the existence of a pathwise unique strong solution $X = (X^{s,x,a}_t)_{t\ge s}$ of \eqref{uno}. 
We  assume that the coefficients in \eqref{anchelei0} depend in a continuous way on the control variable (cf. Hypothesis \ref{base22}). 
 We refer to  \cite{BLP, Bicht, IW, KU04, situ} for the theory of  SDEs with jumps; see also Section \ref{preliminare} for more details.

  The value  function $v$  associated with our stochastic control problem is 
\begin{equation}\label{value_def_Intro}
v(s,x) = \sup_{a 
\in {\mathcal P}_T} \mathbb{E}\bigg[\int_{s}^{T}h\left(r,X_r^{s,x,a}, a_r\right)dr + j\left(X_T^{s,x,a}\right) \bigg], 
\end{equation}
where  $h\colon[0,T] \times \mathbb{R}^d\times \bb \to \mathbb{R}$ is a given measurable and bounded map, which is continuous in the second  and third  variables, and $j\colon \mathbb{R}^d\to \mathbb{R}$ is a given bounded and continuous map.  With this $v$, we prove the following   DPP: 
\begin{equation}\label{DPP_intro}
v(s,x)= 
\sup_{a \in {\mathcal P}_T} \sup_{\theta\in \mathcal{T}_{s,T}}
\mathbb{E}\bigg[\int_{s}^{\theta}h\left(r,X_r^{s,x,a}, a_r\right)dr
+v\left(\theta, X_\theta^{s,x,a}\right)\bigg]
,\quad s\in [0,T),\,x\in\mathbb{R}^d.
\end{equation}
Here  $\theta$ is a  stopping time taking values in $(s,T)$, i.e., $\theta \in  \mathcal{T}_{s,T}$. We refer to  
Theorem \ref{DPP_t} for the complete assertion.  In particular,  the identity in \eqref{DPP_intro} continues to hold replacing $\sup_{\theta\in \mathcal{T}_{s,T}}$ with  $\inf_{\theta\in \mathcal{T}_{s,T}}$.  
\\  
{To the best of our knowledge, this result is  new even
 when there are no large jumps in \eqref{anchelei0} (i.e., $f\equiv0$) and when  the stopping time $\theta$ is fixed (hence there is no $\sup_{\theta\in \mathcal{T}_{s,T}}$ in \eqref{DPP_intro}). We refer to  \cite{touzi, ISHI, praga} for Bellman's principles involving special dynamics that do not cover \eqref{anchelei0}; see also Remark \ref{compare} for more information  and references.  
Notice that our formulation of DPP \eqref{DPP_intro} 
is  stronger than the usual one, 
which assumes  the stopping time $\theta$ to be fixed, see Remark \ref{pham2}. One can compare \eqref{DPP_intro} with the DPP proved in  the continuous  diffusion case in  
\cite[Theorems 3.1.10 and 
3.1.11]{kry}.
    
The proof of the DPP is divided into two parts, corresponding to two opposite inequalities. To prove {\it the first part} (see Subsection \ref{sub_firstDPP}), i.e.,  
  \begin{gather} \label{s22}
v(s,x)\le \sup_{a\in{\mathcal{P}_T}} \inf_{\theta\in \mathcal{T}_{s,T}}
\mathbb{E}\bigg[\int_{s}^{\theta}h\left(r,X_r^{s,x,a}, a_r\right)dr
+v\left(\theta, X_\theta^{s,x,a}\right)\bigg],
\end{gather}
we need to establish the existence of a regular stochastic flow for \eqref{anchelei0} when the coefficients are independent of the control $a$.
  We treat this issue in   Sections \ref{preliminare}, \ref{sec_small} and \ref{large_sec}.  
\\
More specifically, we consider  
  the following SDE of It\^o's type:
\begin{align} \label{uno}
dX_{t}=b\left(t, X_{t}\right) dt +   
\alpha\left(t, X_{t} \right) dW_t+
\int_{U_0}g\left(X_{t-},t,z\right)\widetilde{N}_p\left(dt,dz\right) 
+ 
\int_{U\setminus U_0 } f\left(X_{t-},t,z\right){N}_p\left(dt,dz\right),
\end{align}     
with $X_s=x\in\mathbb{R}^d,\,0\le s \le t \le T.$  We remark that the Lipschitz--assumption on the coefficient $f$ corresponding to the large--jumps part can be dispensed with (cf. Section IV.9 in \cite{IW} and see Hypothesis \ref{base}).     
We  deal with the  problem of finding a   version of the solution of \eqref{uno} which depends in a regular way on all the variables  $(s,t,x)$. 
We prove, in particular,   that there exists a  version of the solution $X$  which is \emph{regular} (or \emph{sharp}) in the following sense: 
there exists  an almost sure event $\Omega'$ 
such that, for every $\omega \in \Omega'$, the map  $(s,x,t)\mapsto  X^{s,x}_t(\omega)$ is càdlàg in $s$ (for $t$ and $x$ fixed), càdlàg in $t$ (for $s$ and $x$ fixed) and  continuous in $x$ (for $s$ and $t$ fixed). Moreover, we  prove  the flow property
\begin{equation}\label{fl11}
 X^{s,x}_t(\omega)=X^{u,X^{s,x}_u(\omega)}_t(\omega),\quad s<u<t\le T,
\end{equation}
  as well as  the stochastic continuity in $s$, locally uniformly in ${x}$ and uniformly in $t$ (see \eqref{prob1}). When $f\equiv0$, this result extends the regularity properties obtained in \cite{Ku19} by H. Kunita (see the comments after \eqref{kol} and Remark \ref{kuu} for more details).  We call this version a  \emph{regular stochastic flow} (or \emph{sharp stochastic flow}) generated by  \eqref{uno}. 
  
We refer to  Definition \ref{sharpala} and Theorems \ref{main1}-\ref{cadlag1} for  more general   assertions.  
 \\
 Going back to \eqref{anchelei0}, in Subsection \ref{sub5.2}, we show  that  there exists a regular stochastic flow for  SDEs controlled by suitable step  processes  $a\in\mathcal A \subset   {\mathcal P}_T$. This, in particular, enables us to handle  
   identities like  
\begin{gather*}  
X_r^{s,x,a} = X_r^{\theta,X^{s,x,a}_\theta,a},
\end{gather*}
which are meaningful even when  $\theta$ is a  stopping time in  $\mathcal{T}_{s,T}$ (see also Remark \ref{flow1no}). 
 Such identities  are useful to deduce \eqref{s22} when ${\mathcal P}_T$ is replaced by $\mathcal A$, which implies that \eqref{s22} holds, as well.

To prove {\it the second  part of the DPP}, i.e.,  
  \begin{equation}\label{opposite1}
v(s,x)\ge \sup_{a\in{\mathcal{P}}_T} \sup_{\theta\in \mathcal{T}_{s,T}}
\mathbb{E}\bigg[\int_{s}^{\theta}h\left(r,X_r^{s,x,a}, a_r\right)dr
+v\left(\theta, X_\theta^{s,x,a}\right)\bigg],
\end{equation}
we   introduce a new approach. It relies on a suitable subclass of
 predictable finitely generated step controls $\mathcal{B}^s\subset \mathcal{P}_T$,
which allows us to apply a basic measurability selection theorem from \cite{BP} (see Theorem \ref{BP_Sim}  and Remark \ref{seq}). 
 We believe that this method is of independent interest and could be adapted  to prove  DPPs arising in other stochastic control problems.\\ In this second  part, we also employ the lower semicontinuity of the value function $v$, which we prove in Lemma \ref{lsc}. Additional properties of $v$ 
 might be investigated and  will be the subject of a future research.

We now complete this introduction by presenting  further  discussions on the regular stochastic flow and the proof of the  DPP. In these paragraphs we  also describe the structure of the paper.    
    
\paragraph{Novelty and significance of  the regular stochastic flow}  
 We  recall that,  for 
 SDEs driven by a Brownian motion, namely Equation \eqref{uno} with $f\equiv0$ and $g\equiv0$, it is well known that there exists a regular   stochastic flow  
$X = (X^{s,x}_t)_{t\ge s}$ such that, for  $\mathbb P-$a.s. $\omega$,  the mapping  $(s,x,t)\mapsto  X^{s,x}_t(\omega)$ is  
continuous in $s$,  $t$ and $x$, and such that the flow property \eqref{fl11} holds. This is a consequence of the Kolmogorov--Chentsov test, which can be applied thanks to the well--known estimate   (see, for instance, \cite{Ku19})
\begin{equation}\label{kol}
{\mathbb E} \big|X^{s,x}_t - X^{s',x'}_{t'}\big|^p \le C \left (  |x-x'|^p + |s-s'|^{p/2} + | t-t'  
|^{p/2} \right), \quad p \ge 2.
\end{equation} 
This continuous stochastic flow  is  deeply  investigated in \cite{KU90}, where, in particular, it is employed to study first order stochastic PDEs when the coefficients are sufficiently smooth.  The previous  technique based on the  Komogorov--Chentsov test, however,  can only give  a continuous modification; 
it  cannot be applied  to SDEs with jumps to obtain a version of the solution that depends on $(s,t,x)$  in a regular way.

We are interested in the  case of {time--dependent coefficients}, where the main challenge in analyzing the flow regularity  with respect to $(s,t,x)$ is the dependence on   the initial time $s$.   Such an issue has  been    mentioned in \cite[Remark 1.2]{P20} and 
 \cite[Introduction]{AIHP18}. 
 Concerning \eqref{uno} with $f\equiv0$, the  problem of finding a regular (or sharp) version of the solution also appears in
    Kunita's book \cite{Ku19}. More precisely, in \cite{Ku19}, SDEs with small jumps are treated in Theorem 3.3.1 
(see also Theorems 3.4.1 and  3.4.2, where the differentiability of the flow is addressed); it is  shown that the solution has a modification continuous 
 in $x$  (for $t$ fixed) and càdlàg in $t$ (for $x$ fixed). 
 As we mentioned before,  for SDEs driven by Brownian motions, \cite[Theorem 3.4.3]{Ku19} gives a modification of the solution which is continuous in $(s,t,x)$, hence a stronger result.
   Thanks to Theorems  \ref{main1}-\ref{cadlag1}, we study the regularity in $s$ even in the jump case, obtaining all the conditions required in the definition of right--continuous stochastic flow  on   
\cite[Page 86]{Ku19}. We  refer to Remark \ref{kuu}
for a discussion about \cite{KU04}.    
    \\
 We also mention 
Theorem 5 in \cite{KS}, which provides the existence of a regular stochastic   flow for SDEs like \eqref{uno}  under the assumption that all the coefficients     are time--independent. Additionally, to apply  this theorem,  it is required that both $g$ and $f$ have a specific form (see Remark \ref{Scheutzow} for more details). This result cannot be used  to  study controlled SDEs. 

   We state the main results on regular stochastic flows in Section \ref{preliminare}.  Here, we deal with   a complete probability space  endowed with a general filtration $\mathbb{F}=\left(\mathcal{F}_t\right)_{0\le t\le T}$ satisfying the usual hypotheses. 
To prove Theorem \ref{main1},    we first consider the case  $f\equiv0$, corresponding to SDEs without large--jumps component. 
In this case, the result can be deduced from a stronger one (see Theorem \ref{cadlag1}), which shows that  the solution of \eqref{uno} can be obtained employing a càdlàg stochastic process $Z=\left(Z_s\right)$ with values in ${ C}(\R^d; {\mathcal D}_0)$ (see Section \ref{preliminare}-\ref{sec_small} for more details). In particular, for $\mathbb P-$a.s. $\omega$,
$$
[Z_s(\omega)](x) = X^{s, x}_{\, \cdot} (\omega) \in {\mathcal D}_0,\;\; x \in \R^d.
$$  
Here  $\mathcal{D}_0$ stands for the non--separable metric space of $\mathbb{R}^d-$valued, càdlàg functions endowed with the uniform norm  in $[0,T]$.  Indeed, we cannot use the Skorokhod topology $J_1$ on ${\mathcal D}_0$ to get our results (cf. Remark \ref{skor}).
In order to prove    Theorem \ref{cadlag1} we   employ  an extension of a càdlàg criterium from \cite{bez}, which   can be applied to the process $Z$ taking values in the non--separable metric space ${ C}(\R^d; {\mathcal D}_0)$. Such an extension is proved in the appendix (see Appendix \ref{ap_BZ}), which also contains additional measure theoretic results that we have not found in the literature  (see in particular Appendix \ref{ap_sigma}).  
\\
The proof of   Theorem \ref{cadlag1} requires also Proposition \ref{not_sep}, which is a variant of Theorem \ref{imk_t}, a generalized
Garsia--Rodemich--Rumsey type lemma due to \cite{imk}. This proposition and its Corollary \ref{cor5} enable us to estimate integrals like 
\begin{equation}\label{4eee}
\mathbb{E}\bigg[\sup_{\left|x\right|\le N}\sup_{s\le t\le T}\left|\int_{s}^{t}\!\int_{U_0}g\left(X^{s,x}_{r-},r,z\right)\widetilde{N}_p\left(dr,dz\right)\right|^\gamma \bigg], 
\end{equation} 
which are crucial for the proof of Theorem \ref{cadlag1} given in Subsections  \ref{st_cont_section}-\ref{cadlag_sec}.


In Section \ref{large_sec} we consider the full SDE \eqref{uno}, i.e., the SDE \eqref{uno} including also the large--jumps component determined by the coefficient $f$. 
This part is quite involved.  The issue is that we  cannot follow the standard interlacing procedure, see for instance  \cite[Section IV.9]{IW}  and   \cite[Section 3.2]{Bre}, to preserve our sharp stochastic flow. Specifically, the main difficulty is 
to maintain  the regularity of $X$ with respect to  $s$.
To overcome this challenge, we carefully modify the interlacing method using the stochastic flow  already obtained in Section \ref{sec_small}. This also gives formulae for  the solution of \eqref{uno} which could be  of independent interest (see, e.g., \eqref{meas}). 





 \paragraph{Novelty and significance of the DPP}  
 In Section \ref{sec_DPP},    we investigate the  controlled SDE \eqref{anchelei0} and state the associated  DPP, see Subsection \ref{sub_DPP}. As previously mentioned,  we suppose that the control $a$  takes values in a closed convex set  $\bb \subset \R^l$  and is  predictable with respect to the augmented filtration $\mathbb{F}^{W,N_p}$ generated by $W$ and $N_p$. We write $a \in {\mathcal P}_T$ and extensively use the orthogonal projection $\pj : \R^l  \to \bb$.  
  In Subsection \ref{sub_controlledg}, we introduce a notion of convergence in ${\mathcal P}_T$, which we use to show a continuity property  -- uniformly in probability -- of  the solution  $X^{s,x,a}$ of \eqref{anchelei0} with respect to the control $a\in\mathcal{P}_T$, see Theorem \ref{estrap1}. Thanks to this stability result, in Subsection \ref{sub5.2} we consider a suitable subclass of step controls $\mathcal A\subset\mathcal{P}_T$, whose corresponding solutions approximate an arbitrary solution process $X^{s,x,a},\,a\in\mathcal{P}_T$, in the sense of Corollary \ref{estrap}. This enables us to compute the value function $v$ in \eqref{value_def_Intro} as the $\sup_{a\in\mathcal{A}}$ (instead of $\sup_{a\in\mathcal{P}_T}$), see Corollary \ref{l20}. The advantage in introducing the subclass of controls $\mathcal{A}$ is that, for every $a\in\mathcal{A}$, 
  we can rely on results in Section \ref{preliminare} to construct a  stochastic flow $X^{s,x,a}_t$ associated with the controlled SDE \eqref{anchelei0} which is \emph{regular} in the sense of Definition \ref{sharpala}, see Lemma \ref{regular_control}. 
  
  The regularity properties of the flows $X^{s,x,a}_t,\,a\in\mathcal{A}$,   associated with \eqref{anchelei0} are essential for the arguments in Section \ref{sec_proofDPP}, which  is dedicated to the proof of the DPP. In particular,  they constitute the basis for the proof of the first part of the DPP \eqref{s22}, which is carried out in Subsection \ref{sub_firstDPP}.
   On this respect, we also refer to Remark \ref{flow1no}, where we discuss controlled stochastic flows that can be found in the literature.  \\
Subsection \ref{mtr} clarifies the importance of choosing the filtration $\mathbb{F}^{W,N_p}$, together with the requirement on $(U,\mathcal{U})$ to be a Polish space with its Borel $\sigma-$algebra.  In particular,  Lemma \ref{l_21} demonstrates that, in this framework, 
\begin{equation}\label{d55}
	H= L^2(\Omega,  \mathcal{F}^{W,N_p}_{s,t};\mathbb{R}^l) \;\; \text{is separable}.
\end{equation}
This fact is fundamental for 
 the new strategy that we develop in Subsections \ref{sub_mst} and \ref{subsect_2} to prove the second part of the DPP \eqref{opposite1}. In fact, the novel approach  we propose hinges on \eqref{d55} to define a special subclass of finitely generated step controls (denoted by  $ {\mathcal B}^s$) which allows us to apply a basic measurable selection theorem from \cite{BP} (see Theorem \ref{BP_Sim}). As is the case for $\mathcal{A}$, the subclass of controls $\mathcal{B}^s \subset \mathcal{P}_T$ does not change the value function $v$ in \eqref{value_def_Intro}
  (see Lemma \ref{ind_c_l}). Moreover, also considering Remark \ref{finite}, $\mathcal{B}^s$ enables us to obtain controls such as the one in \eqref{eccolo}, which is a crucial passage in the argument to deduce \eqref{opposite1}.

\begin{rem} \label{compare}   Here we comment on  some  works that  consider    DPPs for jump diffusions.  The paper \cite{praga}  analyzes a special case of \eqref{anchelei0}. It treats non--degenerate controlled SDEs with regular coefficients and small jumps  of stable type (large jumps are not included).  In particular, Lemma 3.5 in \cite{praga} states a DPP that  is  then used  to prove the  existence and  smoothness of the 
solution to the corresponding HJB equation. 
\\
  As for \cite{ISHI}, in Section 4 the author   proves  a DPP for an optimal portfolio/control problem in  two--dimensions in a model of interest in finance. 
  According to the introduction of this paper, in the  literature,   the Bellman's principle for jump processes is often just stated 
to hold, or expected to hold, in order to proceed in the study of 
solutions to  HJB equations (see also the  references in \cite{ISHI}). 
 \\
 The paper \cite{touzi} establishes  a general  DPP in a weak form. This result  is then  applied  
 in  Section 5.1 to  a class of controlled SDEs driven by independent 
   Brownian motions and  compound Poisson processes (it is also assumed a Lipschitz--type condition in the control variable). \\
 A class of  controlled SDEs without large--jumps part is considered in infinite dimensions in \cite{SZab}, where the authors  establish a DPP in a weak form.
 \\
 Finally, we  mention that the assumptions on the coefficients of \cite{praga} have been relaxed  in \cite{soner}, which proves that the value function is the unique viscosity solution of  the associated HJB equation and, additionally, it is a Lipschitz continuous function (see also \cite{lions}).  
\end{rem}

\section{Preliminaries and main results on the regular stochastic flow }
\label{preliminare}
In this work, $| \cdot |$ denotes the Euclidean norm in any $\mathbb{R}^m$, $m \ge 1$. Fix $T>0$ and let $\left(\Omega,\mathcal{F}, \mathbb{P}\right)$ be a complete probability space  endowed with a filtration $\mathbb{F}=\left(\mathcal{F}_t\right)_{0\le t\le T}$ satisfying the usual hypotheses. On this probability space, we take an $m-$dimensional $\mathbb{F}-$Brownian motion $W=\left(W_t\right)_{0\le t\le T}$. Moreover, given  a measurable space $(U,\mathcal{U})$, we consider a stationary Poisson point process $p$ on $U$ with intensity measure $dt \otimes \nu(dz)$, where  $\nu(dz)$ is a $\sigma-$finite measure on $U$ (see \cite[Section 9, Chapter \upperRomannumeral{1}]{IW}).  In particular, for every $\omega\in \Omega$, $p(\omega)\colon D_p(\omega)\to U$, where $D_p(\omega)$ is a countable subset of $(0,\infty)$. Let  $N_p$ be the counting measure associated with $p$, namely
\[
N_p\left(\left(0,t\right]\times V\right)(\omega)=\# \left\{s\in D_p(\omega)\cap (0,t] : [p(\omega)](s)\in V\right\},\quad t>0,\,V\in\mathcal{U},\,\omega\in\Omega;
\]
this is a Poisson random measure on $(0,\infty)\times U$. In the sequel, we write $p_s(\omega)=[p(\omega)](s)$ to have a compact notation. We denote by $\widetilde{N}_p(dt,dz)=N_p(dt,dz)-dt\otimes \nu(dz)$ the compensated Poisson random measure. We suppose that $p$ is $\mathbb{F}-$adapted, in the sense that $N_p((0,t]\times V)$ is $\mathcal{F}_t-$measurable for every $V\in \mathcal{U}$ and $t>0$. 

Fix a measurable set $U_0\in\mathcal{U}$ such that $\nu(U\setminus U_0)\in(0,\infty)$. In this section, we concentrate on  SDEs like \eqref{anchelei0} with coefficients independent of the controls that satisfy the following requirements. 
\begin{hyp} \label{base}  
	Let $b\left(t,x\right)=\left(b_j\left(t,x\right)\right)_{j=1,\dots,d}$ be the drift coefficient,  $\alpha\left(t,x\right)=\left(\alpha_{i,j}\left(t,x\right)\right)_{i=1,\dots,d;\, j=1,\dots,m}$ be the diffusion matrix and 
	$g\left(x,t,z\right)=\left(g_j\left(x,t,z\right)\right)_{j=1,\dots,d}$ be the  small--jumps coefficient. We require  $b\colon[0,T]\times \mathbb{R}^d\to \R^d$, $\alpha\colon [0,T]\times \R^d\to \R^{d\times m}$ and $g\colon\R^d\times [0,T]\times U\to \R^d$ to be jointly measurable in their domains.
		
		We assume that $b,\alpha$ and $g$ satisfy linear growth and Lipschitz--type conditions, see \cite{IW}. More precisely, for every $p\ge2$, there exists a constant $K_p$  such that 
		\begin{equation}\label{lg1}
			\left|b\left(t,x\right)\right|^p+
			\left|\alpha\left(t,x\right)\right|^p+
			\int_{U_0}\left|g\left(x,t,z\right)\right|^p\nu\left(dz\right)
			\le K_p\left(1+\left|x\right|^p\right),\quad x\in\mathbb{R}^d,\,t\in\left[0,T\right],	 
		\end{equation}
		and   
		\begin{multline}\label{lip1}
			\left|b\left(t,x\right)-b\left(t,y\right)\right|^p+
			\left|\alpha\left(t,x\right)-\alpha\left(t,y\right)\right|^p\\+
			\int_{U_0}\left|g\left(x,t,z\right)-g\left(y,t,z\right)\right|^p\nu\left(dz\right)\le K_p\left|x-y\right|^p,\quad x,y\in\mathbb{R}^d,\,t\in\left[0,T\right].
		\end{multline}
		%
	Here, $\left|\alpha\right|^2=\sum_{i,j}\left|\alpha_{i,j}\right|^2$. We also consider a large--jumps coefficient $f\colon \mathbb{R}^d\times [0,T] \times U\to \mathbb{R}^d$, supposing that $f$ is a jointly measurable function which is continuous in the first argument. 
\end{hyp} 
In this paper, we study the SDE
\begin{multline}\label{SDEgeneral}
	X_{t}^{} = x+ \int_{s}^{t}b\left(r, X_{r}^{} \right) dr + 
	\int_{s}^{t}\alpha\left(r, X_{r}^{} \right) dW_r\\+
	\int_{s}^{t}\!\int_{U_0}g\left(X_{r-}^{},r,z\right)\widetilde{N}_p\left(dr,dz\right) 
	+ 
	\int_{s}^{t}\!\int_{U\setminus U_0 } f\left(X_{r-}^{},r,z\right){N}_p\left(dr,dz\right),\quad t\in[s,T],
\end{multline} 
where $s\in [0,T)$ and $x\in\mathbb{R}^d$. In particular, the small--jumps case $f\equiv0$ is investigated in Section \ref{sec_small}, while the large--jumps case $f\neq 0$ is analyzed in Section \ref{large_sec}. 
\\
A solution to \eqref{SDEgeneral} is a c\`adl\`ag, $\mathbb{R}^d-$valued, $\mathbb{F}-$adapted process $X=\left(X_t\right)_{s\le t\le T}$ satisfying \eqref{SDEgeneral} up to indistinguishability. We extend the trajectories of $X$ in the whole interval $[0,T]$ by setting $X_t=X_s,\,t\in[0,s]$. 
Under Hypothesis \ref{base}, it is known that \eqref{SDEgeneral} admits a pathwise unique solution for every initial condition $(s,x)$ (see Sections \ref{sec_small}-\ref{large_sec} for the details). Our goal is to prove the existence of a \emph{regular} (or \emph{sharp}) version of the solution $X$ which is  simultaneously càdlàg in the time variables $s,t,$ and continuous in the space variable $x$;  furthermore, $X$ is {stochastically continuous in $s$.}   More precisely, we search for a \emph{regular} (or \emph{sharp}) \emph{stochastic flow} generated by \eqref{SDEgeneral} according to the following definition, where we denote by $\mathcal{D}_0$ the complete metric space of $\mathbb{R}^d-$valued, càdlàg functions on $[0,T]$ endowed with the uniform norm.

\begin{defini}\label{sharpala}
	Let  $X\colon \Omega\times [0,T]\times \mathbb{R}^d\times [0,T]\to \mathbb{R}^d$ be an $\mathcal{F}_{{}}\otimes \mathcal{B}([0,T]\times \mathbb{R}^d\times [0,T])-$measurable function and denote by $X^{s,x}_t(\omega)=X(\omega,s,x,t)$. We say that  $X$ is the \emph{regular} (or \emph{sharp}) \emph{stochastic flow} generated by \eqref{SDEgeneral} if there exists an a.s. event $\Omega'$ --independent of $s,t,x$-- such that the four following requirements are fulfilled for every $\omega\in \Omega'$, $s\in [0,T)$ and $x\in\mathbb{R}^d$.
	\begin{enumerate}[ref=\arabic*.]
		\item \label{comunala}The process $(X^{s,x}_t)_{t\in[s,T]}$ satisfies  \eqref{SDEgeneral} in $\Omega'$;
		\item \label{regularity}
		\begin{enumerate}[label=(\roman*)]
			\item\label{pure} { The map $X^{{s},{x}}_\cdot\left(\omega\right)\colon[0,T]\to \mathbb{R}^d$ is càdlàg;}
			\item\label{conx}{ The map $X^{{s},\cdot}\left(\omega\right)\colon \mathbb{R}^d\to \mathcal{D}_0$ is continuous;}
			\item\label{cons} { The map $X^{\cdot, {x}}\left(\omega\right)\colon [0,T] \to \mathcal{D}_0$ is càdlàg, locally uniformly in ${x}$;}
		\end{enumerate}
		\item \label{flow_1} The flow property holds:
		$\qquad X^{s,x}_t(\omega)=X^{u,X^{s,x}_u(\omega)}_t(\omega),\quad s<u<t\le T.$
		{\item \label{constoch} \color{black}The function $X$ is stochastically continuous in the following sense: for every $\epsilon>0$ and $M>0$, 
			\begin{equation} \label{prob1}
			\lim_{r\to s}\mathbb{P}\Big(\sup_{\left|x\right|\le M}\sup_{0\le t\le  T}\left|X^{r,x}_t-X^{s,x}_t\right|>\epsilon\Big)=0,\quad s\in [0,T].
			\end{equation} } 
	\end{enumerate}		    
\end{defini}
Notice that, by the pathwise uniqueness of \eqref{SDEgeneral} and Point \ref{regularity} in Definition \ref{sharpala}, a regular stochastic flow generated by \eqref{SDEgeneral} is unique up to an a.s. event. 


The next theorem shows the existence of the regular stochastic flow associated with  \eqref{SDEgeneral}.
	\begin{theorem} \label{main1} Under Hypothesis \ref{base}, the regular stochastic flow generated  by \eqref{uno} exists.
	\end{theorem}
		When $f\equiv0$, i.e., in the small--jumps case,  we deduce the previous result from a stronger one, which is presented in Theorem \ref{cadlag1} after introducing some notations. Let $\mathcal{C}_0=(C(\mathbb{R}^d;\mathcal{D}_0), d_0^{lu})$ be the metric space of continuous functions defined on $\mathbb{R}^d$ with values in $\mathcal{D}_0$ with  the usual distance $d_0^{lu}$ (defined below in \eqref{d0_un}).
   We endow $\mathcal{C}_0$ with the $\sigma-$algebra $\mathcal{C}$ generated by the projections $\pi_x\colon C(\mathbb{R}^d;\mathcal{D}_0) \to (\mathcal{D}_0,\mathcal{D}),\,x\in\mathbb{R}^d$, defined by $\pi_x(f)=f(x),\,f\in \mathcal{C}_0$. Here $\mathcal{D}$ is the $\sigma-$algebra on $\mathcal{D}_0$ generated by the Skorokhod topology (see the discussion around \eqref{pil}).
	\begin{theorem}\label{cadlag1}
		When $f\equiv0$, under Hypothesis \ref{base}, the  regular stochastic flow generated  by \eqref{uno} exists and is a stochastically continuous, càdlàg  $(\mathcal{C}_0,\mathcal{C})-$valued process.
	\end{theorem}

\begin{rem} \label{kuu}  
We make some comments on 
the stochastic flow for SDEs  studied in \cite{KU04} and \cite{Ku19}. We concentrate on  the SDEs in \cite{Ku19}
 which  have only
 ``small jumps'' 
and  are similar to \eqref{uno} with $f\equiv0$. On the other hand, \cite{KU04}
 considers 
more general   SDEs  with possibly random coefficients.   

\vv   \begin{enumerate}[label*=(\roman*)]
	\item \cite[Theorem 3.2]{KU04} and \cite[Theorem 3.4.1]{Ku19}  show the existence of a  modification of the solution which is, $\mathbb{P-}$a.s., continuous in the initial state  $x\in\mathbb{R}^d$ and càdlàg in $t$: the regularity with respect to the initial time $s$ is not considered. A version of the solution that is càdlàg in $s$ is given by
	\cite[Proposition 3.8.2]{Ku19}, which, however, requires to fix a time $t$. Therefore, these results do not prove the  simultaneous c\`adl\`ag property in $s$ and $t$.
	\item At the end of \cite[Page 353]{KU04},  it is stated that if  $(\xi_t(x))_{t\ge t_0}$ denotes the solution of the  SDE starting from $x$ at $t_0$, then 
	{the inverse flow   $\{ \xi_t^{-1}\colon \R^d\to \R^d ,\, t \ge t_0 \}$ is also a  c\`adl\`ag  process taking values in $C(\mathbb{R}^d , \mathbb{R}^d )$, just as $ \{ \xi_t \, : t \ge t_0 \}$.} If this holds true, then  the simultaneous c\`adl\`ag  property in $s$ and $t$    
	of the solution $\xi_{s,t}(x)$ of the SDE starting from $x$ at $s$ would be  a simple consequence of the relations
	\[
	\xi_{s,t}(x) =x,   \text{ if } t \le  s;\qquad    \xi_{s,t} (x)=  \xi_t \circ  \xi_s^{-1} (x),  \text{ if } t \ge  s.
	\]
	In this way, the claims at the beginning of \cite[Page 354]{KU04} would be  completely justified. However,  
	we have  not found a proof of the statement about the c\`adl\`ag property of the inverse flow for  SDEs with jumps in either \cite{KU04} or \cite{Ku19}.   Since this is a significant gap, we have followed a different path -- not relying on the inverse flow -- to demonstrate the simultaneous c\`adl\`ag  property in $s$ and $t$.
\end{enumerate}
Theorem \ref{cadlag1} solves this inconsistency by considering the further regularity in $s$. Thus, in the notation of \cite{KU04},  we are able to prove the simultaneous c\`adl\`ag property in $s$ and $t$ of $\xi_{s,t}$, a property that appears to be claimed without a proof on \cite[Page 354]{KU04}. Theorem \ref{main1} extends this regularity result to SDEs with a large--jumps component, i.e., $f\neq0$. 

We remark that, in this work, we do not   discuss the homeomorphism property of \eqref{uno}, which is the subject of \cite[Section 3.4]{KU04} (see also \cite[Section 3.7]{Ku19}) and  requires additional assumptions on the coefficients (Lipschitz--type conditions are not enough).
\end{rem}

\begin{rem} \label{Scheutzow} The paper  \cite{KS} considers the SDE   
\begin{equation} \label{sch}
dX_t = l(X_{t-}) \,dZ_t, \;\;\; t \ge s, \qquad X_s =x \in {\mathbb R}^d,
\end{equation}
from the point of view of  
random dynamical systems. Here  
$(Z_t)_{t \ge 0}$ is an ${\mathbb R}^k-$valued semimartingale 
with values in ${\mathbb R}^m$ 
and $l : \R^d \to \R^{d \times k}$ is Lipschitz continuous. 
 Despite some differences with the assertions in Theorem \ref{main1},   \cite[Theorem 5]{KS} gives  a version of the solution  of  the SDE \eqref{sch} which is regular in the variables $(s,t,x)$ and satisfies  the flow property
 (the stochastic continuity is not investigated in \cite{KS}). In particular,
 \cite[Theorem 5]{KS} can be applied to the SDE \eqref{SDEgeneral} when the coefficients  are time--independent and $g$ and $f$ have a special form; it cannot be used to study controlled SDEs. More precisely, $\alpha$ and  $b$ must be time--independent and 
$ g(x,r,z)= g_1(x) g_2(r,z),$ $ f(x,r,z)= f_1(x) f_2(r,z),$
where $f_2, g_2  $ are measurable in their domains with values in $\R^k$ and $g_2$ verifies  
$\int_0^T dr \int_{U_0}\left|g_2\left(r,z\right)\right|^2\nu\left(dz\right) < \infty$. Moreover, one has to require that $g_1,  f_1 : {\mathbb{R}^d} \to {\mathbb{R}^{d\times k}} $ are Lipschitz continuous. 
In this case,
the semimartingale $Z$ in \eqref{sch} takes the form 
$
Z_t = \Big(t, W_t, \int_0^t \int_{U_0} g_2\left(r,z\right) \widetilde{N}_p\left(dr,dz\right) , \int_0^t \int_{U \setminus U_0} f_2(r,z) {N}_p \left(dr,dz\right) \Big)\in \mathbb{R}^{1+ m+2k}. 
$
The proof of  \cite[Theorem 5]{KS} is different from the one of Theorem \ref{main1}, which relies on Theorem \ref{cadlag1}. On the other hand, 
in \cite{KS} there are no results related to Theorem \ref{cadlag1} (the space $\mathcal{C}_0 =(C(\mathbb{R}^d;\mathcal{D}_0), d_0^{lu})$ is introduced in this paper). 
\end{rem}

\section{Proof of existence of the regular stochastic flow  for SDEs with small jumps\nor }
		\label{sec_small}
In this section, we are interested in the study of \eqref{SDEgeneral} with $f\equiv 0$, i.e., the SDE with \emph{small jumps}. In particular, we consider the following SDE:
\begin{equation}\label{SDE}  
Y_t=\eta+\int_{s}^{t}b\left(r,Y_{r}\right)dr+\int_{s}^{t}\alpha\left(r,Y_{r}\right)dW_r+
\int_{s}^{t}\!\int_{U_0}g\left(Y_{r-},r,z\right)\widetilde{N}_p\left(dr,dz\right),\quad t\in\left[s,T\right],
\end{equation}
where $s\in\left[0,T\right)$ and 
$$
\eta\in L^0\left(\mathcal{F}_s\right),
$$  
i.e., $\eta$ is an $\mathcal{F}_s-$measurable random variable with values in $\mathbb{R}^d$.  A solution of this equation is a c\`adl\`ag, $\mathbb{R}^d-$valued, $\mathbb{F}-$adapted process $Y=\left(Y_t\right)_{s\le t\le T}$ satisfying \eqref{SDE} up to indistinguishability. 

Conditions \eqref{lg1}-\eqref{lip1} guarantee the existence of a solution $Y$ to \eqref{SDE} for every  $\eta\in L^p(\Omega)\cap L^0(\mathcal{F}_s)$, with $p\ge2$, see \cite[Theorem 3.1]{KU04} or \cite[Theorem 117]{situ}. Such a solution is pathwise unique and satisfies
\[
\mathbb{E}\big[\sup_{s\le t\le T}\left|Y_t\right|^p\big]<\infty.
\] 
We denote by $Y^{s,\eta}$ the solution of \eqref{SDE} starting from $\eta$ at time $s$. We also set $Y^{s,\eta}_r = Y^{s,\eta}_s$ if $0 \le r < s$. 
\subsection{Flow property and continuity in $x$}\label{pre_flow}
The pathwise uniqueness of \eqref{SDE} immediately implies the \emph{cocycle property}: for every $x\in\mathbb{R}^d$ and $0\le s<u\le T$, there exists an a.s. event $\Omega_{s,u,x}$ such that
\begin{equation}\label{cocycle}
Y^{u,Y^{s,x}_u}_t\left(\omega\right)=Y^{s,x}_t\left(\omega\right),\quad t\in\left[u,T\right],\,\omega\in\Omega_{s,u,x}.
\end{equation}
The notation $\Omega_{s,u,x}$ indicates an (a.s.) event which may depend on $s,u$ and $x$ (it is independent of $t$). This notation will be adopted for the rest of the paper. 

The next result  is an extension of \cite[Equation $\left(3.7\right)$]{KU04} to random initial conditions, see also \cite[Lemma 3.3.3]{Ku19}. The proof contains useful estimates (in particular, see \eqref{bdg_alpha}-\eqref{bdg_g}) which will be used several times hereinafter.
\begin{lemma}\label{l1}
Fix $p\ge2$. Then, for every $\mathcal{F}_s-$measurable random variables $\xi,\eta\in L^p\left(\Omega\right)$, one has 
\begin{equation}\label{3.6} 
\mathbb{E}\Big[\sup_{s\le t\le T}\left|Y_t^{s,\eta}-Y_t^{s,\xi}\right|^p\Big]\le 4^{p-1}e^{c\left(T-s\right)} \mathbb{E}\left[\left|\eta-\xi\right|^p\right],\quad s\in\left[0,T\right),
\end{equation} 
where $c>0$ is a constant depending only on $p,d,m, T,K_2, K_p$.
\end{lemma}
\begin{proof}
Fix $0\le s<T$ and two $\mathcal{F}_s-$measurable random variables $\xi,\eta\in L^p(\Omega)$. By \eqref{SDE}, in an a.s. event $\Omega_{s,\xi,\eta}$ we have, using {H\"older's inequality} and the Lipschitz condition \eqref{lip1} on $b$,
\begin{align*}
&\left|Y_t^{s,\eta}-Y_t^{s,\xi}\right|^p
\\&\quad \le 4^{p-1}\bigg(\left|\eta-\xi\right|^p+K_pT^{p-1}\int_{s}^{t}\sup_{s\le u\le r}\left|Y^{s,\eta}_{u}-Y^{s,\xi}_{u}\right|^pdr
+\sup_{s\le u\le t}\left|\int_{s}^{u}\left(\alpha\left(r,Y_{r}^{s,\eta}\right)-\alpha\left(r,Y_{r}^{s,\xi}\right)\right)dW_r\right|^p\\
&\qquad
+\sup_{s\le u\le t}\left|\int_{s}^{u}\!\int_{U_0}\left(g\left(Y^{s,\eta}_{r-},r,z\right)-g\left(Y^{s,\xi}_{r-},r,z\right)\right)\widetilde{N}_p\left(dr,dz\right)\right|^p
\bigg),\quad t\in\left[s,T\right].
\end{align*}
Taking the supremum and the expectation  we obtain
\begin{align}\label{l1_1}
\notag&\mathbb{E}\left[\sup_{s\le u\le t}\left|Y_u^{s,\eta}-Y_u^{s,\xi}\right|^p\right]\le 4^{p-1}\bigg(\mathbb{E}\left[\left|\eta-\xi\right|^p\right]\\&\qquad\qquad
+K_pT^{p-1}\int_{s}^t\mathbb{E}\left[\sup_{s\le u\le r}\left|Y^{s,\eta}_{u}-Y^{s,\xi}_{u}\right|^p\right]dr+\mathbb{E}\left[\sup_{s\le u\le t}\left|\int_{s}^{u}\left(\alpha\left(r,Y_{r}^{s,\eta}\right)-\alpha\left(r,Y_{r}^{s,\xi}\right)\right)dW_r\right|^p\right]\notag\\&\qquad \qquad
+\mathbb{E}\left[\sup_{s\le u\le t}\left|\int_{s}^{u}\!\int_{U_0}\left(g\left(Y^{s,\eta}_{r-},r,z\right)-g\left(Y^{s,\xi}_{r-},r,z\right)\right)\widetilde{N}_p\left(dr,dz\right)\right|^p\right]
\bigg),\quad t\in\left[s,T\right].
\end{align}
By the {Burkholder--Davis--Gundy inequality} 
and the Lipschitz condition \eqref{lip1} on $\alpha$ we have 
\begin{multline}\label{bdg_alpha}
\mathbb{E}\left[\sup_{s\le u\le t}\left|\int_{s}^{u}\left(\alpha\left(r,Y_{r}^{s,\eta}\right)-\alpha\left(r,Y_{r}^{s,\xi}\right)\right)dW_r\right|^p\right]\le c_p\,\left(dm\right)^p\,T^{\frac{p}{2}-1}\mathbb{E}\left[\int_{s}^{t}\left|\alpha\left(r,Y^{s,\eta}_{r}\right)-\alpha\left(r,Y^{s,\xi}_{r}\right)\right|^pdr\right]\\
\le c_p\,(dm)^{p}\,T^{\frac{p}{2}-1}K_p\int_{s}^{t}\mathbb{E}\left[\sup_{s\le u\le r}\left|Y^{s,\eta}_{u}-Y^{s,\xi}_{u}\right|^p\right]dr,\quad t\in\left[s,T\right].
\end{multline} 
where $c_p>0$ is a constant only depending on $p$. As for the integral with respect to $\widetilde{N}_p$, \cite[Theorem $2.11$]{KU04} yields, for every $t\in\left[s,T\right],$
\begin{align}\label{bdg_g}
\notag&\mathbb{E}\left[\sup_{s\le u\le t}\left|\int_{s}^{u}\!\int_{U_0}\left(g\left(Y^{s,\eta}_{r-},r,z\right)-g\left(Y^{s,\xi}_{r-},r,z\right)\right)\widetilde{N}_p\left(dr,dz\right)\right|^p\right]\\&\qquad 
\notag\le c_{1,p}\,d^p\left(\mathbb{E}\left[\left(\int_{s}^{t}dr\int_{U_0}\nu\left(dz\right)\left|g\left(Y^{s,\eta}_{r-},r,z\right)-g\left(Y^{s,\xi}_{r-},r,z\right)\right|^2\right)^{\frac{p}{2}}\right]\right.\\&
\qquad \qquad \notag
+ 
\left.\,\mathbb{E}\left[\int_{s}^{t}dr\int_{U_0}\nu\left(dz\right)\left|g\left(Y^{s,\eta}_{r-},r,z\right)-g\left(Y^{s,\xi}_{r-},r,z\right)\right|^p\right]\right)\\&\qquad 
\le 
c_{1,p}\,d^p \left[ T^{\frac{p}{2}-1}
K_2^\frac{p}{2} + 
K_p
\right]  
\int_{s}^{t}\mathbb{E}\left[\sup_{s\le u\le r}\left|Y^{s,\eta}_{u}-Y^{s,\xi}_{u}\right|^p\right]dr,
\end{align}
where $c_{1,p}>0$. Going back to \eqref{l1_1}, we combine \eqref{bdg_alpha} and \eqref{bdg_g} to get the existence of a constant $c=c\left(p,d,m,T, K_2,K_p\right)>0$ such that 
\begin{equation*}
\mathbb{E}\Big[\sup_{s\le u\le t}\left|Y^{s,\eta}_{u}-Y^{s,\xi}_{u}\right|^p\Big]\le 4^{p-1}\mathbb{E}\Big[\left|\eta-\xi\right|^p\Big] +c\int_{s}^{t}\mathbb{E}\Big[\sup_{s\le u\le r}\left|Y^{s,\eta}_{u}-Y^{s,\xi}_{u}\right|^p\Big]dr,   \quad t\in\left[s,T\right].
\end{equation*}
At this point {Gronwall's lemma} provides us with the assertion.
\end{proof}
Denote by $\mathcal{D}_0=\left(\mathcal{D}\left(\left[0,T\right];\mathbb{R}^d\right),\norm{\cdot}_0\right),$ the metric space of $\mathbb{R}^d-$valued, càdlàg functions with the uniform norm $\norm{\cdot}_0$ in $[0,T]$:  $\mathcal{D}_0$ is  complete but not separable. Inspired by \cite[Chapter \upperRomannumeral{5}]{pollard}, we endow $\mathcal{D}_0$ with the $\sigma-$algebra $\mathcal D$ generated by the projections 
\begin{equation}\label{pil}
	\pi_t\colon \mathcal{D}_0 \to \mathbb{R}^d,\,t\in [0,T], \quad \text{defined by } \pi_t(f)=f(t),\,f\in\mathcal{D}_0.
\end{equation}
 It is well known that $\mathcal{D}$ coincides with the Borel $\sigma-$algebra generated by the Skorokhod topology $J_1$, see  \cite[Theorem $12.5$]{Bill}, \cite[Theorem 1.14, Chapter \upperRomannumeral{6}]{js}) and \cite[Corollary 2.4]{J}. On the contrary, $\mathcal{D}$ is strictly smaller than the Borel $\sigma-$algebra of the uniform distance (cf. \cite[Eq. (15.2)]{Bill}). Notice that the difference of two càdlàg functions, considered as a mapping from $(\mathcal{D}_0\times \mathcal{D}_0, \mathcal{D}\otimes \mathcal{D})$ to $(\mathcal{D}_0,\mathcal{D})$, is measurable. Indeed, this is an immediate consequence of the measurability of the following map:
\[
	\left(\mathcal{D}_0\times \mathcal{D}_0, \mathcal{D}\otimes \mathcal{D}\right)\to \mathbb{R}^{2d},\quad \left(x,y\right)\mapsto \left(x\left(t\right),y\left(t\right)\right),\quad t\in \left[0,T\right].
\]
Moreover, observe that also $\norm{\cdot}_0\colon (\mathcal{D}_0,\mathcal{D})\to\mathbb{R}$ is measurable, because the càdlàg property allows to compute the supremum on a countable  dense set of $[0,T]$.\\
Let us fix $s\in\left[0,T\right)$ and consider the random field 
$
Y^{s,\cdot}=\left(Y^{s,x}\right)_{x\in\mathbb{R}^d}.
$
For every $x\in\mathbb{R}^d$, the map $Y^{s,x}\colon \Omega\to (\mathcal{D}_0, \mathcal{D})$ is a random variable, i.e., it is $\mathcal{D}-$measurable. Hence, by the previous discussion the function  $\omega\mapsto \sup_{0\le t \le T}|Y_t^{s,x}(\omega)-Y_t^{s,y}(\omega)|$ is measurable for every $x,\,y\in\mathbb{R}^d$. Thanks to \eqref{3.6}, choosing $p>d$ we can  apply the  \emph{Kolmogorov--Chentsov continuity criterion} as in \cite[Lemma A.$2.37$]{Bicht} to find  a continuous modification $\widetilde{Y}^{s,\cdot}$ of $Y^{s,\cdot}$. Hence there exist a.s. events $\Omega_s$ and $\Omega_{s,x}$ such that \begin{equation}\label{2'}
	\widetilde{Y}^{s,x}=Y^{s,x} \text{ in $\Omega_{s,x}$,\quad and\quad  $\widetilde{Y}^{s,\cdot}\left(\omega\right)\colon\mathbb{R}^d\to\mathcal{D}_0$ is continuous for any $\omega \in \Omega_s$.}
\end{equation}  
By setting $\widetilde{Y}^{s,x}_t\left(\omega\right)=x$ for every $x\in\mathbb{R}^d,\,t\in\left[0,T\right]$ and $\omega\in\Omega_s^c$,  we get the continuity of $\widetilde{Y}^{s,\cdot}(\omega)$ for all $\omega\in\Omega$. 
From now on, we will always work with this continuous version, which we keep denoting by $Y^{s,\cdot}.$ 

The following result shows the dependence of $\mathbb{P}-$a.s. path of the solution on the initial condition.
\begin{proposition}\label{p2}
For every $s\in\left[0,T\right)$ and $\eta\in L^p\left(\Omega\right)\cap L^0(\mathcal{F}_s),\,p\ge2,$  there exists an a.s. event $\Omega_{s,\eta}$ such that
\begin{equation}\label{dependence}
Y^{s,\eta}_t\left(\omega\right)=Y_t^{s,\eta\left(\omega\right)}\left(\omega\right),\quad t\in\left[s,T\right],\,\omega\in\Omega_{s,\eta}.
\end{equation}
\end{proposition}
\begin{proof}
Fix $p\ge2$ and $s\in [0,T)$. First, we notice that \eqref{dependence} is an immediate consequence of the pathwise uniqueness of the solutions to \eqref{SDE} when $\eta\in L^p\left(\Omega\right)\cap L^0(\mathcal{F}_s)$ is simple, namely 
$
\eta=\sum_{k=1}^{n}a_k1_{A_k},
$
where $n\in\mathbb{N}, \,a_k\in\mathbb{R}^d$ and $\left(A_k\right)_{k}$ is an $\mathcal{F}_s-$measurable partition of $\Omega$, $k=1,\dots,n$. 

Secondly, we consider a generic $\eta\in L^p(\Omega)\cap L^0(\mathcal{F}_s)$ and take a sequence $\left(\eta_n\right)_n$ of simple, $\mathcal{F}_s-$measurable random variables converging to it both in the $L^p-$sense and almost surely. By the previous step, we can find an a.s. event $\Omega'_{s,\eta}$ (independent of $n$) such that
\[
Y^{s,\eta_n}\left(\omega\right)=Y^{s,\eta_n\left(\omega\right)}\left(\omega\right),\quad n\in\mathbb{N},\, \omega\in\Omega'_{s,\eta}.
\]
Without loss of generality, suppose that $\eta_n\to\eta$ pointwise on $\Omega'_{s,\eta}$. The continuity of the random field $\left(Y^{s,x}\right)_{x\in\mathbb{R}^d}$ in  $\Omega$ yields
\[
\mathcal{D}_0-\!\!\lim_{n\to\infty}Y^{s,\eta_n\left(\omega\right)}\left(\omega\right)=Y^{s,\eta\left(\omega\right)}\left(\omega\right),\quad \omega\in\Omega'_{s,\eta}.
\]
On the other hand, an application of Lemma \ref{l1} shows that, possibly passing to a subsequence,
\[
\mathcal{D}_0-\!\!\lim_{k\to\infty}Y^{s,\eta_{n_k}}\left(\omega\right)=Y^{s,\eta}\left(\omega\right),\quad \omega\in\Omega_{s,\eta},
\] 
with an a.s. event $\Omega_{s,\eta}\subset\Omega'_{s,\eta}$. 
By the three previous assertions we infer \eqref{dependence} and the proof is complete. 
\end{proof} 
Combining the cocycle property in \eqref{cocycle} with Proposition \ref{p2} we get the \emph{flow property} expressed in the next corollary. This result improves \eqref{cocycle}, because in \eqref{cocycle} the a.s. event $\Omega_{s,u,x}$ possibly depends on $x \in \mathbb{R}^d$.

\begin{corollary}\label{cor_flow}
For every $0\le s<u\le T$, there exists an a.s. event $\Omega_{s,u}$ such that
\begin{equation}\label{flow}
Y_t^{u,Y^{s,x}_u\left(\omega\right)}\left(\omega\right)=Y^{s,x}_t\left(\omega\right),\quad t\in\left[u,T\right],\, x\in\mathbb{R}^d,\,\omega\in\Omega_{s,u}.
\end{equation}
\end{corollary}
\begin{proof}
Fix $0\le s< u\le T$. Equations \eqref{cocycle}-\eqref{dependence} imply, for every $x\in\mathbb{R}^d,$ the existence of an a.s. event $\Omega_{s,u,x}$ where \eqref{flow} holds. Therefore, it is  sufficient to remove the dependence of this event from $x$ to prove the assertion. Let  $\Omega_{s,u}=\cap_{x\in\mathbb{Q}^d}\Omega_{s,u,x}$; then $\mathbb{P}\left(\Omega_{s,u}\right)=1$ and 
\begin{equation}\label{flow_rat}
Y_t^{u,Y^{s,x}_u\left(\omega\right)}\left(\omega\right)=Y^{s,x}_t\left(\omega\right),\quad t\in\left[u,T\right],\, x\in\mathbb{Q}^d,\,\omega\in\Omega_{s,u}.
\end{equation}
For a point $x\in\mathbb{R}^d\setminus\mathbb{Q}^d$, take a sequence $\left(x_n\right)_n\subset\mathbb{Q}^d$ such that $x_n\to x$. Given $\omega\in\Omega_{s,u}$, by the continuity of the random field $Y^{s,\cdot}$ (see \eqref{2'} and the subsequent comment), one can pass to the limit in \eqref{flow_rat} to show that \eqref{flow_rat} holds in $x$, as well. This gives \eqref{flow}, completing the proof.
\end{proof}

Observe that each process $Y^{s,x},\,x\in\mathbb{R}^d,$ satisfies \eqref{SDE} in an a.s. event $\Omega_{s,x}$:  we now want to find a common a.s. event $\Omega_s$ --independent of $x$-- where \eqref{SDE} holds (with $\eta=x$). To do this, we consider suitable modifications of the stochastic integrals.  We start off by taking $\Omega'_s= \cap_{x\in\mathbb{Q}^d} \Omega_{s,x}$, so that, for every $\omega\in\Omega'_{s}$,
\begin{equation}\label{ar_1}
Y^{s,x}_t-x-\int_{s}^{t}b\left(r,Y^{s,x}_{r}\right)dr=\int_{s}^{t}\alpha\left(r,Y^{s,x}_{r}\right)dW_r+
\int_{s}^{t}\!\int_{U_0}g\left(Y^{s,x}_{r-},r,z\right)\widetilde{N}_p\left(dr,dz\right),\quad t\in\left[s,T\right],\,x\in\mathbb{Q}^d.
\end{equation}
 We can now construct two continuous random fields corresponding to the addends on the right--hand side of the previous equation. Specifically, for the term $(Z^{s,x}_{1})_{x\in\mathbb{R}^d}=\left(\int_{s}^{\cdot}\alpha\left(r,Y^{s,x}_{r}\right)dW_r\right)_{x\in\mathbb{R}^d}$, we combine \eqref{3.6} with the estimate in \eqref{bdg_alpha} to write, for any $p>d$,
\[
\mathbb{E}\Big[\sup_{s\le t\le T}\left|\int_{s}^{t}\left(\alpha\left(r,Y_{r}^{s,y}\right)-\alpha\left(r,Y_{r}^{s,x}\right)\right)dW_r\right|^p\Big]
\le C\left|x-y\right|^p,\quad x,y\in\mathbb{R}^d,
\]  
where $C=C\left(T,p,d,m,K_2,K_p\right)>0$. Hence the  Kolmogorov--Chentsov criterion ensures the existence of a  version of this random field which is continuous in an a.s. event $\Omega''_s$. We set this modification equal to $0$ outside $\Omega_s''$, so that it is continuous in the whole space $\Omega$,  and keep denoting it by $Z_{1}^{s,x}=\int_{s}^{\cdot}\alpha\left(r,Y^{s,x}_{r}\right))dW_r$. Moreover, we can think of $(Z^{s,x}_{1})_x$ as a continuous, $(\mathcal{D}_0,\mathcal{D})-$valued random field defining $Z^{s,x}_{1,t}(\omega)= Z^{s,x}_{1,s}(\omega)$, for every $t\in\left[0,s\right),\,x\in\mathbb{R}^d$ and $\omega\in\Omega.$  \\
As for   $Z^{s,x}_{2}=\int_{s}^{\cdot}\!\int_{U_0}g\left(Y^{s,x}_{r-},r,z\right)\widetilde{N}_p\left(dr,dz\right),\,{x\in\mathbb{R}^d},$ the argument to obtain a continuous, $(\mathcal{D}_0,\mathcal{D})-$valued  modification is the same once we consider the estimate in \eqref{bdg_g}. This construction ensures that, in an a.s. event $\Omega_s\subset \Omega_s'$,  \eqref{ar_1} holds with the right--hand side being the sum of continuous random fields.\\ 
Finally, it is easy to see that $(Z^{s,x}_{3})_{x\in\mathbb{R}^d}=\left(\int_{s}^{\cdot}b\left(r,Y^{s,x}_{r}\right)dr\right)_{x\in\mathbb{R}^d}$ is a continuous random field in $\Omega$ by the continuity of $Y^{s,\cdot}$ and the condition \eqref{lip1} on $b$. 
If we define $Z^{s,x}_{3}(\omega)(t)=0, \,t\in[0,s), \,\omega\in\Omega,\,x\in\mathbb{R}^d$, then $(Z^s_{3,x})_x$ is a  continuous, $(\mathcal{D}_0,\mathcal{D})-$valued random field.
Going back to \eqref{ar_1}, we deduce that
\begin{align}\label{ar_1vera}
	Y^{s,x}_t=x+\int_{s}^{t}b\left(r,Y^{s,x}_{r}\right)dr+\int_{s}^{t}\alpha\left(r,Y^{s,x}_{r}\right)dW_r+
	\int_{s}^{t}\!\int_{U_0}g\left(Y^{s,x}_{r-},r,z\right)\widetilde{N}_p\left(dr,dz\right),
\end{align}
with $t\in\left[0,T\right],\,x\in\mathbb{R}^d,\,\omega\in\Omega_s.$ \nor  We  conclude this subsection with a corollary showing the consequences of the cocycle property \eqref{cocycle} of $(Y^{s,x})_{x\in\mathbb{R}^d}$ (see also Corollary \ref{cor_flow}) on the continuous vector fields $(Z^{s,x}_{i})_{x\in\mathbb{R}^d},\,i=1,2,3.$
\begin{corollary}\label{flow_Z_cor}
	For every $s\in\left[0,T\right)$ and $\eta\in L^p\left(\Omega\right)\cap L^0(\mathcal{F}_s),\,p\ge2,$ there exists an a.s. event $\Omega_{s,\eta}$ such that
	\begin{equation}\label{dependence_Z}\begin{aligned}
				&\left(\int_{s}^{t}\alpha\left(r,Y^{s,\eta}_{r}\right)dW_r\right)\left(\omega\right)=Z_{1,t}^{s,\eta\left(\omega\right)}\left(\omega\right), 
			\qquad 
			 \left(\int_{s}^{t}\!\int_{U_0}g\left(Y^{s,\eta}_{r-},r,z\right)\widetilde{N}_p\left(dr,dz\right)\right)\left(\omega\right)=Z_{2,t}^{s,\eta\left(\omega\right)}\left(\omega\right),
			\\&\int_{s}^{t}b\left(r,Y^{s,\eta}_{r}\left(\omega\right)\right)dr=Z_{3,t}^{s,\eta\left(\omega\right)}\left(\omega\right),
		\end{aligned}
	\end{equation}
for all $t\in\left[s,T\right]$ and $\omega\in\Omega_{s,\eta}$.
	Furthermore, for every $u\in(s, T],$ there  exists an a.s. event  $\Omega_{s,u}$ such that
	\begin{equation}\label{flow_Z}
		Z_{i,u}^{s,x}+Z_{i,t}^{u,Y^{s,x}_u\left(\omega\right)}\left(\omega\right)=Z_{i,t}^{s,x}\left(\omega\right),\quad t\in\left[u,T\right],\, x\in\mathbb{R}^d,\,i=1,2,3,\,\omega\in\Omega_{s,u}.
	\end{equation}
\end{corollary}
\begin{proof}
	The equalities in \eqref{dependence_Z} can be inferred with the same argument as in the proof of Proposition \ref{p2}, recalling the estimates \eqref{l1_1}, \eqref{bdg_alpha} and \eqref{bdg_g} in the proof of Lemma \ref{l1}. 
	
	\noindent As for \eqref{flow_Z}, we focus only on $Z_1$, being the other cases analogous. Fix $0\le s <u\le T$ and compute, by the  cocycle property in \eqref{cocycle} and \eqref{dependence_Z},
	\[
		Z^{s,x}_{1,t}\left(\omega\right)=\left(\int_{s}^{u}\alpha\left(r,Y^{s,x}_{r}\right)dW_r+\int_{u}^{t}\alpha\left(r,Y^{u,Y^{s,x}_u}_{r}\right)dW_r\right)\left(\omega\right)=Z^{s,x}_{1,u}\left(\omega\right)+Z^{u,Y^{s,x}_u\left(\omega\right)}_{1,t}\left(\omega\right),
	\]
	which holds for every $t\in[u,T]$, $x\in \mathbb{Q}^d$ and $\omega\in \Omega_{s,u}$, where $\Omega_{s,u}$ is an a.s.  event  independent of $x$. In fact, the previous equation is valid also for $x\in\mathbb{R}^d\setminus \mathbb{Q}^d$, by the continuity of the  $(\mathcal{D}_0,\mathcal{D})-$valued random fields $Y^{s,\cdot},\,Z^{s,\cdot}_1$ and $Z^{u,\cdot}_1$. Hence we recover \eqref{flow_Z},  completing the proof.
\end{proof}

\subsection{The stochastic continuity in the initial time $s$}\label{st_cont_section}
Let $\mathcal{C}_0=(C(\mathbb{R}^d;\mathcal{D}_0), d_0^{lu})$ be the metric space of continuous functions defined on $\mathbb{R}^d$ with values in $\mathcal{D}_0$, where the distance $d_0^{lu}$ is given by 
\begin{equation}\label{d0_un}
d_0^{lu}\left(f,g\right)=
\sum_{N=1}^\infty\frac{1}{2^N}\frac{\sup_{\left|x\right|\le N}\norm{f\left(x\right)-g\left(x\right)}_0}{1+\sup_{\left|x\right|\le N}\norm{f\left(x\right)-g\left(x\right)}_0}
,\quad f,g\in C\left(\mathbb{R}^d;\mathcal{D}_0\right).
\end{equation}
The space $\mathcal{C}_0$ is complete but not separable. Hence, instead of endowing it with the Borel $\sigma-$algebra associated with $d^{lu}_0$, we consider the $\sigma-$algebra $\mathcal{C}$ generated by the projections $\pi_x\colon C(\mathbb{R}^d;\mathcal{D}_0) \to (\mathcal{D}_0,\mathcal{D}),\,x\in\mathbb{R}^d$, defined by $\pi_x(f)=f(x),\,f\in \mathcal{C}_0$. 
In Appendix \ref{ap_sigma} (see Lemma \ref{l=}), we prove that $\mathcal{C}$ can be read as a Borel $\sigma-$algebra of $C(\mathbb{R}^d;\mathcal{D}_0)$ endowed with the metric $d_S^{lu}$.  Although we are not going to use Lemma \ref{l=} in this paper, it is worth presenting because it is an analogue of the fact that $\mathcal{D}$ coincides with the Borel $\sigma-$algebra generated by $J_1$ in $\mathcal{D}([0,T];\mathbb{R}^d)$, see the discussion around \eqref{pil} 
and references therein.\\  
Arguments similar to those in Subsection \ref{pre_flow} about the space $(\mathcal{D}_0,\mathcal{D})$ show that $d_0^{lu}\colon (\mathcal{C}_0\times \mathcal{C}_0, \mathcal{C}\otimes\mathcal{C})\to \mathbb{R}$ is measurable. 
Indeed, for every $N\in \mathbb{N}$, the map $(f,g)\mapsto \sup_{\left|x\right|\le N}\norm{f(x)-g(x)}_0$ from $(\mathcal{C}_0\times \mathcal{C}_0, \mathcal{C}\otimes\mathcal{C})$ to $\mathbb{R}$ is measurable (by continuity,   the $\sup$ can be computed on a countable dense subset of $\mathbb{R}^d$).

\vskip 1mm
We consider the process $Y=\left(Y_s\right)_{s\in\left[0,T\right]}$, where
\begin{equation} \label{yyy}
Y_s=Y^{s,\cdot},\,0\le s<T, \quad  \text{and}\quad
Y_T\left(\omega\right)\colon\mathbb{R}^d\to\mathcal{D}_0,\;\; \left[Y_T\left(\omega\right)\left(x\right)\right]\left(t\right)=Y^{T,x}_t\left(\omega\right)= x,
\; \omega \in \Omega,
\end{equation}
 $0\le t\le T,\,x\in\mathbb{R}^d$ (see \eqref{ar_1vera}).    
 Since $(Y^{s,x})^{-1}(A)\in \mathcal{F} $ for every $A\in\mathcal{D}$ and $x\in\mathbb{R}^d$, the map $Y_s\colon \Omega\to (\mathcal{C}_0,\mathcal{C})$ is a random variable for all $s\in[0,T]$. 
This fact coupled with the above discussion shows that
\begin{equation}\label{Bic}
	\omega\mapsto d_0^{lu}\left(Y_s\left(\omega\right),Y_t\left(\omega\right)\right) \text{ is measurable for every $s,t\in[0,T]$}.
\end{equation} 
As in  \cite{AIHP18}, we want to apply \cite[Theorem $4.2$]{bez} to show the càdlàg property of the $(\mathcal{C}_0,\mathcal{C})-$valued process $Y$. The aforementioned theorem requires the stochastic continuity of $Y$, which is then the aim of this {subsection}. 

Before presenting our result (see Lemma \ref{st_cont}), we need some preparation. An important tool that we are going to use is \cite[Theorem $1.1$]{imk} (see also \cite[Theorem $3.1$]{AIHP18}), which in turn is based on a generalized \emph{Garsia--Rodemich--Rumsey} type lemma (see \cite{ai}). For the reader's convenience we report its statement, where we denote by $\log^+$ the positive part of the logarithm, namely $\log^+x=\log x\vee0,\,x\in\mathbb{R}_+$.
\begin{theorem}[\cite{imk}]\label{imk_t}
	Let $\left(M,\rho\right)$ be a separable metric space and $\phi\colon\Omega\times \mathbb{R}^d\to M$ be an $\mathcal{F}\otimes \mathcal{B}\left(\mathbb{R}^d\right)/\mathcal{B}(M)$--measurable map such that $\phi\left(\omega,\cdot\right)$ is continuous  for every $\omega\in\Omega$. Suppose that there are $p>2d$  and  $c>0$ such that
	\[
	\mathbb{E}\left[\left(\rho\left(\phi\left(\cdot,x\right),\phi\left(\cdot,y\right)\right)\right)^p\right]\le c\left|x-y\right|^p,\quad x,y\in\mathbb{R}^d.
	\]
	For any $\alpha>1$, define the map $f_\alpha\left(x\right)=([x^d\left(\log ^+x\right)^{\alpha}] \vee \,1)^{-1}\!,\,x\in\mathbb{R}_+.$ Then the function $Z\colon\Omega\to\left[0,\infty\right]$ given by 
	\begin{equation}\label{Z_imk}
		Z\left(\omega\right)=\left(\int_{{\mathbb R}^d}\int_{{\mathbb R}^d}\left(\frac{\rho\left(\phi\left(\omega,x\right),\phi\left(\omega,y\right)\right)}{\left|x-y\right|}\right)^pf_{\alpha}\left(\left|x\right|\right)f_{\alpha}\left(\left|y\right|\right)dx\,dy\right)^{\frac{1}{p}},\quad \omega\in\Omega,
	\end{equation}
	is a $p-$integrable random variable satisfying
	\begin{align}\label{est_imk}
    \rho\left(\phi\left(\omega,x\right),\phi\left(\omega,y\right)\right)\le c_0Z\left(\omega\right)\left|x-y\right|^{1-\frac{2d}{p}}\left(\left[\left(\left|x\right|\vee\left|y\right|\right)^\frac{2d}{p}\left(\log^+\left(\left|x\right|\vee\left|y\right|\right)\right)^{\frac{2\alpha}{p}}\right]\vee1\right),
	\end{align}
for all $x,y\in\mathbb{R}^d,\omega\in\Omega,$  	where $c_0$ is a positive constant depending on $\alpha,p,d$.
\end{theorem}
We wish to apply the previous result to  $(\mathcal{D}_0,\mathcal{D})-$valued, continuous random fields. Although $\mathcal{D}_0$ is not separable and $\mathcal{D}$ is not the Borel $\sigma-$algebra generated by $\norm{\cdot}_0$, this can be done thanks to the following proposition.
\begin{proposition}\label{not_sep}
	 Theorem \ref{imk_t} holds substituting $(\mathcal{D}_0, \| \cdot \|_0,\mathcal{D})$ for $(M,  \rho,  \mathcal{B}(M))$.
	\end{proposition}
\begin{proof} 
	We note that the map $\omega\mapsto \norm{\phi(\omega,x)-\phi(\omega,y)}_0$ is measurable for every $x,\,y\in\mathbb{R}^d$ (see Subsection \ref{pre_flow}, where this fact is proved for $Y^{s,\cdot}$). As a consequence, since $\phi \colon\Omega\times \mathbb{R}^d\to \mathcal{D}_0$ is $\mathcal{F}\otimes \mathcal{B}(\mathbb{R}^d)/\mathcal{D}-$measurable
	 and $ \phi (\omega, \cdot)$ is continuous for each $\omega \in \Omega$
	by hypothesis, the function $K\colon\Omega\times\mathbb{R}^d\times\mathbb{R}^d\to \mathbb{R}$ defined by
	\begin{equation*}
		K\left(\omega, x,y\right) =
		\norm{\phi\left(\omega,x\right)-\phi\left(\omega,y\right)}_0 ,\quad x,y\in\mathbb{R}^d,\,\omega\in\Omega
	\end{equation*}
	is jointly measurable. Looking now at the proof of  \cite[Theorem 1.1]{imk} and the results cited therein, it turns out that the separability of the arrival space $(M,\rho)$ is only used  to ensure the measurability of the function $Z$ in \eqref{Z_imk}. When  $M=(\mathcal{D}_0,\mathcal{D})$,  this property can be inferred directly. Indeed, for any $p>0$ and $\alpha>1$, 
	\begin{equation*}
		Z(\omega) =  \left(\int_{{\mathbb R}^d} \int_{{\mathbb R}^d} \left(\frac{  {K\left(\omega,x,y\right)}} {\left|x-y\right|}\right)^pf_\alpha\left(\left|x\right|\right)f_\alpha\left(\left|y\right|\right)dx\,dy    \right)^{1/p}
,\quad \omega\in\Omega;
	\end{equation*}
	since $K$ is non--negative, as well as $f_\alpha$ and $\left|x-y\right|^{-1},\,x\neq y,$ the desired measurability for $Z$ is given by {Tonelli's theorem}.
\end{proof}
Fix $\alpha=2$ and denote by $f=f_\alpha$. 
Notice that, for any $\gamma>0$, $(\log(x))^\gamma=\text{O}(x)$ as $x\to \infty$. Hence, combining  Proposition \ref{not_sep} with \eqref{3.6} in Lemma \ref{l1} we obtain the next corollary. 
\begin{corollary}\label{cor5}
	For every $p>2d$ and $ s\in\left[0,T\right)$, the $p-$integrable random variable $U_{s,p}$ defined by
	\[
		U_{s,p}(\omega)=
		\left(\int_{{\mathbb R}^d}\int_{{\mathbb R}^d}\left(\frac{\norm{Y^{s,x}(\omega)-Y^{s,y}(\omega)}_0}{\left|x-y\right|}\right)^pf_{}\left(\left|x\right|\right)f_{}\left(\left|y\right|\right)dx\,dy\right)^{\frac{1}{p}},\quad \omega\in\Omega,
	\] 
	is such that
	\begin{align}\label{est23}
		\sup_{0\le t \le T } \left|Y^{s,x}_t\left(\omega\right)   -   Y^{s,y}_t \left(\omega\right)\right|   \le c_1 \, U_{s,p}\left(\omega\right)   \left|x-y\right|^{1- 2d/p}\left[ \left(\left|x\right| \vee \left|y\right|\right)^{\frac{2d + 1}{p}} \vee 1\right],\quad 
		x, y \in {\mathbb R}^d,\,\omega\in\Omega,
	\end{align}
	where $c_1=c_1\left(d,p\right)>0$. Furthermore,
	\begin{equation} \label{ri4}
		\sup_{s \in \left[0,T\right)} \mathbb{E} \left[U_{s,p}^p\right] \le 4^{p-1}e^{cT}c_2,
	\end{equation}
	where $c>0$ is the same constant as in \eqref{3.6} and $c_2=(\int_{{\mathbb R}^d}f\left(\left|x\right|\right)dx)^2<\infty$.
\end{corollary}
We remark that results similar to Corollary \ref{cor5} hold with the continuous random fields $Z_i^{s,\cdot},\,i=1,2,3,$ instead of $Y^{s,\cdot}$. We are now ready  to prove the main result of this {\color{black}subsection}. 
\begin{lemma}\label{st_cont}
	The $(\mathcal{C}_0,\mathcal{C})-$valued process $Y=\left(Y_s\right)_{s\in[0,T]}$ considered in  \eqref{yyy}  is continuous in probability.
\end{lemma}
\begin{proof}
Fix $s\in\left[0,T\right]$ and  take a sequence $\left(s_n\right)_n\subset\left[0,T\right]$ such that $s_n\to s$ as $n\to\infty$. We want to show that
\begin{equation}\label{cont_pr}
	\mathbb{E}\Big[\sup_{\left|x\right|\le N}\norm{Y^{s,x}-Y^{s_n,x}}_0\Big]=\mathbb{E}\Big[\sup_{\left|x\right|\le N}\sup_{0\le t\le T}\left|Y^{s,x}_t-Y^{s_n,x}_t\right|\Big]\underset{n\to\infty}{\longrightarrow} 0,\quad N\ge1.
\end{equation}
In effect, this is a sufficient condition to obtain the stochastic continuity of $Y$ in $s$, as the next argument explains. By definition of continuity in probability, we aim to prove that $\lim_{n\to\infty}\mathbb{P}\left(d^{lu}_0\left(Y_{s_n},Y_s\right)>\epsilon\right)=0$ for any $\epsilon>0$, which is equivalent to \[
\lim_{n\to \infty}\mathbb{E}\left[d^{lu}_0\left(Y_{s_n},Y_s\right)\left(1+d^{lu}_0\left(Y_{s_n},Y_s\right)\right)^{-1}\right]=0.
\]
Therefore, it is enough to show that
\begin{equation}\label{dom}
	\lim_{n\to \infty}\mathbb{E}\left[d^{lu}_0\left(Y_{s_n},Y_s\right)\right]=\lim_{n\to\infty}\sum_{N=1}^\infty \frac{1}{2^N}\mathbb{E}\left[\frac{\sup_{\left|x\right|\le N}\norm{Y^{s,x}-Y^{s_n,x}}_0}{1+\sup_{\left|x\right|\le N}\norm{Y^{s,x}-Y^{s_n,x}}_0}\right]=0.
\end{equation}
An application of the dominated convergence theorem  --endowing $\mathbb{N}$ with the canonical counting measure-- gives \eqref{dom} knowing \eqref{cont_pr}.
 
Fix $N\ge1$. We start off by proving the right stochastic continuity in a point $s\in\left[0,T\right)$. Take a sequence $\left(s_n\right)_n\subset\left(s,T\right)$ such that $s_n\downarrow s$ and  split the expectation in \eqref{cont_pr} as follows:
\begin{multline}\label{split_time}
\mathbb{E}\left[\sup_{\left|x\right|\le N}\sup_{0\le t\le T}\left|Y^{s,x}_t-Y^{s_n,x}_t\right|\right]\le
\mathbb{E}\left[\sup_{\left|x\right|\le N}\sup_{0\le t\le s}\left|Y^{s,x}_t-Y^{s_n,x}_t\right|\right]+
\mathbb{E}\left[\sup_{\left|x\right|\le N}\sup_{s\le t\le s_n}\left|Y^{s,x}_t-Y^{s_n,x}_t\right|\right]\\+
\mathbb{E}\left[\sup_{\left|x\right|\le N}\sup_{s_n\le t\le T}\left|Y^{s,x}_t-Y^{s_n,x}_t\right|\right],\quad n\in \mathbb{N}.
\end{multline}
We analyze the second and third addends in the right--hand side of \eqref{split_time}, the first being equal to $0$. As for the second,  by \eqref{ar_1vera}, for every $n\in \mathbb{N}$, we can find an a.s. event $\Omega_{s,s_n}$ independent of $x$ where
\begin{multline}\label{flo1}
\sup_{s\le t\le s_n}\left|Y^{s,x}_t-Y^{s_n,x}_t\right|
\le
\sup_{s\le t\le s_n}\left|\int_{s}^{t}b\left(r,Y^{s,x}_{r}\right)dr\right|
+
\sup_{s\le t\le s_n}\left|\int_{s}^t\alpha\left(r,Y^{s,x}_{r}\right)dW_r\right|\\
+\sup_{s\le t\le s_n}\left|\int_{s}^{t}\!\int_{U_0}g\left(Y^{s,x}_{r-},r,z\right)\widetilde{N}_p\left(dr,dz\right)\right|,\quad x\in\mathbb{R}^d.
\end{multline}
Thus,  we prove that, as $n\to \infty$,
\begin{align}\label{second_pr}
\begin{split}
&\mathbb{E}\left[\sup_{\left|x\right|\le N}\sup_{s\le t\le s_n}\left|\int_{s}^{t}b\left(r,Y^{s,x}_{r}\right)dr\right|\right]\to 0,\qquad
\mathbb{E}\left[\sup_{\left|x\right|\le N}\sup_{s\le t\le s_n}\left|\int_{s}^{t}\alpha\left(r,Y^{s,x}_{r}\right)dW_r\right|\right]\to 0,\\&
\mathbb{E}\left[\sup_{\left|x\right|\le N}\sup_{s\le t\le s_n}\left|\int_{s}^{t}\!\int_{U_0}g\left(Y^{s,x}_{r-},r,z\right)\widetilde{N}_p\left(dr,dz\right)\right|\right]\to 0.
\end{split}
\end{align}
The limits in \eqref{second_pr} are all dealt with using the same technique, so we just focus on the one  appearing in the second line. In particular, we want to apply Proposition \ref{not_sep} to the continuous, $(\mathcal{D}_0,\mathcal{D})-$valued random field $(Z^{s,x}_2)_{x\in\mathbb{R}^d}$. To do this, we take $\gamma>2d$ and write
\begin{multline}\label{c_1v}
	 \mathbb{E}\Big[\sup_{\left|x\right|\le N}\sup_{s\le t\le s_n} \left|Z^{s,x}_{2,t}\right|^{{\gamma}}\Big]\le 2^{\gamma-1}
	\Big(\mathbb{E}\Big[\sup_{s\le t\le s_n} \left|Z^{s,0}_{2,t}\right|^{{\gamma}}\Big]
	+\mathbb{E}\Big[\sup_{\left|x\right|\le N}\sup_{s\le t\le s_n} \left|Z^{s,x}_{2,t}-Z^{s,0}_{2,t}\right|^{{\gamma}}\Big]\Big)\\\eqqcolon 2^{\gamma-1}\left(\mathbf{\upperRomannumeral{1}}_n+\mathbf{\upperRomannumeral{2}}_n\left(N\right)\right).
\end{multline}
As for $\mathbf{\upperRomannumeral{1}}_n$, by the linear growth condition \eqref{lg1} for $g$ and \cite[Theorem $2.11$]{KU04} we infer that, for some $C>0,$
\begin{align}\label{ku3.6}
&\notag\mathbf{\upperRomannumeral{1}}_n=\mathbb{E}\left[\sup_{s\le t\le s_n}\left|
\int_{s}^{t}\!\int_{U_0}g\left(Y^{s,0}_{r-},r,z\right)\widetilde{N}_p\left(dr,dz\right)
\right|^{\gamma}\right]\\\notag&\quad 
\le c_{d,\gamma}\!\left(\!
\mathbb{E}\left[\left(K_2\int_{s}^{s_n}\left(1+\left|Y^{s,0}_{r-}\right|\right)^2dr\right)^{\frac{\gamma}{2}}\right]+
\mathbb{E}\left[K_{\gamma}\int_{s}^{s_n} \left(1+\left|Y^{s,0}_{r-}\right|\right)^{\gamma} dr\right]
\right)
\\&\quad 
\le c_{d,\gamma}\left(\left(s_n-s\right)^{\frac{\gamma}{2}}K_2^{\frac{\gamma}{2}}
+
\left(s_n-s\right)K_{\gamma}
\right)
\mathbb{E}\left[\sup_{s\le t\le T}\left(1+\left|Y^{s,0}_t\right|\right)^\gamma\right]
\le C\cdot \text{o}\left(1\right),\quad \text{as }n\to\infty,
\end{align}
where in the last inequality we use \cite[Equation $\left(3.6\right)$]{KU04}. As regards $\mathbf{\upperRomannumeral{2}}_n$, by the estimates in \eqref{3.6} and \eqref{bdg_g} we have, for some $C_0=C_0\left(\gamma,d,m,T,K_2,K_\gamma\right)>0$,
\begin{equation}\label{ku3.7}
\mathbb{E}\left[\sup_{s\le t\le s_n}\left|
\int_{s}^{t}\!\int_{U_0}\left(g\left(Y^{s,x}_{r-},r,z\right)-g\left(Y^{s,y}_{r-},r,z\right)\right)\widetilde{N}_p\left(dr,dz\right)
\right|^{\gamma}\right]
\le C_0\,\left(s_n-s\right)\left|x-y\right|^\gamma,\quad x,y\in\mathbb{R}^d.
\end{equation}
 Hence  Proposition \ref{not_sep} yields, arguing as in Corollary \ref{cor5} with $Y^{s,\cdot}$ replaced by $Z_2^{s,\cdot}$,
 \[
\sup_{\left|x\right|\le N}\sup_{s\le t\le s_n} \left|Z^{s,x}_{2,t}-Z^{s,0}_{2,t}\right|\le 
c\left(d,p\right)\widetilde{U}^{(2)}_{s,s_n,\gamma}N^{1+\frac{1}{\gamma}}
,\quad \text{in }\Omega,
 \]
 where 
 	\begin{equation*}
 \widetilde{U}^{(2)}_{s,s_n,\gamma}(\omega)=
 \left(\int_{{\mathbb R}^d}\int_{{\mathbb R}^d}\left(\frac{\sup_{s\le t \le s_n}{\left|Z_{2,t}^{s,x}(\omega)-Z_{2,t}^{s,y}(\omega)\right|}}{\left|x-y\right|}\right)^\gamma f_{}\left(\left|x\right|\right)f_{}\left(\left|y\right|\right)dx\,dy\right)^{\frac{1}{\gamma}},\quad \omega\in\Omega.
 \end{equation*}
  In particular,  $\widetilde{U}^{(2)}_{s,s_n,\gamma}$ is  a  $\gamma-$integrable random variable such that, by \eqref{ku3.7}, 
 $
 	\mathbb{E}\big[\big(\widetilde{U}^{(2)}_{s,s_n,\gamma}\big)^\gamma\big]\le C_1\left(s_n-s\right),
 $
 where $C_1=C_1(\gamma,d,m,T,K_2,K_\gamma)$.
 Consequently, 
 \begin{equation}\label{ku_37}
 	\mathbf{\upperRomannumeral{2}}_n\left(N\right)\le c^\gamma C_1 N^{\gamma+1}	\left(s_n-s\right)\to 0 \quad \text{as }n\to \infty.
 \end{equation}
Combining \eqref{ku3.6}-\eqref{ku_37} in \eqref{c_1v} we deduce that 
\begin{equation}\label{4.4ku}
\mathbb{E}\left[\sup_{\left|x\right|\le N}\sup_{s\le t\le s_n}\left|\int_{s}^{t}\!\int_{U_0}g\left(Y^{s,x}_{r-},r,z\right)\widetilde{N}_p\left(dr,dz\right)\right|^\gamma\right]\to 0 \quad \text{ as }n\to \infty.
\end{equation}
Note that \eqref{4.4ku} holds  for every $\gamma>0$ by {Jensen's inequality}, hence we recover the second line of \eqref{second_pr} as a particular case. In an analogous way one can prove that, for any $\gamma>0$,
\begin{equation}\label{vabedai}
\mathbb{E}\left[\sup_{\left|x\right|\le N}\sup_{s\le t\le s_n}\left|\int_{s}^{t}b\left(r,Y^{s,x}_{r}\right)dr\right|^{\gamma}\right]\underset{n\to\infty}{\longrightarrow}0,\quad
\mathbb{E}\left[\sup_{\left|x\right|\le N}\sup_{s\le t\le s_n}\left|\int_{s}^{t}\alpha\left(r,Y^{s,x}_{r}\right)dW_r\right|^{\gamma}\right]\underset{n\to\infty}{\longrightarrow} 0,
\end{equation}
whence
\begin{equation}\label{second_pr1}
\mathbb{E}\left[\sup_{\left|x\right|\le N}\sup_{s\le t\le s_n}\left|Y^{s,x}_t-Y^{s_n,x}_t\right|\right]\to 0 \quad \text{ as }n\to \infty.
\end{equation}
Next, we study the third addend in \eqref{split_time} using the flow property in Corollary \ref{cor_flow}.
In particular, by \eqref{flow} there exists an almost sure event $\Omega_{s,s_n}$ independent of $x$ where
\[
\sup_{s_n\le t\le T}\left|Y^{s,x}_t-Y^{s_n,x}_t\right|=
\sup_{s_n\le t\le T}\left|Y^{s_n,Y^{s,x}_{s_n}}_t-Y^{s_n,x}_t\right|\le 
\sup_{0\le t\le T}\left|Y^{s_n,Y^{s,x}_{s_n}}_t-Y^{s_n,x}_t\right|,\quad x\in\mathbb{R}^d.
\]
Now we choose $p\ge2d+1$ and apply \eqref{est23} in Corollary \ref{cor5} to deduce that
\begin{equation*}
\sup_{0\le t\le T}\left|Y^{s_n,Y^{s,x}_{s_n}}_t-Y^{s_n,x}_t\right|\le
c_1\,U_{s_n,p}\left[\left(\left|x\right|\vee\left|Y^{s,x}_{s_n}\right|\right)^{\frac{2d+1}{p}}\vee 1\right]\left|x-Y^{s,x}_{s_n}\right|^{1-\frac{2d}{p}},\quad x\in\mathbb{R}^d,
\end{equation*}
which holds  in the whole space $\Omega.$ Moreover, by \eqref{ar_1vera}, there exists a full probability set $\Omega_s$ where 
\[
\left|Y^{s,x}_{s_n}\right|\le N+\left|\int_{s}^{s_n}b\left(r,Y^{s,x}_{r}\right)dr\right|+
\left|\int_{s}^{s_n}\alpha\left(r,Y^{s,x}_{r}\right)dW_r\right|+
\left|\int_{s}^{s_n}\!\int_{U_0}g\left(Y^{s,x}_{r-},r,z\right)\widetilde{N}_p\left(dr,dz\right)\right|,\quad \left|x\right|\le N.
\]
Combining the three previous expressions we get the existence of an almost sure event $\Omega'_{s,s_n}$ where
\begin{align}\label{flo2}
&\notag\sup_{\left|x\right|\le N}\sup_{s_n\le t\le T}\left|Y^{s,x}_t-Y^{s_n,x}_t\right|\le c_1\,U_{s_n,p}\sup_{\left|x\right|\le N}\Bigg\{\\&\quad\! \notag
\left[N^{\frac{2d+1}{p}}+\left|\int_{s}^{s_n}b\left(r,Y^{s,x}_{r}\right)dr\right|^{\frac{2d+1}{p}}+\left|\int_{s}^{s_n}\alpha\left(r,Y^{s,x}_{r}\right)dW_r\right|^{\frac{2d+1}{p}}+
\left|\int_{s}^{s_n}\!\int_{U_0}g\left(Y^{s,x}_{r-},r,z\right)\widetilde{N}_p\left(dr,dz\right)\right|^{\frac{2d+1}{p}}\right]\\&\quad \!
\times\!\!
\left[\left|\int_{s}^{s_n}b\left(r,Y^{s,x}_{r}\right)dr\right|^{1-\frac{2d}{p}}\!+\left|\int_{s}^{s_n}\alpha\left(r,Y^{s,x}_{r}\right)dW_r\right|^{1-\frac{2d}{p}}\!\!+
\left|\int_{s}^{s_n}\!\int_{U_0}g\left(Y^{s,x}_{r-},r,z\right)\widetilde{N}_p\left(dr,dz\right)\right|^{1-\frac{2d}{p}}
\right]
\Bigg\}.
\end{align}
In the estimates in \eqref{flo2}, we use the subadditivity  property of the function $x^r$, for $r=\left(2d+1\right)/p$ or $r=1-2d/p$.
Let us denote by $C_2\coloneqq\sup_{0< t< T}\mathbb{E}\left[U_{t,p}^2\right]$, which is finite by \eqref{ri4} in Corollary \ref{cor5}. Taking the expected value in \eqref{flo2}, we apply the Cauchy--Schwarz inequality to write
{\begin{align*}
&\mathbb{E}\left[\sup_{\left|x\right|\le N}\sup_{s_n\le t\le T}\left|Y^{s,x}_t-Y^{s_n,x}_t\right|\right]^2\le12\,c_1\,C_2\,\mathbb{E}\Bigg[\sup_{\left|x\right|\le N}\Bigg\{
\\&\quad 
N^{\frac{4d+2}{p}}+\left|\int_{s}^{s_n}b\left(r,Y^{s,x}_{r}\right)dr\right|^{\frac{4d+2}{p}}+\left|\int_{s}^{s_n}\alpha\left(r,Y^{s,x}_{r}\right)dW_r\right|^{\frac{4d+2}{p}}+
\left|\int_{s}^{s_n}\!\int_{U_0}g\left(Y^{s,x}_{r-},r,z\right)\widetilde{N}_p\left(dr,dz\right)\right|^{\frac{4d+2}{p}}\Bigg\}\\&\quad
\times\sup_{\left|x\right|\le N}\left\{\left|\int_{s}^{s_n}b\left(r,Y^{s,x}_{r}\right)dr\right|^{2-\frac{4d}{p}}\!\!+\left|\int_{s}^{s_n}\alpha\left(r,Y^{s,x}_{r}\right)dW_r\right|^{2-\frac{4d}{p}}\!\!+
\left|\int_{s}^{s_n}\!\!\int_{U_0}g\left(Y^{s,x}_{r-},r,z\right)\widetilde{N}_p\left(dr,dz\right)\right|^{2-\frac{4d}{p}}
\right\}\!\Bigg].
\end{align*}}
Invoking the Cauchy--Schwarz inequality one more time,
\begin{align*}
&\mathbb{E}\left[\sup_{\left|x\right|\le N}\sup_{s_n\le t\le T}\left|Y^{s,x}_t-Y^{s_n,x}_t\right|\right]^2\le12\sqrt{12}\,c_1\,C_2\Bigg(N^{\frac{4d+2}{p}}+\Bigg(\mathbb{E}\Bigg[\sup_{\left|x\right|\le N}\Bigg\{
\\&\quad\!
\left|\int_{s}^{s_n}b\left(r,Y^{s,x}_{r}\right)dr\right|^{\frac{8d+4}{p}}
\!+\left|\int_{s}^{s_n}\alpha\left(r,Y^{s,x}_{r}\right)dW_r\right|^{\frac{8d+4}{p}}\!+
\left|\int_{s}^{s_n}\!\!\int_{U_0}g\left(Y^{s,x}_{r-},r,z\right)\widetilde{N}_p\left(dr,dz\right)\right|^{\frac{8d+4}{p}}\Bigg\}\Bigg]\Bigg)^{\frac{1}{2}}\Bigg)\Bigg(
\mathbb{E}\Bigg[
\\&\quad\!
\sup_{\left|x\right|\le N}\left\{\left|\int_{s}^{s_n}b\left(r,Y^{s,x}_{r}\right)dr\right|^{4-\frac{8d}{p}}\!\!+\left|\int_{s}^{s_n}\alpha\left(r,Y^{s,x}_{r}\right)dW_r\right|^{4-\frac{8d}{p}}\!\!+
\left|\int_{s}^{s_n}\!\!\int_{U_0}g\left(Y^{s,x}_{r-},r,z\right)\widetilde{N}_p\left(dr,dz\right)\right|^{4-\frac{8d}{p}}\right\}\Bigg]\Bigg)^{\frac{1}{2}}\!.
\end{align*}
Notice that the right--hand side of the previous equation goes to $0$ as $n\to \infty$ by \eqref{4.4ku}-\eqref{vabedai}. Summing up, 
\begin{equation}\label{third_pr}
\mathbb{E}\left[\sup_{\left|x\right|\le N}\sup_{s_n\le t\le T}\left|Y^{s,x}_t-Y^{s_n,x}_t\right|\right]\to 0 \quad \text{ as }n\to \infty.
\end{equation}

Combining \eqref{third_pr} with \eqref{second_pr1} we obtain \eqref{cont_pr},
which yields the desired right--stochastic continuity for the process $Y$ in $s$.  The left--continuity of $Y$ in $(0,T]$ can be argued in a similar way. The proof is complete.
\end{proof}
Recall the continuous, $(\mathcal{D}_0,\mathcal{D})-$valued random fields $Z^{s,\cdot}_i,\,i=1,2,3,$ introduced at the end of Subsection \ref{pre_flow} and appearing on the right--hand side of \eqref{ar_1vera}. For every $i=1,2,3,$ we define the $(\mathcal{C}_0,\mathcal{C})-$valued stochastic process $Z_i=(Z_{i,s})_{s\in[0,T]}$, whose random variables are $Z_{i,s}=Z_{i}^{s,\cdot},\,0\le s <T$, and we set, for every $\omega\in\Omega$,
\[
	Z_{i,T}\left(\omega\right)\colon\mathbb{R}^d\to\mathcal{D}_0,\qquad \left[Z_{i,T}\left(\omega\right)\left(x\right)\right]\left(t\right)=Z^{T,x}_{i,t}\left(\omega\right)= 0,\quad x\in\mathbb{R}^d,\,0\le t\le T.
\]
Using a strategy similar to the one followed in  Lemma \ref{st_cont} for the process $Y$, in the next result we show the continuity in probability of the processes $Z_i$.
\begin{corollary}\label{st_cont_Z}
		For every $i=1,2,3,$ the $(\mathcal{C}_0,\mathcal{C})-$valued process $Z_i=\left(Z_{i,s}\right)_{s\in[0,T]}$ is continuous in probability.
\end{corollary}
\begin{proof}
	Fix $i=1,2,3$ and $s\in\left[0,T\right)$. As in Lemma \ref{st_cont}, we only prove the right stochastic  continuity of $Z_i$ in $s$. Hence we consider a sequence $\left(s_n\right)_n\subset\left(s,T\right)$ such that $s_n\downarrow s$ and we show that (cf. \eqref{cont_pr})
	\begin{equation}\label{cont_pr_Z}
		\mathbb{E}\left[\sup_{\left|x\right|\le N}\norm{Z_i^{s,x}-Z_i^{s_n,x}}_0\right]=\mathbb{E}\left[\sup_{\left|x\right|\le N}\sup_{0\le t\le T}\left|Z^{s,x}_{i,t}-Z^{s_n,x}_{i,t}\right|\right]\underset{n\to\infty}{\longrightarrow} 0,\quad N\ge1.
	\end{equation}
	Taking $N\ge1$, we  split the expectation in \eqref{cont_pr_Z} as follows:
	\begin{align*}
		\notag\mathbb{E}\left[\sup_{\left|x\right|\le N}\sup_{0\le t\le T}\left|Z^{s,x}_{i,t}-Z^{s_n,x}_{i,t}\right|\right]&\le
		\mathbb{E}\left[\sup_{\left|x\right|\le N}\sup_{0\le t\le s}\left|Z^{s,x}_{i,t}-Z^{s_n,x}_{i,t}\right|\right]+
		\mathbb{E}\left[\sup_{\left|x\right|\le N}\sup_{s\le t\le s_n}\left|Z^{s,x}_{i,t}-Z^{s_n,x}_{i,t}\right|\right]\\&\quad +
		\mathbb{E}\left[\sup_{\left|x\right|\le N}\sup_{s_n\le t\le T}\left|Z^{s,x}_{i,t}-Z^{s_n,x}_{i,t}\right|\right]
		\eqqcolon \left(\mathbf{\upperRomannumeral{1}}^{(i)}_n
		+
		\mathbf{\upperRomannumeral{2}}^{(i)}_n
		+\mathbf{\upperRomannumeral{3}}^{(i)}_n\right)\left(N\right),\quad n\in \mathbb{N}.
	\end{align*}
Since $\mathbf{\upperRomannumeral{1}}^{(i)}_n\left(N\right)=0,\,n\in\mathbb{N}$, by construction, and $\lim_{n\to \infty}\mathbf{\upperRomannumeral{2}}^{(i)}_n\left(N\right)=0$ by \eqref{second_pr}, we only study $\mathbf{\upperRomannumeral{3}}^{(i)}_n\left(N\right).$ In particular, we prove that $\lim_{n\to \infty}\mathbf{\upperRomannumeral{3}}^{(i)}_n\left(N\right)=0$. To do this, we invoke Corollary \ref{flow_Z_cor}, precisely \eqref{flow_Z}, which guarantees the existence of an almost sure event $\Omega_{s,s_n}$ --independent of $x$-- where, for every $x\in\mathbb{R}^d$,
\begin{equation}\label{part_1}
\sup_{s_n\le t\le T}\left|Z^{s,x}_{i,t}-Z^{s_n,x}_{i,t}\right|=
\sup_{s_n\le t\le T}\left|Z^{s,x}_{i,s_n}+Z^{s_n,Y^{s,x}_{s_n}}_{i,t}-Z^{s_n,x}_{i,t}\right|
\le 
\sup_{s\le t\le s_n}\left|Z^{s,x}_{i,t}\right|+
\sup_{0\le t\le T}\left|Z^{s_n,Y^{s,x}_{s_n}}_{i,t}-Z^{s_n,x}_{i,t}\right|.
\end{equation}
Choose $p\ge2d+1$. By Lemma \ref{l1} and the estimates in its proof, we can apply Proposition \ref{not_sep} to deduce the existence of a $p-$integrable random variable $\widetilde{U}^{(i)}_{s_n,p}$  such that, for some $c=c(p,d)>0,$
\begin{equation*}
	\sup_{0\le t\le T}\left|Z^{s_n,Y^{s,x}_{s_n}}_{i,t}-Z^{s_n,x}_{i,t}\right|\le
	c\,\widetilde{U}^{(i)}_{s_n,p}\left[\left(\left|x\right|\vee\left|Y^{s,x}_{s_n}\right|\right)^{\frac{2d+1}{p}}\vee 1\right]\left|x-Y^{s,x}_{s_n}\right|^{1-\frac{2d}{p}},\quad x\in\mathbb{R}^d,
\end{equation*}
which holds  in the whole space $\Omega.$ Since (again by Lemma \ref{l1}) $\sup_{n\in\mathbb{N} }\mathbb{E}\big[\big(\widetilde{U}^{(i)}_{s_n,p}\big)^2\big]<\infty$, we can proceed as in \eqref{flo2} and the subsequent estimates to obtain
\[
	\lim_{n\to \infty}\mathbb{E}\Big[\sup_{\left|x\right|\le N}\sup_{0\le t\le T}\Big|Z^{s_n,Y^{s,x}_{s_n}}_{i,t}-Z^{s_n,x}_{i,t}\Big|\Big]=0.
\]
Taking now the supremum over $|x|\le N$ and expectations in \eqref{part_1}, we have
\[
\mathbf{\upperRomannumeral{3}}^{(i)}_n\left(N\right)\le 
\mathbf{\upperRomannumeral{2}}^{(i)}_n\left(N\right)+
	\mathbb{E}\Big[\sup_{\left|x\right|\le N}\sup_{0\le t\le T}\Big|Z^{s_n,Y^{s,x}_{s_n}}_{i,t}-Z^{s_n,x}_{i,t}\Big|\Big]\to 0\quad \text{as }n\to\infty,\vspace{-.6em}
\]
completing the proof.
\end{proof} 
\subsection{The càdlàg property in the initial time \emph{s}}\label{cadlag_sec}
In this subsection we study the existence of a càdlàg modification of the process $Y=\left(Y_s\right)_s$ in \eqref{yyy}. 
 Such a property is obtained by Theorem \ref{thm_BZ}, which is a reformulation of \cite[Theorem $4.2$]{bez} (see also \cite[Theorem $1.2$]{BZ2}). 
 We compare
 Theorem \ref{thm_BZ}
 with the original result in  \cite{bez}-\cite{BZ2} in  Remark \ref{rem_BZ}. 
\begin{theorem}\label{thm_BZ}
	Let $X = (X_t)_{t\in[0,T ]}$ be a family of functions defined on a complete probability space $(\Xi, \mathcal{G},\mathbb{Q})$  with values in a complete metric space $(E, \Delta)$. Let $0 \le s \le t \le u \le T$ and denote by $\Delta(s,t,u)=\Delta(X_s,X_t)\wedge \Delta(X_t,X_u)$. Suppose that the map $\omega\mapsto \Delta(X_s(\omega),X_t(\omega))$ is measurable for any $s,t\in[0,T]$, and that  $X$ is continuous in probability. If  there exist continuous, increasing functions $\delta\colon [0,T]\to \mathbb{R}_+$ and $\theta\colon [0,T\vee 1]\to \mathbb{R}_+$, with $\delta(0)=\theta(0)=0$ and $\theta$ concave,  such that
	\begin{equation}\label{class_BZ}
		\mathbb{E}_\mathbb{Q}\left[\Delta (s,t,u)1_A\right]\le \delta\left(u-s\right)\cdot \theta\left(\mathbb{Q}\left(A\right)\right),\quad 0\le s \le t \le u\le T,\,A\in\mathcal{G},
	\end{equation}
	and that
	\begin{equation}\label{cond_inte}
		\int_{0}^{T}u^{-2}\delta\left(u\right)\theta\left(u\right) du <\infty, 
	\end{equation}
then $X$ has a càdlàg modification $X'$ (modification means that the map $\omega\mapsto \Delta(X_t(\omega),X_t'(\omega))$ is  equal to 0, $\mathbb{Q}-$a.s., for any $t \in [0,T]$). 
\end{theorem}
\begin{rem}\label{rem_BZ}
	Compared to the original assertion in \cite{bez}-\cite{BZ2}, in Theorem \ref{thm_BZ} we do not require the functions $X_t\colon \Xi\to (E,\Delta)$ to be measurable with respect to the Borel $\sigma-$algebra of $E$. This is crucial for our arguments, because we are going to apply Theorem \ref{thm_BZ} to the $(\mathcal{C}_0,\mathcal{C})-$valued process $Y$ and $\mathcal{C}$ is not the Borel $\sigma-$algebra of $\mathcal{C}_0$.  The hypothesis on the measurability of the map $\omega\mapsto \Delta(X_s(\omega),X_t(\omega))$ is inspired by an analogous assumption in \cite[Lemma A.$2.37$]{Bicht}, where the Kolmogorov--Chentsov continuity criterion is proved without supposing the separability of the arrival space. In any case, such an hypothesis does not alter the strategy of the proof of Theorem \ref{thm_BZ}, which is presented in Appendix \ref{ap_BZ} for the sake of completeness.
\end{rem}
The next corollary, which gives a sufficient condition for the existence of a càdlàg version, can be easily deduced from Theorem \ref{thm_BZ}.
\begin{corollary}\label{cor_4.2}
	Under the same hypotheses as in Theorem \ref{thm_BZ}, if there exist $q>1/2,\,r>0$ and $C>0$ such that 
	\begin{equation}\label{cond_BZ}
		\mathbb{E}_\mathbb{Q}\left[\Delta\left(X_s,X_t\right)^q\cdot \Delta\left(X_t,X_u\right)^q\right]\le C\left|u-s\right|^{1+r},\quad 0\le s \le t\le u\le T,
	\end{equation}
then $X$ has a càdlàg modification.
\end{corollary} 
\begin{proof}
Define $\theta(h)= h^{1-\frac{1}{2q}}$ and $\delta(h)=C^{\frac{1}{2q}} h^{\frac{1+r}{2q}},\,h\ge 0$. Then \eqref{class_BZ} can be deduced from \eqref{cond_BZ} as in \cite[Corollary 4.2]{AIHP18}, while \eqref{cond_inte} is satisfied because $r>0$.
\end{proof}
We are now ready to present  the main result of Section \ref{sec_small}. In the proof, we employ the concept of strong solution to Equation \eqref{SDE} (see \eqref{strong}), which allows to follow an argument relying on  conditional expectations with respect to the augmented $\sigma-$algebra generated by $W$ and $N_p$. 
\begin{theorem}\label{cadlag}
There exists a càdlàg version $Z$ of the $(\mathcal{C}_0,\mathcal{C})-$valued process $Y=\left(Y_s\right)_s$.
\end{theorem}
\begin{proof}
We will apply Corollary \ref{cor_4.2}. Note that, by  the completeness of the probability space, the càdlàg version  $Z$ will be automatically a $(\mathcal{C}_0,\mathcal{C})-$valued process.
 
Recall that in \eqref{Bic} we have shown the measurability of the map $\omega\mapsto d^{lu}_0(Y_s(\omega),Y_t(\omega)),\,s,t\in[0,T]$. Thus, according to Corollary \ref{cor_4.2}, in order to find a càdlàg modification of the stochastically continuous process $Y$ it is sufficient to determine $q>1/2, r>0$ and a constant $C>0$ such that
\begin{equation}\label{goal}
\mathbb{E}\left[d_0^{lu}\left(Y_s,Y_u\right)^q\cdot d_0^{lu}\left(Y_u,Y_v\right)^q\right]\le C \left|v-s\right|^{1+r},\quad 0\le s< u< v \le T.
\end{equation}

Let us take a triplet of times $\left(s,u,v\right)$, with $0\le s< u< v \le T$, and denote by $\rho= v-s$. We can assume $\rho<1$, otherwise \eqref{goal} is trivially satisfied for any choice of $q,r\in\mathbb{R}_+$ and $C\ge1$.  By the computations in Subsection~\ref{st_cont_section}, precisely \eqref{flo1}-\eqref{flo2}, there exists an a.s. event $\Omega_{s,u}$ where, for every $p\ge 2d+1$ and $N\ge 1$,
\begin{align}\label{add1}
&\notag\sup_{\left|x\right|\le N}\sup_{0\le t\le T}\left|Y_t^{s,x}-Y_t^{u,x}\right|\\\notag&\quad
\le
\sup_{\left|x\right|\le N}\left\{\sup_{s\le t\le u}\left|\int_{s}^{t}b\left(r,Y^{s,x}_{r}\right)dr\right|
+
\sup_{s\le t\le u}\left|\int_{s}^t\alpha\left(r,Y^{s,x}_{r}\right)dW_r\right|
+
\sup_{s\le t\le u}\left|\int_{s}^{t}\!\int_{U_0}g\left(Y^{s,x}_{r-},r,z\right)\widetilde{N}_p\left(dr,dz\right)\right|\right\}
\\\notag\notag&\qquad
+c_1\,U_{u,p}\sup_{\left|x\right|\le N}\bigg\{\\\notag&\qquad
\left[N^{\frac{2d+1}{p}}+\left|\int_{s}^{u}b\left(r,Y^{s,x}_{r}\right)dr\right|^{\frac{2d+1}{p}}+\left|\int_{s}^{u}\alpha\left(r,Y^{s,x}_{r}\right)dW_r\right|^{\frac{2d+1}{p}}+
\left|\int_{s}^{u}\!\int_{U_0}g\left(Y^{s,x}_{r-},r,z\right)\widetilde{N}_p\left(dr,dz\right)\right|^{\frac{2d+1}{p}}\right]\\&\qquad
\times
\left[\left|\int_{s}^{u}b\left(r,Y^{s,x}_{r}\right)dr\right|^{1-\frac{2d}{p}}+\left|\int_{s}^{u}\alpha\left(r,Y^{s,x}_{r}\right)dW_r\right|^{1-\frac{2d}{p}}+
\left|\int_{s}^{u}\!\int_{U_0}g\left(Y^{s,x}_{r-},r,z\right)\widetilde{N}_p\left(dr,dz\right)\right|^{1-\frac{2d}{p}}
\right]
\bigg\}\notag\\&\quad
\eqqcolon \mathbf{\upperRomannumeral{1}}_1+c_1U_{u,p}\mathbf{\upperRomannumeral{2}}_1.\vspace{-.6em}
\end{align}
 If we compute the product appearing in $\mathbf{\upperRomannumeral{2}}_1$, then introducing the set $A=\left\{\frac{2d+1}{p},1-\frac{2d}{p}\right\}$ we can estimate \vspace{-.6em}
 \begin{align*}
 	&\mathbf{\upperRomannumeral{2}}_1\le \sup_{\left|x\right|\le N}\bigg\{N^{\frac{2d+1}{p}}
 	\sup_{s\le t\le u}\left|\int_{s}^{t}b\left(r,Y^{s,x}_{r}\right)dr\right|^{1-\frac{2d}{p}}
 	+
 	N^{\frac{2d+1}{p}}\sup_{s\le t\le u}\left|\int_{s}^t\alpha\left(r,Y^{s,x}_{r}\right)dW_r\right|^{1-\frac{2d}{p}}
 		\\&\qquad +
 	N^{\frac{2d+1}{p}}\sup_{s\le t\le u}\left|\int_{s}^{t}\!\int_{U_0}g\left(Y^{s,x}_{r-},r,z\right)\widetilde{N}_p\left(dr,dz\right)\right|^{1-\frac{2d}{p}}
 	+\sup_{s\le t\le u}
 	\left|\int_{s}^{t}b\left(r,Y^{s,x}_{r}\right)dr\right|^{1+\frac{1}{p}}
 		\\&\qquad +\sup_{s\le t\le u}
 	\left|\int_{s}^{t}\alpha\left(r,Y^{s,x}_{r}\right)dW_r\right|^{1+\frac{1}{p}}
 	+
 	\sup_{s\le t\le u}
 	\left|\int_{s}^{t}\!\int_{U_0}g\left(Y^{s,x}_{r-},r,z\right)\widetilde{N}_p\left(dr,dz\right)\right|^{1+\frac{1}{p}} 
 	\bigg\}
 	+
 	S\left(N\right) 
 	\eqqcolon \mathbf{\upperRomannumeral{3}}_1,
 \end{align*}
where we denote by $\mathcal{S}\left(N\right),\,N\ge 1$, the quantity
\begin{align}\label{S_def}
\notag\mathcal{S}\left(N\right)&=\sum_{i,j\in A,\,i\neq j}\sup_{\left|x\right|\le N}\sup_{s\le t\le u} \Bigg\{
		\\\notag&\qquad
		\left|\int_{s}^{t}b\left(r,Y^{s,x}_{r}\right)dr\right|^i
		\left|\int_{s}^{t}\!\int_{U_0}g\left(Y^{s,x}_{r-},r,z\right)\widetilde{N}_p\left(dr,dz\right)\right|^j
		+\left|\int_{s}^{t}b\left(r,Y^{s,x}_{r}\right)dr\right|^i\left|\int_{s}^{t}\alpha\left(r,Y^{s,x}_{r}\right)dW_r\right|^j
		\\&\qquad+
		\left|\int_{s}^{t}\alpha\left(r,Y^{s,x}_{r}\right)dW_r\right|^i
		\left|\int_{s}^{t}\!\int_{U_0}g\left(Y^{s,x}_{r-},r,z\right)\widetilde{N}_p\left(dr,dz\right)\right|^j
		\Bigg\}.
\end{align}
Furthermore, notice that $\mathbf{\upperRomannumeral{1}}_1\le \mathbf{\upperRomannumeral{3}}_1$. 
 Therefore, going back to \eqref{add1} we have 
\begin{align}\label{4_1}
\notag
\sup_{\left|x\right|\le N}&\sup_{0\le t\le T}\left|Y_t^{s,x}-Y_t^{u,x}\right|
\le 
 \left(1+c_1\,U_{u,p}\right)\mathbf{\upperRomannumeral{3}}_1
\le \left(1+c_1\,U_{u,p}\right)\\&\qquad\notag
\times\Bigg[ N^{\frac{2d+1}{p}}
\sup_{\left|x\right|\le N}\sup_{s\le t\le u} 
\left|\int_{s}^{t}b\left(r,Y^{s,x}_{r}\right)dr\right|^{1-\frac{2d}{p}}
+
N^{\frac{2d+1}{p}}\sup_{\left|x\right|\le N}\sup_{s\le t\le u} \left|\int_{s}^{t}\alpha\left(r,Y^{s,x}_{r}\right)dW_r\right|^{1-\frac{2d}{p}}
\\&\qquad+\notag
N^{\frac{2d+1}{p}}\sup_{\left|x\right|\le N}\sup_{s\le t\le u} 
\left|\int_{s}^{t}\!\int_{U_0}g\left(Y^{s,x}_{r-},r,z\right)\widetilde{N}_p\left(dr,dz\right)\right|^{1-\frac{2d}{p}}
\\&\qquad\notag
+\sup_{\left|x\right|\le N}\sup_{s\le t\le u} \left|\int_{s}^{t}b\left(r,Y^{s,x}_{r}\right)dr\right|^{1+\frac{1}{p}}
+\sup_{\left|x\right|\le N}\sup_{s\le t\le u} \left|\int_{s}^{t}\alpha\left(r,Y^{s,x}_{r}\right)dW_r\right|^{1+\frac{1}{p}}
\\&\qquad+\sup_{\left|x\right|\le N}\sup_{s\le t\le u} \left|\int_{s}^{t}\!\int_{U_0}g\left(Y^{s,x}_{r-},r,z\right)\widetilde{N}_p\left(dr,dz\right)\right|^{1+\frac{1}{p}}+\mathcal{S}\left(N\right)
\Bigg].
\end{align}
Since we can pick $p\ge2d+1$ arbitrarily large, we consider $p^{-1}{\left(2d+1\right)}<1-{2d}{p}^{-1}$. Moreover, note that given two numbers $a,b\ge0$, 
\begin{equation}\label{min}
\left(ab\right)\wedge1\le\left(1+a\right)\left(b\wedge1\right)\quad \text{and} \quad\left(a+b\right)\wedge1\le a\wedge1+b\wedge1.
\end{equation}
With these considerations, using the inequality  $x/(1+x)\le 1\wedge x,\,x\ge0$, we simplify the expression in \eqref{4_1} to deduce that 
\begin{align*}
&\frac{\sup_{\left|x\right|\le N}\norm{Y^{s,x}-Y^{u,x}}_0}{1+\sup_{\left|x\right|\le N}\norm{Y^{s,x}-Y^{u,x}}_0}
\le
\sup_{\left|x\right|\le N}\sup_{0\le t\le T}\left|Y_t^{s,x}-Y_t^{u,x}\right|\wedge 1
\\&\qquad\le \left(2+c_1\,U_{u,p}\right)\left(N^{\frac{2d+1}{p}}+2\right)
\Bigg[\mathcal{S}\left(N\right)\wedge 1+
\sup_{\left|x\right|\le N}\sup_{s\le t\le u}
\left(\left|\int_{s}^{t}\!\int_{U_0}g\left(Y^{s,x}_{r-},r,z\right)\widetilde{N}_p\left(dr,dz\right)\right|\wedge1\right)^{\frac{2d+1}{p}}
\\&\qquad+
\sup_{\left|x\right|\le N}\sup_{s\le t\le u}\left(\left|\int_{s}^{t}b\left(r,Y^{s,x}_{r}\right)dr\right|\wedge1\right)^{\frac{2d+1}{p}}
+
\sup_{\left|x\right|\le N}\sup_{s\le t\le u} \left(\left|\int_{s}^{t}\alpha\left(r,Y^{s,x}_{r}\right)dW_r\right|\wedge 1\right)^{\frac{2d+1}{p}}
\Bigg].
\end{align*}
Multiplying by $2^{-N}$ and then summing over $N$, by the definition in \eqref{d0_un} we infer that
\begin{multline}\label{4_2}
d_0^{lu}\left(Y_s,Y_u\right)
\le 
2\left(2+c_1\,U_{u,p}\right)\sum_{N=1}^{\infty}\frac{N+1}{2^N}
\Bigg[\mathcal{S}\left(N\right)\wedge1 + \sup_{\left|x\right|\le N}\sup_{s\le t\le u} \left(\left|\int_{s}^{t}\!\int_{U_0}g\left(Y^{s,x}_{r-},r,z\right)\widetilde{N}_p\left(dr,dz\right)\right|\wedge1\right)^{\!\!\frac{2d+1}{p}}
\\
+\sup_{\left|x\right|\le N}\sup_{s\le t\le u}\left(\left|\int_{s}^{t}b\left(r,Y^{s,x}_{r}\right)dr\right|\wedge 1\right)^{\frac{2d+1}{p}}
+
\sup_{\left|x\right|\le N}\sup_{s\le t\le u}\left(\left|\int_{s}^{t}\alpha\left(r,Y^{s,x}_{r}\right)dW_r\right|\wedge1\right)^{\frac{2d+1}{p}}
\Bigg],
\end{multline}
which holds in $\Omega_{s,u}$. Now we want to split the series in \eqref{4_2}. To do this, we first notice that  the function $2^{-x}\left(x+1\right)$ is strictly decreasing in $\left[1,\infty\right)$, hence we can estimate
\begin{equation*}
\sum_{N=\widebar N+1}^{\infty} \frac{N+1}{2^N}
\le \int_{\widebar{N}}^{\infty}\frac{x+1}{2^x}\,dx=
\frac{1}{\left(\log2\right)^2}\frac{\left(\widebar{N}+1\right)\log 2+1}{2^{\widebar{N}}},\quad \widebar{N}\ge1.
\end{equation*}
Therefore, for any $\gamma>0$ --a new leverage parameter \emph{which has to be fixed}-- there exists a constant $c_\gamma>0$ such that 
$
\sum_{N=\widebar N+1}^{\infty} {(N+1)}{2^{-N}}\le c_\gamma {(\widebar{N}+1)^{-\gamma}},\, \widebar{N}\ge1.
$
\\Secondly, we introduce another parameter $\sigma>0$ --once again, \emph{to be determined}-- and set $\widebar{N}=\widebar{N}\left(\rho\right)=\left[\rho^{-\sigma}\right]$ (recall that $\rho=v-s<1$). Note that 
$
	\sum_{N=\left[\rho^{-\sigma}\right] +1}^{\infty} (N+1){2^{-N}}\le c_\gamma (\left[\rho^{-\sigma}\right]+1)^{-\gamma}\le c_\gamma\rho^{\sigma\gamma},
$
and that $\sum_{N=1}^{\infty}(N+1)2^{-N}=3$. Hence from \eqref{4_2} we write
\begin{multline}\label{4_3}
d_0^{lu}\left(Y_s,Y_u\right)\le 2\left(2+c_1\,U_{u,p}\right)\Bigg[3\Bigg(
\mathcal{S}\left(\rho^{-\sigma}\right)\wedge 1+
\sup_{\left|x\right|\le \rho^{-\sigma}}\sup_{s\le t\le u} \left|\int_{s}^{t}\int_{U_0}g\left(Y^{s,x}_{r-},r,z\right)\widetilde{N}_p\left(dr,dz\right)\right|^{\frac{2d+1}{p}}
\\
+ 
\sup_{\left|x\right|\le \rho^{-\sigma}}\sup_{s\le t\le u}\left|\int_{s}^{t}b\left(r,Y^{s,x}_{r}\right)dr\right|^{\frac{2d+1}{p}}
+
\sup_{\left|x\right|\le \rho^{-\sigma}}\sup_{s\le t\le u}\left|\int_{s}^{t}\alpha\left(r,Y^{s,x}_{r}\right)dW_r\right|^{\frac{2d+1}{p}}\Bigg)
+4\,c_\gamma\,\rho^{\sigma\gamma}
\Bigg].
\end{multline}
At this point, we revert to the definition of $\mathcal{S}\left(\cdot\right)$ in \eqref{S_def}. By \eqref{min}, 
\begin{align*}
&\Big(\left|\int_{s}^{t}b\left(r,Y^{s,x}_{r}\right)dr\right|^{\frac{2d+1}{p}}
\left|\int_{s}^{t}\!\int_{U_0}g\left(Y^{s,x}_{r-},r,z\right)\widetilde{N}_p\left(dr,dz\right)\right|^{1-\frac{2d}{p}}\Big)\wedge1
\\&\hspace{4em}
\le\Big(1+\left|\int_{s}^{t}b\left(r,Y^{s,x}_{r}\right)dr\right|^{\frac{2d+1}{p}}\Big) \Big(\left|\int_{s}^{t}\!\int_{U_0}g\left(Y^{s,x}_{r-},r,z\right)\widetilde{N}_p\left(dr,dz\right)\right|\wedge1\Big)^{1-\frac{2d}{p}}\\
&\hspace{4em}
\le 
\left|\int_{s}^{t}\!\int_{U_0}g\left(Y^{s,x}_{r-},r,z\right)\widetilde{N}_p\left(dr,dz\right)\right|^{\frac{2d+1}{p}}
+
\left|\int_{s}^{t}b\left(r,Y^{s,x}_{r}\right)dr\right|^{\frac{2d+1}{p}},\quad x\in\mathbb{R}^d,\,t\in\left[s,u\right].
\end{align*}
Repeating the previous argument we find an upper bound for $\mathcal{S}\left(\rho^{-\sigma}\right)\wedge 1$, namely
\begin{multline*}
	\mathcal{S}\left(\rho^{-\sigma}\right)\wedge 1\le 4\Bigg(\sup_{\left|x\right|\le \rho^{-\sigma}}\sup_{s\le t\le u} \left|\int_{s}^{t}\int_{U_0}g\left(Y^{s,x}_{r-},r,z\right)\widetilde{N}_p\left(dr,dz\right)\right|^{\frac{2d+1}{p}}
	\\
	+
	\sup_{\left|x\right|\le \rho^{-\sigma}}\sup_{s\le t\le u}\left|\int_{s}^{t}b\left(r,Y^{s,x}_{r}\right)dr\right|^{\frac{2d+1}{p}}
	+
	\sup_{\left|x\right|\le \rho^{-\sigma}}\sup_{s\le t\le u}
	\left|\int_{s}^{t}\alpha\left(r,Y^{s,x}_{r}\right)dW_r\right|^{\frac{2d+1}{p}}\Bigg).
\end{multline*}
Therefore \eqref{4_3} becomes
\begin{multline*}
d_0^{lu}\left(Y_s,Y_u\right)\le 2\left(2+c_1\,U_{u,p}\right)\Bigg[
15\Bigg(
\sup_{\left|x\right|\le \rho^{-\sigma}}\sup_{s\le t\le u} \left|\int_{s}^{t}\int_{U_0}g\left(Y^{s,x}_{r-},r,z\right)\widetilde{N}_p\left(dr,dz\right)\right|^{\frac{2d+1}{p}}
\\
+
\sup_{\left|x\right|\le \rho^{-\sigma}}\sup_{s\le t\le u}\left|\int_{s}^{t}b\left(r,Y^{s,x}_{r}\right)dr\right|^{\frac{2d+1}{p}}
+
\sup_{\left|x\right|\le \rho^{-\sigma}}\sup_{s\le t\le u}
\left|\int_{s}^{t}\alpha\left(r,Y^{s,x}_{r}\right)dW_r\right|^{\frac{2d+1}{p}}\Bigg)
+4\,c_\gamma\,\rho^{\sigma\gamma}
\Bigg].
\end{multline*}
Analogously, in an a.s. event $\Omega_{u,v}$ one has, using also that $d^{lu}_0(f,g)\le 1$ for every $f,g\in \mathcal{C}_0,$
\begin{multline*}
d_0^{lu}\left(Y_u,Y_v\right)\le30
\left(2+c_1\,U_{v,p}\right)\Bigg[\Bigg(
\sup_{\left|x\right|\le \rho^{-\sigma}}\sup_{u\le t\le v} \left|\int_{u}^{t}\int_{U_0}g\left(Y^{u,x}_{r-},r,z\right)\widetilde{N}_p\left(dr,dz\right)\right|^{\frac{2d+1}{p}}
\\
+
\sup_{\left|x\right|\le \rho^{-\sigma}}\sup_{u\le t\le v}\left|\int_{u}^{t}b\left(r,Y^{u,x}_{r}\right)dr\right|^{\frac{2d+1}{p}}
+
\sup_{\left|x\right|\le \rho^{-\sigma}}\sup_{u\le t\le v}\left|\int_{u}^{t}\alpha\left(r,Y^{u,x}_{r}\right)dW_r\right|^{\frac{2d+1}{p}}
+c_\gamma\,\rho^{\sigma\gamma}
\Bigg)\wedge1\Bigg].
\end{multline*}
If we multiply the two previous expressions, raise both sides to a power $q\in\left(1/2,p\right)$ and take expectations, then 
\begin{align*}
\notag
\mathbb{E}&\left[d^{lu}_0\left(Y_s,Y_u\right)^q\cdot d_0^{lu}\left(Y_u,Y_v\right)^q\right]
\le
30^{2q}\left(4^{q-1}\vee1\right)^2
\mathbb{E}\Bigg[
\left(2+c_1\,U_{u,p}\right)^q \left(2+c_1\,U_{v,p}\right)^q\\\notag&\quad
\times \Bigg\{
\sup_{\left|x\right|\le \rho^{-\sigma}}\sup_{s\le t\le u} \left|\int_{s}^{t}\int_{U_0}g\left(Y^{s,x}_{r-},r,z\right)\widetilde{N}_p\left(dr,dz\right)\right|^{q\frac{2d+1}{p}}
\\\notag&\quad
+
\sup_{\left|x\right|\le \rho^{-\sigma}}\sup_{s\le t\le u}\left|\int_{s}^{t}b\left(r,Y^{s,x}_{r}\right)dr\right|^{q\frac{2d+1}{p}}
+
\sup_{\left|x\right|\le \rho^{-\sigma}}\sup_{s\le t\le u}\left|\int_{s}^{t}\alpha\left(r,Y^{s,x}_{r}\right)dW_r\right|^{q\frac{2d+1}{p}}
+c_{\gamma}^q\,\rho^{\sigma\gamma q}
\Bigg\}\\\notag&\quad
\times\Bigg(\Bigg\{
\sup_{\left|x\right|\le \rho^{-\sigma}}\sup_{u\le t\le v} \left|\int_{u}^{t}\int_{U_0}g\left(Y^{u,x}_{r-},r,z\right)\widetilde{N}_p\left(dr,dz\right)\right|^{q\frac{2d+1}{p}}
\\&\quad
+
\sup_{\left|x\right|\le \rho^{-\sigma}}\sup_{u\le t\le v}\left|\int_{u}^{t}b\left(r,Y^{u,x}_{r}\right)dr\right|^{q\frac{2d+1}{p}}
+
\sup_{\left|x\right|\le \rho^{-\sigma}}\sup_{u\le t\le v}\left|\int_{u}^{t}\alpha\left(r,Y^{u,x}_{r}\right)dW_r\right|^{q\frac{2d+1}{p}}
+c_{\gamma}^q\,\rho^{\sigma\gamma q}
\Bigg\}\wedge 1\Bigg)
\Bigg].
\end{align*}
 Notice that the definition in Corollary \ref{cor5} and Lemma \ref{l1} yield the $\eta-$integrability of the random variables $U_{u,p}$ and $U_{v,p}$ for every $\eta>1$. We fix $\eta=\frac{5}{4}$  and apply H\"older's inequality with exponents $\eta$ and $\eta'=\eta(\eta-1)^{-1}=5$ to deduce that
\begin{align}\label{4_miss1}
\notag
&\mathbb{E}\left[d^{lu}_0\left(Y_s,Y_u\right)^q\cdot d_0^{lu}\left(Y_u,Y_v\right)^q\right]
\le  30^{2q}4^{2\frac{\eta-1}{\eta}}\left(4^{q-1}\vee1\right)^2
\left(\mathbb{E} \left[
\left(2+c_1\,U_{u,p}\right)^{5 q} \left(2+c_1\,U_{v,p}\right)^{5 q} \right]\right)^{1/5}
\\&\quad \quad
\notag
\times \Bigg(\mathbb{E}\Bigg[\Bigg\{
c^{\eta q}_{\gamma}\,\rho^{\eta\sigma\gamma q}
+
\sup_{\left|x\right|\le \rho^{-\sigma}}\sup_{s\le t\le u} \left|\int_{s}^{t}\int_{U_0}g\left(Y^{s,x}_{r-},r,z\right)\widetilde{N}_p\left(dr,dz\right)\right|^{\eta q\frac{2d+1}{p}}
\\\notag&\quad \quad 
+
\sup_{\left|x\right|\le \rho^{-\sigma}}\sup_{s\le t\le u}\left|\int_{s}^{t}b\left(r,Y^{s,x}_{r}\right)dr\right|^{\eta q\frac{2d+1}{p}}
+
\sup_{\left|x\right|\le \rho^{-\sigma}}\sup_{s\le t\le u}\left|\int_{s}^{t}\alpha\left(r,Y^{s,x}_{r}\right)dW_r\right|^{\eta q\frac{2d+1}{p}} \Bigg\}
\\\notag&\quad \quad
\times\Bigg( \Bigg \{c^{\eta q}_{\gamma}\rho^{\eta \sigma\gamma q}+
\sup_{\left|x\right|\le \rho^{-\sigma}}\sup_{u\le t\le v} \left|\int_{u}^{t}\int_{U_0}g\left(Y^{u,x}_{r-},r,z\right)\widetilde{N}_p\left(dr,dz\right)\right|^{\eta q\frac{2d+1}{p}}
\\&\quad \quad
+
\sup_{\left|x\right|\le \rho^{-\sigma}}\sup_{u\le t\le v}\left|\int_{u}^{t}b\left(r,Y^{u,x}_{r}\right)dr\right|^{\eta q\frac{2d+1}{p}}
+
\sup_{\left|x\right|\le \rho^{-\sigma}}\sup_{u\le t\le v}\left|\int_{u}^{t}\alpha\left(r,Y^{u,x}_{r}\right)dW_r\right|^{\eta q\frac{2d+1}{p}}
 \Bigg\}\wedge1\Bigg)\Bigg]\Bigg)^{\frac{1}{\eta}}. 
\end{align}
For every couple of time indeces $0\le t_1\le t_2\le T$, let us denote by 
	\begin{equation}\label{filtr1}
	\mathcal{F}_{t_1,t_2}^{W,N_p}
	=
	\sigma\left(\left\{W_t-W_{t_1},\,N_p\left(\left(t_1,t\right]\times E\right),\,t\in \left[t_1,t_2\right],\,E\in\mathcal{U}\right\}\cup \mathcal{N}\right),
    \end{equation}
	where $\mathcal{N}$ is the family of negligible events in $(\Omega,\mathcal{F},\mathbb{P})$: we set $\mathbb{F}^{W,N_p}_{t_1}=(\mathcal{F}^{W,N_p}_{t_1,t})_{t\in[t_1,T]}$. In particular, $\mathbb{F}_0^{W,N_p}=(\mathcal{F}^{W,N_p}_{0,t})_{t\in[0,T]}$ is the augmented filtration generated by $W$ and $N_p$. We are considering the SDE \eqref{SDE} with deterministic coefficients, which are obviously predictable with respect to $\mathbb{F}_{t_1}^{W,N_p}$, for every $t_1\in[0,T)$, once restricted to the interval $[t_1, T]$. Hence we can invoke \cite[Theorem $117$]{situ} (see also \cite[Theorem $2$]{BLP}) to claim that, for every $t_1\in [0,T)$ and $x\in\mathbb{R^d}$,
	\begin{equation}\label{strong}
		Y^{t_1,x}\quad \text{is a strong solution of (\ref{SDE})}.
	\end{equation}
	This means that $Y^{t_1,x}$ is $\mathbb{F}^{W,N_p}_{t_1}-$adapted (in fact, it is also an $\mathbb{F}^{W,N_p}_{t_1}-$Markov process).  As a consequence, for all $x\in\mathbb{R}^d, \,i=1,2,3,$ and $t\in\left[u,v\right]$, the random variable $Z^{u,x}_{i,t}$ (defined at the end of {\color{black}Subsection }\ref{pre_flow}) is $\mathcal{F}^{W,N_p}_{u,v}-$measurable. Recalling the continuity of the $(\mathcal{D}_0,\mathcal{D})-$valued random fields $Z^{u,\cdot}_{i}$ and using the fact that countable $\sup$ of measurable functions is measurable, we deduce that the random variables  $\sup_{|x|\le \rho^{-\sigma}}\sup_{u\le t \le v}|Z^{u,x}_{i,t}|,\,i=1,2,3,$ appearing in \eqref{4_miss1} are $\mathcal{F}^{W,N_p}_{u,v}-$measurable. It then follows that they are independent of  $\mathcal{F}^{W,N_p}_{s,u}$. Indeed, since $W$ and $N_p$ are mutually independent (see \cite[Theorem $6.3$, Chapter \upperRomannumeral{2}]{IW}, or \cite[Proposition $2.6$]{KU04}), a standard argument based on \emph{Dynkin's theorem} ensures that $\mathcal{F}^{W,N_p}_{t_1,t_2}$ and $\mathcal{F}^{W,N_p}_{t_2,t_3}$ are independent, for any $0\le t_1<t_2<t_3\le T$.  Therefore
\begin{align}\label{Mark}
	\notag&\mathbb{E}\left[\Bigg( \Bigg \{c^{\eta q}_{\gamma}\rho^{\eta \sigma\gamma q}+
	\sup_{\left|x\right|\le \rho^{-\sigma}}\sup_{u\le t\le v} \left|\int_{u}^{t}\int_{U_0}g\left(Y^{u,x}_{r-},r,z\right)\widetilde{N}_p\left(dr,dz\right)\right|^{\eta q\frac{2d+1}{p}}
	\right.\\&\notag\left.
	\qquad +
	\sup_{\left|x\right|\le \rho^{-\sigma}}\sup_{u\le t\le v}\left|\int_{u}^{t}b\left(r,Y^{u,x}_{r}\right)dr\right|^{\eta q\frac{2d+1}{p}}
	+
	\sup_{\left|x\right|\le \rho^{-\sigma}}\sup_{u\le t\le v}\left|\int_{u}^{t}\alpha\left(r,Y^{u,x}_{r}\right)dW_r\right|^{\eta q\frac{2d+1}{p}}
	\Bigg\}\wedge1\Bigg)\Bigg| \mathcal{F}^{W,N_p}_{s,u}\right]
	\\&\notag\quad =
	\mathbb{E}\left[\Bigg( \Bigg \{c^{\eta q}_{\gamma}\rho^{\eta \sigma\gamma q}+
	\sup_{\left|x\right|\le \rho^{-\sigma}}\sup_{u\le t\le v} \left|\int_{u}^{t}\int_{U_0}g\left(Y^{u,x}_{r-},r,z\right)\widetilde{N}_p\left(dr,dz\right)\right|^{\eta q\frac{2d+1}{p}}
	\right.\\&\left.
	\qquad +
	\sup_{\left|x\right|\le \rho^{-\sigma}}\sup_{u\le t\le v}\left|\int_{u}^{t}b\left(r,Y^{u,x}_{r}\right)dr\right|^{\eta q\frac{2d+1}{p}}
	+
	\sup_{\left|x\right|\le \rho^{-\sigma}}\sup_{u\le t\le v}\left|\int_{u}^{t}\alpha\left(r,Y^{u,x}_{r}\right)dW_r\right|^{\eta q\frac{2d+1}{p}}
	\Bigg\}\wedge1\Bigg)\right].
\end{align}
 From now on, we denote by $\tilde{p}=\eta q{(2d+1)}{p^{-1}}$ and enumerate positive constants implicitly assuming their dependence on the parameters $\gamma,\eta, p,q,d,m,T,\{K_r,\,r\ge2 \}$. By virtue of Corollary \ref{cor5} and \cite[Equation ($3.6$)]{KU04}, for every $\tilde{q}>2d,$
\begin{multline}\label{suY}
	\mathbb{E}\Big[\sup_{|x|\le \rho^{-\sigma}}\sup_{s\le t\le u}\Big|Y_t^{s,x}\Big|^{\tilde{q}}
	\Big]
	\le 	2^{\tilde{q}-1}
	\Big(\mathbb{E}\Big[\sup_{s\le t\le u}\Big|Y^{s,0}_t\Big|^{\tilde{q}}\Big]
	+\mathbb{E}\Big[
	\sup_{|x|\le \rho^{-\sigma}}\sup_{s\le t\le u}\Big|Y_t^{s,x}-Y_t^{s,0}\Big|^{\tilde{q}}\Big]
	\Big)
	\\ 
	\le
	c_2\left(1+\rho^{-\sigma\left(\tilde{q}+1\right)}\right)\le 2c_2\rho^{-\sigma\left(\tilde{q}+1\right)}.
\end{multline}
Taking $\tilde{p}>d$, \eqref{suY} with $\tilde{q}=2\tilde{p}$, \eqref{lg1} and H\"older's inequality entail
\begin{align}\label{int_b}
	\begin{split}
		\mathbb{E}\left[\sup_{\left|x\right|\le \rho^{-\sigma}}\sup_{s\le t\le u}\left|\int_{s}^{t}b\left(r,Y^{s,x}_{r}\right)dr\right|^{2\tilde{p}}
		\right]
		\le c_3\,\rho^{2\tilde{p}-\sigma \left(2\tilde{p}+1\right)}.
	\end{split}
\end{align}

In order to estimate the expected value in the right--hand side of \eqref{4_miss1}, we also want to  apply Proposition~\ref{not_sep}  to the continuous, $(\mathcal{D}_0,\mathcal{D})-$valued random fields 
\begin{align*}
Z_1^{\xi,x}=\int_{\xi}^{\cdot}\alpha\left(r,Y^{\xi,x}_{r}\right)d W_r,\qquad
 Z_2^{\xi,x}= \int_{\xi}^{\cdot}\!\int_{U_0}g\left(Y^{\xi,x}_{r-},r,z\right)\widetilde{N}_p\left(dr,dz\right),\quad x\in\mathbb{R}^d, 
\end{align*}
where $\xi=s,u$.
We only show the case $\xi=s$, being $\xi=u$ analogous. We start from $Z_1^{s,\cdot}$, writing
\begin{multline*}
	\notag \mathbb{E}\left[\sup_{\left|x\right|\le \rho^{-\sigma}}\sup_{s\le t\le u} \left|\int_{s}^{t}\alpha\left(r,Y^{s,x}_{r}\right)dW_r\right|^{2\tilde{p}}\right]\le c_4\Bigg(
	\mathbb{E}\left[\sup_{s\le t\le u} \left|\int_{s}^{t}\alpha\left(r,Y^{s,0}_{r}\right)dW_r\right|^{2\tilde{p}}\right]
	\\
	+\mathbb{E}\left[\sup_{\left|x\right|\le \rho^{-\sigma}}\sup_{s\le t\le u} \left|\int_{s}^{t}\left(\alpha\left(r,Y^{s,x}_{r}\right)-\alpha\left(r,Y^{s,0}_{r}\right)\right)dW_r\right|^{2\tilde{p}}\right]\Bigg)\eqqcolon c_4\left(\mathbf{\upperRomannumeral{1}}_2+\mathbf{\upperRomannumeral{2}}_2\right).
\end{multline*}
Regarding $\mathbf{\upperRomannumeral{1}}_2$, by \eqref{lg1},  the Burkholder--Davis--Gundy inequality and \cite[Equation ($3.6$)]{KU04} we obtain $\mathbf{\upperRomannumeral{1}}_2\le c_5\rho^{\tilde{p}}$.
As for $\mathbf{\upperRomannumeral{2}}_2$, since $\tilde{p}>d$ we can apply Proposition \ref{not_sep}, which implies that, in the whole space  $\Omega$, 
\begin{align*}
	\sup_{\left|x\right|\le \rho^{-\sigma}}\sup_{s\le t\le u} \left|\int_{s}^{t}\left(\alpha\left(r,Y^{s,x}_{r}\right)-\alpha\left(r,Y^{s,0}_{r}\right)\right)dW_r\right|\
	\le c_6\widetilde{U}^{(1)}_{s,u,\tilde{p}}\rho^{-\sigma\left(1+\frac{1}{2\tilde{p}}\right)}.
\end{align*}
Here $\widetilde{U}^{(1)}_{s,u,\tilde{p}}$ is the $2\tilde{p}-$integrable random variable
	\[
\widetilde{U}^{(1)}_{s,u,\gamma}(\omega)=
\Big(\int_{{\mathbb R}^d}\int_{{\mathbb R}^d}\Big(\frac{\sup_{s\le t \le u}{\left|Z_{1,t}^{s,x}(\omega)-Z_{1,t}^{s,y}(\omega)\right|}}{\left|x-y\right|}\Big)^{2\tilde{p}} f_{}\left(\left|x\right|\right)f_{}\left(\left|y\right|\right)dx\,dy\Big)^{\frac{1}{2\tilde{p}}},\quad \omega\in\Omega,
\]   
 which satisfies, by \eqref{lip1}, Lemma \ref{l1} and the Burkholder--Davis--Gundy inequality,
\begin{equation*}
	\mathbb{E}\left[\left(\widetilde{U}^{(1)}_{s,u,\tilde{p}}\right)^{2\tilde{p}}\right]\le c_7\,\rho^{\tilde{p}}.
\end{equation*}
Thus,
\begin{equation}\label{int_alph}
	\mathbb{E}\left[\sup_{\left|x\right|\le \rho^{-\sigma}}\sup_{s\le t\le u} \left|\int_{s}^{t}\alpha\left(r,Y^{s,x}_{r}\right)dW_r\right|^{2\tilde{p}}\right]
	\le c_8 \left[
	\rho^{\tilde{p}}+\rho^{\tilde{p}-\sigma \left(2\tilde{p}+1\right)}\right]\\
	\le 2\,c_8\,\rho^{\tilde{p}-\sigma \left(2\tilde{p}+1\right)}.
\end{equation}     
Moving on to $Z^{s,\cdot}_2$, we can argue as in \eqref{c_1v} to conclude that 
\begin{equation}\label{int_g_2}
	\mathbb{E}\left[\sup_{\left|x\right|\le \rho^{-\sigma}}\sup_{s\le t\le u} \left|\int_{s}^{t}\int_{U_0}g\left(Y^{s,x}_{r-},r,z\right)\widetilde{N}_p\left(dr,dz\right)\right|^{2\tilde{p}}\right]
	\le c_{9} \left(\rho+\rho^{1-\sigma\left(2\tilde{p}+1\right)}\right)\le 
	 2\,c_{9}\,\rho^{1-\sigma\left(2\tilde{p}+1\right)}. 
\end{equation} 
Analogously, if we require $\tilde{p}>2d,$ then 
\begin{equation}\label{int_g_1}
	\mathbb{E}\left[\sup_{\left|x\right|\le \rho^{-\sigma}}\sup_{s\le t\le u} \left|\int_{s}^{t}\int_{U_0}g\left(Y^{s,x}_{r-},r,z\right)\widetilde{N}_p\left(dr,dz\right)\right|^{\tilde{p}}\right]
\le 
	c_{10}\,\rho^{1-\sigma\left(\tilde{p}+1\right)}. 
\end{equation}
Going back to \eqref{4_miss1}, by \eqref{Mark} and the law of iterated expectations with respect to $\mathcal{F}^{W,N_p}_{s,u}$ we compute
\begin{align}\label{daje}
	\notag
	&\mathbb{E}\left[d^{lu}_0\left(Y_s,Y_u\right)^q\cdot d_0^{lu}\left(Y_u,Y_v\right)^q\right]
	\le  c_{11}\Bigg\{\rho^{\sigma\gamma q}\\\notag
	&\qquad 
	+\left(\mathbb{E}\left[
		\sup_{\left|x\right|\le \rho^{-\sigma}}\sup_{s\le t\le u}\left|\int_{s}^{t}b\left(r,Y^{s,x}_{r}\right)dr\right|^{\tilde{p}}
+
	\sup_{\left|x\right|\le \rho^{-\sigma}}\sup_{s\le t\le u}\left|\int_{s}^{t}\alpha\left(r,Y^{s,x}_{r}\right)dW_r\right|^{\tilde{p}}\right]\right)^{\frac{1}{\eta}}
\\&
\qquad \notag+\Bigg(\mathbb{E}\left[\sup_{\left|x\right|\le \rho^{-\sigma}}\sup_{s\le t\le u} \left|\int_{s}^{t}\int_{U_0}g\left(Y^{s,x}_{r-},r,z\right)\widetilde{N}_p\left(dr,dz\right)\right|^{\tilde{p}}\right]
\Bigg( 
\rho^{\eta \sigma\gamma q}+\mathbb{E}\Bigg[
	\sup_{\left|x\right|\le \rho^{-\sigma}}\sup_{u\le t\le v}\left|\int_{u}^{t}b\left(r,Y^{u,x}_{r}\right)dr\right|^{\tilde{p}}
	\\\notag
	&\qquad +\sup_{\left|x\right|\le \rho^{-\sigma}}\sup_{u\le t\le v} \left|\int_{u}^{t}\int_{U_0}g\left(Y^{u,x}_{r-},r,z\right)\widetilde{N}_p\left(dr,dz\right)\right|^{\tilde{p}}+
	\sup_{\left|x\right|\le \rho^{-\sigma}}\sup_{u\le t\le v}\left|\int_{u}^{t}\alpha\left(r,Y^{u,x}_{r}\right)dW_r\right|^{\tilde{p}}
	\Bigg)\Bigg]\Bigg)^{\frac{1}{\eta}}\Bigg\}\notag
	\\&\qquad \eqqcolon 
	c_{11}\left\{
	\rho^{\sigma \gamma q}+ \mathbf{\upperRomannumeral{1}}_3+\mathbf{\upperRomannumeral{2}}_3
	\right\}.
\end{align}
To perform this passage, we observe that the argument of the second expected value in \eqref{4_miss1} is a product between a sum $\widebar{S}$, which includes an integral in $\widetilde{N}_p$, and $\min\{S,1\}$, where $S$ is another sum. In \eqref{daje},  $\mathbf{\upperRomannumeral{2}}_3$ is obtained multiplying  the integral with respect to $\widetilde{N}_p$ in $\widebar{S}$ with $S$, while  $\rho^{\sigma \gamma q}$ and $ \mathbf{\upperRomannumeral{1}}_3$ multiplying the remaining terms in $\widebar{S}$ by $1$.

By \eqref{int_b}-\eqref{int_alph} and the Cauchy--Schwarz inequality,
\begin{equation*}
	\mathbf{\upperRomannumeral{1}}_3\le c_{12}\rho^{\frac{\tilde{p}}{2\eta}-\frac{\sigma}{2\eta}\left(2\tilde{p}+1\right)}.
\end{equation*}
The same results together with \eqref{int_g_2}-\eqref{int_g_1} also yield
\begin{align}
{\mathbf{\upperRomannumeral{2}}_3}\notag&\le
c_{13}\rho^{\frac{1}{\eta}-\frac{\sigma}{\eta}\left(\tilde{p}+1\right)}\left(\rho^{\sigma\gamma q}+\rho^{\frac{1}{2\eta}-\frac{\sigma}{2\eta}\left(2\tilde{p}+1\right)}\right).
\end{align}
Combining the two previous  estimates in \eqref{daje} we deduce  that 
	\begin{equation}\label{quasi2}
		\mathbb{E}\left[d^{lu}_0\left(Y_s,Y_u\right)^q\cdot d_0^{lu}\left(Y_u,Y_v\right)^q\right]
	\le 
	c_{14}\left(\rho^{\sigma\gamma q }+\rho^{\frac{\tilde{p}}{2\eta}-\frac{\sigma}{2\eta}\left(2\tilde{p}+1\right)}
	+\rho^{\sigma\gamma q+\frac{1}{\eta}\left(1-\sigma\left(\tilde{p}+1\right)\right)}
	+
	\rho^{\frac{3}{2\eta}-\frac{\sigma}{\eta}\left(2\tilde{p}+\frac{3}{2}\right)}
\right).
	\end{equation}
At this point, it only remains to select appropriate parameters to recover \eqref{goal} from \eqref{quasi2}. Recall that $\eta=\frac{5}{4},$ hence $\frac{3}{2\eta}=\frac{6}{5}>1.$ Collecting the conditions written throughout the lines above, we  pick $\left(p,q,\sigma,\gamma\right)$ according to the following steps.
\begin{enumerate}
	\item $p\in\left(4d+1,\infty\right)$, so that $p^{-1}{\left(2d+1\right)}<1-2dp^{-1}$;
	\item $q\in({p(2d+\frac{1}{2})}{(2d+1)^{-1}},\,p)$, where  the lower bound ensures that $q\left(2d+1\right)p^{-1}>2d+\frac{1}{2}$. In turn, this yields
	\[
		q>1/2,\qquad \tilde{p}=\eta q\frac{2d+1}{p}>2d,\qquad \frac{\tilde{p}}{2\eta}>\frac{3}{2\eta};
	\]
	\item $\sigma\in(0,(8\tilde{p}+6)^{-1})$, i.e., $\sigma$ is so small that
	\begin{equation}\label{r}
		\frac{3}{2\eta}-\frac{\sigma}{\eta}\left(2\tilde{p}+\frac{3}{2}\right)>1.
	\end{equation}
This bound also guarantees that $\sigma\left(\tilde{p}+1\right)<1$;
	\item $\gamma\in\big(3(2\eta \sigma q)^{-1},\infty\big)$, so that $\sigma\gamma q>\frac{3}{2\eta}$.
\end{enumerate}
With the previous prescriptions and noticing that $\frac{2\tilde{p}+1}{2}<2\tilde{p}+\frac{3}{2}$, from \eqref{quasi2} we conclude that
\begin{align}\label{ooo}
	\notag&\mathbb{E}\left[d^{lu}_0\left(Y_s,Y_u\right)^q\cdot d_0^{lu}\left(Y_u,Y_v\right)^q\right]
\le 
c_{15}\rho^{\frac{3}{2\eta}-\frac{\sigma}{\eta}\left(2\tilde{p}+\frac{3}{2}\right)}\eqqcolon
c_{15}\rho^{1+r},\quad \text{where }\\
&\qquad r=	\frac{3}{2\eta}-\frac{\sigma}{\eta}\left(2\tilde{p}+\frac{3}{2}\right)-1
.
\end{align}
Since $r>0$ by \eqref{r}, recalling that $\rho=v-s<1$, we see that Equation \eqref{ooo} reduces to \eqref{goal}. The proof is then complete.
\end{proof}
 Using  the càdlàg version $Z$ of the process $Y$ given by Theorem \ref{cadlag}, we consider
\begin{equation}\label{bis}\begin{aligned}
		&Z_{1,t}^{s,x}=\int_{s}^{t}\alpha\left(r,Z^{s,x}_{r}\right)dW_r,
		\qquad 
		Z_{2,t}^{s,x}=\int_{s}^{t}\!\int_{U_0}g\left(Z^{s,x}_{r-},r,z\right)\widetilde{N}_p\left(dr,dz\right),
		\qquad Z_{3,t}^{s,x}=\int_{s}^{t}b\left(r,Z^{s,x}_{r}\right)dr.
	\end{aligned}
\end{equation}
Thanks  to \eqref{part_1}, for every $i=1,2,3$, the estimate \eqref{4_2} for $d_0^{lu}(Y_{s},Y_{u})$ constitutes an upper bound for $d_0^{lu}(Z_{i,s},Z_{i,u})$, upon substituting  $\big(3+c_1\widetilde{U}^{(i)}_{u,p}\big)$ for $(2+c_1U_{u,p})$ in the right--hand side. Here the random variables $\widetilde{U}^{(i)}_{u,p}$ are defined according to Proposition \ref{not_sep} applied to the
$(\mathcal{D}_0,\mathcal{D})-$valued random fields $Z^{u,\cdot}_i$. Therefore, recalling that   $Z_i=(Z_{i,s})_{s\in[0,T]}$ are stochastically continuous processes with values in $(\mathcal{C}_0,\mathcal{C})$ by Corollary \ref{st_cont_Z}, the same computations as in the proof of Theorem~\ref{cadlag} allow to invoke Corollary \ref{cor_4.2}, which yields the following result.
\begin{corollary}\label{cadlag_Z}
	For any $i=1,2,3$, there exists a càdlàg version of the $(\mathcal{C}_0,\mathcal{C})-$valued process $Z_i=(Z_{i,s})_{s}$. 
\end{corollary}
In order not to complicate the notation, we keep denoting by $Z_i,\,i=1,2,3$, the càdlàg processes given by  Corollary \ref{cadlag_Z}. Without loss of generality, we assume that $Z$ and $Z_i$ are càdlàg in the whole space $\Omega$. \\
At the end of Subsection \ref{pre_flow} (see \eqref{ar_1vera}), we have determined the existence of an a.s. event $\Omega_s$ independent of $x$ where the SDE \eqref{SDE} is satisfied. Now, combining Theorem \ref{cadlag} with Corollary \ref{cadlag_Z}, we can get rid of the dependence of such $\Omega_s$ from the initial time $s$. This is done in the next lemma, where we are also able to establish the flow property \eqref{flow} in an a.s. event independent of the space and time variables. 
\begin{lemma}\label{common}
	There exists an a.s. event $\Omega'$ (independent of $x,\,s$ and $t$) such that
	\begin{align}\label{common_eq}
	\notag	&Z^{s,x}_t=x+\int_{s}^{t}b\left(r,Z^{s,x}_{r}\right)dr+\int_{s}^{t}\alpha\left(r,Z^{s,x}_{r}\right)dW_r+
		\int_{s}^{t}\!\int_{U_0}g\left(Z^{s,x}_{r-},r,z\right)\widetilde{N}_p\left(dr,dz\right),\\&\qquad s\in\left[0,T\right),\,t\in\left[0,T\right],\,x\in\mathbb{R}^d,\,\omega\in\Omega'.
	\end{align}
Furthermore,
\begin{equation}\label{flow_X}
		Z^{s,x}_t\left(\omega\right)=Z^{u,Z^{s,x}_u\left(\omega\right)}_t\left(\omega\right),\quad 0\le s<u< t\le T,\,x\in\mathbb{R}^d,\,\omega\in\Omega'.
\end{equation}
\end{lemma}
\begin{proof}
	By \eqref{ar_1vera}, there exists an a.s. event $\Omega_1$ --independent of $s\in[0,T)$ and $x\in\mathbb{R}^d$-- such that 
	\begin{equation}\label{ar2}
	Z_t^{s,x}\left(\omega\right)=x+	\sum_{i=1}^{3}Z^{s,x}_{i,t}\left(\omega\right),\quad s\in\left[0,T\right)\cap \mathbb{Q},\,t\in\left[0,T\right],\,x\in\mathbb{R}^d,\,\omega\in\Omega_1.
	\end{equation}
	Thanks to the càdlàg property of the $(\mathcal{C}_0,\mathcal{C})-$valued processes $Z$ and $Z_i,\,i=1,2,3,$ a standard approximation argument in $s$ ensures that  \eqref{ar2} holds for every $s\in[0,T)$. Hence \eqref{common_eq} is satisfied  in $\Omega_1$. 

As for the flow property in \eqref{flow_X}, note that by \eqref{flow} in Corollary \ref{cor_flow} there is an a.s. event $\Omega_2$ --independent of $x,\,s,\,u$-- such that
\begin{equation}\label{pre_flow_Q}
	Z^{s,x}_t\left(\omega\right)=Z^{u,Z^{s,x}_u\left(\omega\right)}_t\left(\omega\right),\quad 0\le s<u< t\le T,\, s,u\in\mathbb{Q},\,x\in\mathbb{R}^d,\,\omega\in\Omega_2.
\end{equation}
Fix $x\in\mathbb{R}^d,\,\omega\in\Omega_2$ and $s,u\in\mathbb{R}\setminus\mathbb{Q}$ such that $0\le s<u<T$. Consider $t\in(u,T]$ and a sequence $(s_n)_n\subset [0,u)\cap \mathbb{Q}$ such that $s_n\downarrow s$ as $n\to \infty$: we know that $\lim_{n\to \infty}Z^{s_n,x}_{t}(\omega)=Z^{s,x}_t(\omega)$. Moreover, we take a sequence $(u_n)_n\subset (u,t)\cap \mathbb{Q}$, with $u_n\downarrow u$ as $n\to \infty$, so that $\lim_{n\to \infty}Z^{s_n,x}_{u_n}(\omega)=Z^{s,x}_u(\omega).$ At this point, from  the càdlàg property of the $(\mathcal{C}_0,\mathcal{C})-$valued process $Z$, we deduce that $Z^{u,y}\left(\omega\right)=\mathcal{D}_0\!-\lim_{n\to \infty}Z^{u_n,y}\left(\omega\right)$ locally uniformly in $y\in\mathbb{R}^d$. As a consequence, 
\[
	Z^{u,Z^{s,x}_u\left(\omega\right)}\left(\omega\right)=\mathcal{D}_0\!-\lim_{n\to \infty}Z^{u_n,Z^{s_n,x}_{u_n}\left(\omega\right)}\left(\omega\right).
\]
By \eqref{pre_flow_Q},
$
		Z^{s_n,x}_t\left(\omega\right)=Z^{u_n,Z^{s_n,x}_{u_n}\left(\omega\right)}_t\left(\omega\right)
	,\, n\in\mathbb{N},
$
hence we can pass to the limit as $n\to \infty$ to obtain \eqref{flow_X} in $\Omega_2$.

The a.s. event $\Omega'$ is obtained by setting $\Omega'=\Omega_1\cap \Omega_2$, completing the proof.
\end{proof}
{\color{black} Theorem \ref{cadlag} and Lemma \ref{common} coupled with Lemma \ref{st_cont} show  that \emph{$Z$ is the regular (or sharp) stochastic flow generated by the SDE \eqref{SDE}} (without large jumps)  according to {\color{black} Theorem \ref{cadlag1}}. 

More precisely, $Z$ satisfies Points \ref{comunala}-\ref{flow_1} in  Definition \ref{sharpala} by Lemma \ref{common}, while the fact that $Z$ is a stochastically continuous, càdlàg process with values in $(\mathcal{C}_0,\mathcal{C})$ --which entail Points \ref{regularity}-\ref{constoch} in Definition \ref{sharpala}-- is guaranteed by Theorem \ref{cadlag} and Lemma \ref{st_cont}. 

Combining stochastic continuity and càdlàg property, we also infer that the process {\sl $Z=(Z_s)_{s\in[0,T]}$ has no fixed--time discontinuities,}  meaning that, for every $s\in [0,T]$,
\begin{equation}\label{nftd}
	Z_{s-}(\omega)=Z_s(\omega),\quad \omega\in\Omega_s. 
\end{equation}}
We conclude this part by stating a couple of lemmas discussing further properties of the sharp flow $Z^{s,x}_t$: they will be used in Section \ref{large_sec} while studying the SDE \eqref{SDEgeneral} with large jumps. Their proofs are postponed to Appendix \ref{ap_lemmi}.  The first result regards the joint--measurability.
\begin{lemma}\label{joint_meas}
	For every $\bar{s},\,\bar{t}\in [0,T]$, the mapping $Z\colon \Omega\times [0,\bar{s}]\times \mathbb{R}^d\times [0,{\bar{t}}]\to \mathbb{R}^d$ defined by $Z(\omega,s,x,t)=Z^{s,x}_t(\omega)$ is $\mathcal{F}_{\bar{t}}\otimes \mathcal{B}([0,\bar{s}]\times \mathbb{R}^d\times [0,{\bar{t}}])-$measurable.
\end{lemma}

The second lemma shows that the flow $Z^{s,x}_t$ can be used to construct a solution to \eqref{SDE} when the initial condition $\eta$ is only a measurable random variable. 
\begin{lemma}\label{le_se}
	 Fix $s\in [0,T]$ and let $\eta\in L^0(\mathcal{F}_s)$. Then, denoting by $Z^{s,\eta}$ the process defined by $Z^{s,\eta}_t(\omega)=Z^{s,\eta(\omega)}_t(\omega),\,\omega\in\Omega,\,t\in [0,T]$, there exists an a.s. event $\Omega_{s,\eta}$ where the following equation is satisfied:
	\begin{multline}\label{serve_1}
		Z_t^{s, \eta}=\eta+
		\int_{0}^{t}1_{\{r>s\}}b\left(r, Z_{r}^{s,\eta}\right) dr + 
		\int_{0}^{t}1_{\{r>s\}}\alpha\left(r, Z_{r}^{s,\eta} \right) dW_r\\+
		\int_{0}^{t}\!\int_{U_0}1_{\{r>s\}}g\left(Z_{r-}^{s,\eta},r,z\right)\widetilde{N}_p\left(dr,dz\right),\quad t\in[0,T].
	\end{multline}
{\color{black} In particular, $Z^{s,\eta}$ is the pathwise unique solution of \eqref{serve_1}.}
 Further, for every $\omega\in\Omega_{s,\eta},$  $t\in[0,T]$ (cf. \eqref{bis}$)$,
	\begin{equation}\label{dep_Zmeas}\begin{aligned}
		&\left(\int_{0}^{t}1_{\{r>s\}}\alpha\left(r,Z^{s,\eta}_{r}\right)dW_r\right)\left(\omega\right)=Z_{1,t}^{s,\eta\left(\omega\right)}\left(\omega\right), 
		\qquad \int_{0}^{t}1_{\{r>s\}}b\left(r,Z^{s,\eta}_{r}\left(\omega\right)\right)dr=Z_{3,t}^{s,\eta\left(\omega\right)}\left(\omega\right),
		\\&	\left(\int_{0}^{t}\!\int_{U_0}1_{\{r>s\}}g\left(Z^{s,\eta}_{r-},r,z\right)\widetilde{N}_p\left(dr,dz\right)\right)\left(\omega\right)=Z_{2,t}^{s,\eta\left(\omega\right)}\left(\omega\right).
	\end{aligned}
\end{equation}
\end{lemma}
We observe that, using \eqref{dep_Zmeas} and arguing as in the proof of Lemma \ref{common}, it is possible to show that $Z^{s,\eta}$ in Lemma \ref{le_se} satisfies \eqref{serve_1} in an a.s. event $\Omega_\eta$ depending only on $\eta$. 
\begin{rem}\label{skor}
	In Subsections \ref{st_cont_section}-\ref{cadlag_sec} we have considered the process $Y$ with values in the complete metric space $\mathcal{C}_0=(C(\mathbb{R}^d;\mathcal{D}_0), d_0^{lu})$. Since $\mathcal{C}_0$ is not separable, we have endowed it with the $\sigma-$algebra $\mathcal{C}$ generated by the projections $\pi_x$ --strictly smaller than the Borel $\sigma-$algebra--  in order to overcome measurability issues. 
	
	An alternative approach which, at a first glance, might appear to be more natural is the following one. Denote by $\mathcal{D}_S$  the space of $\mathbb{R}^d-$valued, càdlàg functions on $\left[0,T\right]$ endowed with the Skorokhod topology $J_1$, i.e., $\mathcal{D}_S=\left(\mathcal{D}\left(\left[0,T\right];\mathbb{R}^d\right),J_1\right)$. According to \cite[Section 12, Chapter 3]{Bill}, $\mathcal{D}_S$ is a Polish space with the following metric defining the topology:
 \begin{equation} \label{sko1}
	d_S\left(x,y\right)= \inf_{\lambda\in\Lambda}\left\{\norm{\lambda}^0\vee\norm{x-y\circ\lambda}_0\right\},\quad x,y\in\mathcal{D}\left(\left[0,T\right];\mathbb{R}^d\right).
	\end{equation}
	Here $\Lambda$ is the set of continuous and strictly increasing functions $\lambda$ such that $\lambda\left(0\right)=0$ and $\lambda\left(T\right)=T$, and 
	 $
	\norm{\lambda}^0=\sup_{s<t}\big|\log\big(\frac{\lambda\left(t\right)-\lambda\left(s\right)}{t-s}\big)\big|.
	$ 
	Note that $J_1$ is weaker than the topology generated by the uniform convergence. Indeed, taking $\lambda=I$,
	\begin{equation}\label{ineq_dist}
	 d_S\left(x,y\right)\le\norm{x-y}_0,\quad x,y\in\mathcal{D}\left(\left[0,T\right];\mathbb{R}^d\right).
	\end{equation} 
Hence, for every $s\in\left[0,T\right]$, $Y_s\in C(\mathbb{R}^d;\mathcal{D}_S)$. By \cite{Kh}, the complete metric space $\left(C\left(\mathbb{R}^d;\mathcal{D}_S\right),{d}^{lu}_{S}\right),$ where
\[ 
d^{lu}_{S}\left(f,g\right)= \sum_{N=1}^{\infty}\frac{1}{2^N}\frac{\sup_{\left|x\right|\le N} d_S(f\left(x\right),g\left(x\right))}{1+\sup_{\left|x\right|\le N} d_S(f\left(x\right),g\left(x\right))},\quad f,g\in C\left(\mathbb{R}^d;\mathcal{D}_S\right),
\]
is also separable. Therefore we can argue as at the beginning of page 702 in \cite{AIHP18} to infer the measurability of $Y_s$ with respect to the Borel $\sigma-$algebra associated with $d^{lu}_S$. Observe that, by \eqref{ineq_dist} and the fact that $x(1+x)^{-1}$ is increasing in $\mathbb{R}_+$, 
$
	d^{lu}_S\left(f,g\right)\le d^{lu}_0\left(f,g\right),\, f,g\in C\left(\mathbb{R}^d;\mathcal{D}_S\right).
$
Thus, we can exploit the same computations as those presented in the paper to obtain the existence of a càdlàg modification $\widetilde{Y}$ of the $C\left(\mathbb{R}^d;\mathcal{D}_S\right)-$valued process $Y$. Moreover, using {\cite[Proposition 2.1, Chapter \upperRomannumeral{6}]{js}} we can prove Lemma  \ref{common}, as well.  
However, $\widetilde{Y}$ is not the \emph{regular} stochastic flow associated with  \eqref{SDE} according to Definition  \ref{sharpala}, because \ref{conx} and \ref{cons}] in Point \ref{regularity} hold in a weaker sense, namely replacing $\mathcal{D}_0$ with $\mathcal{D}_S$. As a consequence, for every $\bar{s}\in[0,T],\,t\in [0,T],\,x\in\mathbb{R}^d$ and $\omega\in\Omega$, we can not deduce that
$\lim_{x\to\bar{x}}\widetilde{Y}_t^{\bar{s}, {x}}(\omega)=\widetilde{Y}^{\bar{s},\bar{x}}_t(\omega)$
or
  $\lim_{s\downarrow\bar{s}}\widetilde{Y}_t^{s,\bar{x}}(\omega)=\widetilde{Y}^{\bar{s},\bar{x}}_t(\omega)$. 
\end{rem}

\section{Proof of existence of the regular stochastic flow  for   SDEs with large jumps}
\label{large_sec}
In this section,  we investigate the SDE \eqref{SDEgeneral} with $f\neq0$. 
Given $s\in[0,T)$ and $\eta\in L^0(\mathcal{F}_s),$ we study 
\begin{multline}\label{SDEJ} 
	X_{t}^{s,\eta} = \eta+ \int_{s}^{t}b\left(r, X_{r}^{s,\eta} \right) dr + 
	\int_{s}^{t}\alpha\left(r, X_{r}^{s,\eta} \right) dW_r\\+
	\int_{s}^{t}\!\int_{U_0}g\left(X_{r-}^{s,\eta},r,z\right)\widetilde{N}_p\left(dr,dz\right) 
	+ 
	\int_{s}^{t}\!\int_{U\setminus U_0 } f\left(X_{r-}^{s,\eta},r,z\right){N}_p\left(dr,dz\right),\quad t\in [s,T].
\end{multline} 
Compared to the SDE \eqref{SDE} that we have been discussing in Section \ref{sec_small}, \eqref{SDEJ} presents an additional integral with respect to the (non--compensated) Poisson random measure $N_p$. For this reason, \eqref{SDEJ} is often referred to as an SDE with \emph{large jumps}.  In particular, given $\omega\in \Omega$, one can read
\[
\left(\int_{s}^{t}\!\int_{U\setminus U_0 } f\left(X_{r-}^{s,\eta},r,z\right){N}_p\left(dr,dz\right)\right)(\omega)
=
\sum_{r\in D_p(\omega)\cap (s,t]} 1_{U\setminus U_0} \left(p_r(\omega)\right)f(X^{s,\eta}_{r-}(\omega),r,p_r(\omega)),
\]
with the sum on the right--hand side which is finite $\mathbb{P}-$a.s., because $\nu(U\setminus U_0)<\infty$ implies that $D_p(\omega)$ is discrete, $\mathbb{P}-$a.s. We study \eqref{SDEJ} by adapting an interlacing method described, for example, in  \cite{Bre, IW, LM}. Such an adaptation is not trivial for our scope of finding a regular stochastic flow generated by \eqref{SDEJ}, as detailed in Remark \ref{nonbanale}.

Recall that a solution to \eqref{SDEJ} is a c\`adl\`ag, $\mathbb{R}^d-$valued, $\mathbb{F}-$adapted process $X^{s,\eta}=\left(X^{s,\eta}_t\right)_{s\le t\le T}$ satisfying \eqref{SDEJ} up to indistinguishability. As usual, we extend the trajectories of $X^{s,\eta}$ in the whole interval $[0,T]$ by setting $X^{s,\eta}_t=X^{s,\eta}_s,\,t\in[0,s]$. Under our assumptions on the coefficients (see Section \ref{preliminare}), there exists a pathwise unique solution of \eqref{SDEJ}. 
\begin{rem}\label{nonbanale}
The existence of a pathwise unique solution of \eqref{SDEJ} can be proven by adapting the interlacing procedure described in, e.g, \cite[Subsection 3.2]{Bre} and \cite[Section 9, Chapter \upperRomannumeral{4}]{IW}) to the case $s\neq 0$. To do this, starting from $W$, $p$ and $\mathbb{F}$, we  construct a Brownian motion $W^{(s)}$ and a stationary Poisson point process $p^{(s)}$ with respect to a  filtration $\mathbb{F}^{(s)}$,  for every $s\in (0,T)$. On the other hand, it is not clear how to use this approach to prove the existence of a \emph{regular} stochastic flow generated by \eqref{SDEJ} according to Definition \ref{sharpala}. In particular, it is not clear how to analyze the regularity of the flow with respect to the initial time $s$. To overcome this issue, we follow an argument relying on the sharp stochastic flow $Z^{s,x}_t$ generated by the SDE \eqref{SDE} with small jumps (see Section \ref{sec_small}). Remarkably, we are also able to obtain an explicit expression --based on $Z^{s,x}_t$-- for the solution of \eqref{SDEJ}, from which we deduce the regularity properties that we are looking for.
\end{rem}
{\color{black}We now prove Theorem \ref{main1}, which asserts the existence of a regular stochastic flow generated by \eqref{SDEJ} according to Definition \ref{sharpala}. In order to make the proof easier to follow, in Theorem  \ref{thm_big} we reformulate the statement of Theorem \ref{main1} in an expanded version.}
\begin{theorem}\label{thm_big}
	There exist an $\mathcal{F}_{{}}\otimes \mathcal{B}([0,T]\times \mathbb{R}^d\times [0,T])-$measurable function $X\colon \Omega\times [0,T]\times \mathbb{R}^d\times [0,T]\to \mathbb{R}^d$, denoted by $X^{s,x}_t(\omega)=X(\omega,s,x,t)$, and an almost sure event $\Omega''$ (independent of $s,t$ and $x$)  such that
	 \begin{align}\label{SDEJ_common}
		\notag&X_{t}^{s,x} 
		= x+ \int_{s}^{t}b\left(r, X_{r}^{s,x} \right) dr + 
		\int_{s}^{t}\alpha\left(r, X_{r}^{s,x} \right) dW_r
		+
		\int_{s}^{t}\!\int_{U_0}g\left(X_{r-}^{s,x},r,z\right)\widetilde{N}_p\left(dr,dz\right) 
		\\&\qquad\qquad + 
		\int_{s}^{t}\!\int_{U\setminus U_0 } f\left(X_{r-}^{s,x},r,z\right){N}_p\left(dr,dz\right),
				\quad s\in \left[0,T\right),\,t\in\left[0,T\right],\,x\in\mathbb{R}^d,\,\omega\in\Omega'',
	\end{align}
	and such that the flow property holds:
	\begin{equation}\label{flow_finale}
		X^{s,x}_t(\omega)=X^{u,X^{s,x}_u(\omega)}_t(\omega),\quad 0\le s<u<t\le T,\,x\in\mathbb{R}^d,\,\omega\in\Omega''.
	\end{equation}

	{\color{black}Furthermore,  the process $(X^s)_{s\in[0,T]}$ is stochastically continuous in the sense of Point \ref{constoch} in Definition \ref{sharpala}, and, for every $\omega\in\Omega$, the mapping $(s,t,x)\mapsto X^{s,x}_t(\omega)$ satisfies Point \ref{regularity} in Definition \ref{sharpala}.}
\end{theorem}
\begin{proof}
	In order not to complicate the notation, we are going to construct the flow $X^{s,x}_t$ generated by \eqref{SDEJ_common} with $t\in [0,T)$, excluding the upper bound $T$. Since the proof is rather long, we divide it into several steps.
	
	\vspace{1mm}
	 \underline{\emph{Step \upperRomannumeral{1}}}\emph{: Construction of the flow $X^{s,x}_t$.} 
	 Denote by $P_t =N_p((0,t]\times (U\setminus U_0)),\,t>0$, and set $P_0=0$: since $\nu(U\setminus U_0)\in(0,\infty)$, $P=(P_t)_{t\ge0}$ is a Poisson process with intensity $\nu(U\setminus U_0)$. Let $\tau_n,\,n\in\mathbb{N}$, be  the arrival times for the jumps of $P$. It is well known that $\tau_n(\omega)\uparrow \infty $ as $n\to \infty$, for every $\omega\in \Omega_1$, where $\Omega_1$ is an a.s. event (see, e.g., \cite[Theorem 21.3]{Sato} and the subsequent comment). Notice that  $P$ is càdlàg and continuous in probability, hence it does not jump at time $T$, $\mathbb{P}-$a.s. Thus, we suppose that $\tau_n\neq T$ in $\Omega_1$, for every $n\in\mathbb{N}$. 
	
	To construct the solution of \eqref{SDEJ_common} we use $Z=(Z_s)_{s\in[0,T]}$: the càdlàg, $(\mathcal{C}_0,\mathcal{C})-$valued process studied in Subsection \ref{cadlag_sec} satisfying \eqref{common_eq}, see Lemma \ref{common}. Note that \eqref{common_eq} is the analogous of \eqref{SDEJ_common} without the integral in $N_p$, i.e., without the ``large jumps''.  We argue that $Z^{s,x}_\cdot(\omega)$ does not jump at  $ \tau_n(\omega)\in (s,T),$ for every $s\in [0,T),\,x\in\mathbb{R}^d$ and $\omega\in \Omega_{2}$, where $\Omega_{2}$ is an a.s. event. To see this, we take  a sequence $(U_n)_n\subset \mathcal{U}$ such that $\nu(U_n)<\infty$ and $\cup_n\, U_n=U$, which exists because  $\nu(dz)$ is $\sigma-$finite. Moreover, we denote by 
	\[
	g_n\left(r,z,\omega\right)=1_{(-n,n)}\left(g\left(Z_{r-}^{s,x}\left(\omega\right),r,z\right)\right)
	g\left(Z_{r-}^{s,x}\left(\omega\right),r,z\right),\quad n\in\mathbb{N}.
	\]
	By construction of the stochastic integral with respect to $\widetilde{N}_p$ (see  \cite[Section 3, Chapter \upperRomannumeral{2}]{IW}),  
	\begin{equation*}
		\lim_{n\to \infty}	\int_{s}^{t}\!\int_{U_0\cap U_n}
		g_n\left(r,z,\cdot\right)\widetilde{N}_p\left(dr,dz\right) 
		=
		\int_{s}^{t}\!\int_{U_0}g\left(Z_{r-}^{s,x},r,z\right)\widetilde{N}_p\left(dr,dz\right),
	\end{equation*}
	where the limit is uniform on compacts in probability. It follows that, $\mathbb{P}-\text{a.s.}$,
	\begin{equation}\label{uni_conve}
		\sup_{t\in [s,T]}\left|\int_{s}^{t}\!\int_{U_0\cap U_{n_k}}
		g_{n_k}\left(r,z,\cdot\right)\widetilde{N}_p\left(dr,dz\right) 
		-
		\int_{s}^{t}\!\int_{U_0}g\left(Z_{r-}^{s,x},r,z\right)\widetilde{N}_p\left(dr,dz\right)\right|
		\underset{k\to \infty}{\longrightarrow} 0.
	\end{equation}
	Since $p_{\tau_n}\in U\setminus U_0$ and,  for every $k\in\mathbb{N}$, for $\mathbb{P}-$a.s. $\omega\in \Omega$,
	\begin{multline*}
		\left(\int_{s}^{t}\!\int_{U_0\cap U_{n_k}}
		g_{n_k}\left(r,z,\cdot\right)\widetilde{N}_p\left(dr,dz\right)\right)\left(\omega\right) =
		\sum_{r\in D_p(\omega)\cap (s,t]}g_{n_k}\left(r,p_r(\omega),\omega\right)1_{U_0\cap U_{n_k}}\left(p_r(\omega)\right)
		\\-
		\int_{s}^{t}dr\!\int_{U_0\cap U_{n_k}}
		g_{n_k}\left(r,z,\omega\right)\nu\left(dz\right),\quad t\in [s,T],
	\end{multline*}
	these approximating processes do not jump at time $\tau_{{n}}\in (s,T)$ for all $n\in\mathbb{N}, \, \mathbb{P}-$a.s.
	Therefore, by \eqref{uni_conve}, the process $\int_{s}^{\cdot}\!\int_{U_0}g\left(Z_{r-}^{s,x},r,z\right)\widetilde{N}_p\left(dr,dz\right)$ does not jump either at time $\tau_{{n}}\in (s,T)$ in an a.s. event depending on $s$ and $x$. Whence the same conclusion holds for $Z^{s,x}_\cdot$ in an a.s. event $\Omega_{s,x}$, by \eqref{common_eq}.  \\
	Define the a.s. event 
	\[
	\Omega_2=\bigcap_{s\in [0,T)\cap \mathbb{Q}}\bigcap_{x\in\mathbb{Q}^d}\Omega_{s,x}:
	\]
	we are going to show that $Z_\cdot^{s,x}(\omega)$ does not jump at  $\tau_n(\omega)\in(s,T)$ for every $x\in\mathbb{R}^d,\,s\in[0,T)$ and $\omega\in \Omega_2$. 
	Fix $\omega\in\Omega_2$, $s\in[0,T)\cap \mathbb{Q},\, x\in\mathbb{R}^d$ and take a sequence $(x_m)_m\in\mathbb{Q}^d$ such that $x_m\to x $ as $m\to \infty$. Since the map $Z^{{s},\cdot}\left(\omega\right)\colon \mathbb{R}^d\to \mathcal{D}_0$ is continuous and $Z^{s,x_m}_\cdot(\omega)$ is continuous at  $\tau_n(\omega),$ for all $m\in\mathbb{N}$, we conclude that $Z^{s,x}_\cdot(\omega)$ is continuous at $\tau_n(\omega)$, as well. Indeed, uniform convergence in $t$ preserves continuity. 
	An analogous argument relying on the càdlàg property of the map $Z^{\cdot, {x}}\left(\omega\right)\colon [0,T] \to \mathcal{D}_0$ allows to deduce the continuity of $Z^{s,x}_\cdot(\omega)$ at $\tau_n(\omega)$ for every $s\in[0,T)$, as desired.
	
	Recalling the almost certain event $\Omega'$ given by Lemma \ref{common}, we define $\Omega_3=\Omega_1\cap \Omega_2\cap\Omega'$. Without loss of generality, we suppose that $Z^{s,x}_t(\omega)=x$ for every $0\le t\le s \le T, \,x\in\mathbb{R}^d$ and $\omega\in \Omega_3$. 	For the sake of shortness, from now on 
	\begin{equation}\label{chiaro}
	\text{	{we  denote by $ \,\tau_n $ the random variable  $ \tau_n\wedge T,\,n\in\mathbb{N}$. }}
	\end{equation}
We now construct the solution $X$ of \eqref{SDEJ} using an ad hoc, path--by--path, interlacing procedure (see Remark \ref{nonbanale}). For $s = T$, we just assign $X^{T,x}_t = x$. Take $s\in [0,T)$ and $x\in\mathbb{R}^d$. First, we set $X^{s,x}_t(\omega)=x$ for $\omega\in \Omega\setminus \Omega_3,\,t\in [0,T]$. Secondly, fix $\omega\in\Omega_3$ and denote by $\tau_{n(s,\omega)}(\omega)=\min_n\{\tau_{n}(\omega)>s\}$. In words, if $\tau_{{n(s,\omega)}}(\omega)<T$, then it represents the first jump time of $P$ occurring (strictly) after time $s$ and before time $T$. In the sequel, we omit $\omega$ to keep notation simple. We  define
	 \begin{equation}\label{missala}
	X_{t}^{s,x}=  Z_{t}^{s,x},\quad t\in\left[0,\tau_{n(s)}\right).
	\end{equation}
	 If $\tau_{{n(s)}}=T$ the construction is over. Otherwise, for $t = \tau_{n(s)}$ we set 
	\begin{align}\label{1step}
		\notag	X_{\tau_{n(s)}}^{s,x} 
		&=   X_{\tau_{n(s)}- }^{s,x} + f\left( X_{\tau_{n(s)} - }^{s,x}, \tau_{n(s)},  p_{\tau_{n(s)}} \right)
		= Z_{\tau_{n(s)} - }^{s,x} + f \left( Z_{\tau_{n(s)} - }^{s,x}, \tau_{{n(s)}}, p_{\tau_{n(s)}} \right)  \\
		&= Z_{\tau_{n(s)}}^{s,x} + f \left( Z_{\tau_{n(s)} }^{s,x}, \tau_{{n(s)}}, p_{\tau_{n(s)}} \right),  
	\end{align} 
	where the last equality is due to the fact that $\tau_{{n(s)}}$ is not  a jump time for $Z_\cdot^{s,x}$, because $\omega\in \Omega_3\subset \Omega_2$. Next, we define
\begin{equation}\label{2step}
	X_{t}^{s,x}= Z^{\tau_{n(s)}, X_{\tau_{n(s)}}^{s,x} }_t,\quad   t \in \left[\tau_{n(s)}, \tau_{n(s)+1}\right).
\end{equation}

This argument by steps can be repeated to cover the whole interval $[s,T)$. More precisely, for every $m\in\mathbb{N}$ such that $\tau_{{n(s)}+m}<  T$, we define recursively 
	\begin{equation}\label{rec}
		X_t^{s,x}=
		\begin{cases}
			X_{\tau_{n(s)+m}- }^{s,x} + f\left( X_{\tau_{n(s)+m} - }^{s,x}, \tau_{n(s)+m},  p_{\tau_{n(s)+m}} \right),&t=\tau_{{n(s)+m}}, \\
			Z_t^{\tau_{{n(s)+m}},X^{s,x}_{\tau_{{n(s)+m}}}},&t\in [\tau_{{n(s)+m}},\tau_{{n(s)+m+1}} ),
		\end{cases}
	\end{equation}
	In particular, since $\omega\in\Omega_3\subset \Omega_2$ we observe that 
	\begin{equation}\label{meglio}
		X^{s,x}_{\tau_{{n(s)+m}}-}=
		Z_{\tau_{{n(s)+m}}}^{\tau_{{n(s)+m-1}},X^{s,x}_{\tau_{{n(s)+m-1}}}}.
	\end{equation}
	We finally extend the map $X^{s,x}_\cdot$ to $[0,T]$ by setting  $X^{s,x}_T=X^{s,x}_{T-}.$
	
		\vspace{1mm}
	\underline{\emph{Step \upperRomannumeral{2}}}\emph{: The process $(X^{s,x}_t)_{t\in[0,T]}$ is $\mathbb{F}-$adapted.} 
	The claim is trivial if $s=T$  because $X^{T,x}_t=x$, so we consider $s\in [0,T)$. For every $t\in [0,T)$, setting $\tau_0=0$ we have (recall \eqref{chiaro})
	\begin{equation}\label{meas}
		X_t^{s,x}=x1_{\Omega\setminus\Omega_3}+x1_{\{t\le s\}}1_{\Omega_3}+ \sum_{n=1}^{\infty}Z^{\tau_{n-1}\vee s, X^{s,x}_{\tau_{n-1}\vee s}}_t1_{\{\tau_{n-1}\vee s\le t<\tau_n\vee s\}}1_{\Omega_3}.
	\end{equation}
Notice that the series in \eqref{meas} is actually a finite sum, as  $\tau_n(\omega)=T$ definitively in $\Omega_3$, hence $[\tau_{n-1}(\omega)\vee s,\tau_n\vee s(\omega))=\emptyset$ definitively in $\Omega_3$.
	In what follows, we write $\tau^s_{{n}}=\tau_n\vee s,\,n\in\mathbb{N}_0$. Recalling that the filtration $\mathbb{F}$ is complete, $x1_{\Omega\setminus\Omega_3}$ and $x1_{\{t\le s\}}1_{\Omega_3}$ are $\mathcal{F}_t-$measurable. Since $(\tau^s_n)_n$ is a sequence of $\mathbb{F}-$stopping times, the sets $\{\omega :  \tau^s_{n-1}(\omega)\le t<\tau_n^s(\omega)\},\,n\in\mathbb{N},$  are $\mathcal{F}_t-$measurable. As a consequence, $Z_t^{s,x}1_{\{s\le t<\tau_1^ s\}}1_{\Omega_3}$ --the first term of the series in \eqref{meas}-- is $\mathcal{F}_t-$measurable. Moreover, by Lemma \ref{joint_meas},
	$
	Z^{s,x}_{\tau_1^s}1_{\{\tau^s_{1}\le t\}}
	$
	is $\mathcal{F}_t-$measurable, so (by \eqref{1step})
	\[
	X_{\tau_1^s}^{s,x}1_{\{\tau_{1}^s\le  t\}}1_{\Omega_3}
	= 
	\left(Z_{\tau_{1}^s}^{s,x} + 1_{\{\tau_1^s>s\}}f\left( Z_{\tau_{1}^s }^{s,x}, \tau_{1}^s,  p_{\tau_{1}^s} \right)\right)
	1_{\{\tau_{1}^s\le t\}}1_{\Omega_3}
	\]
	is $\mathcal{F}_t-$measurable, too. Hence another application of Lemma \ref{joint_meas} yields the $\mathcal{F}_t-$measurability of  the second term of the series in \eqref{meas}, i.e.,
	\[
	Z_t^{\tau_1^s,X^{s,x}_{\tau_1^s}}1_{\{\tau_{1}^ s\le t<\tau_2^ s\}}1_{\Omega_3}.
	\]
	At this point, an induction argument based on the recursive definition law in \eqref{rec} allows us to conclude that all the addends in the series \eqref{meas} are $\mathcal{F}_t-$measurable. Therefore, considering also that $X^{s,x}$ is left--continuous in $T$, we deduce that the process $X^{s,x}$ is $\mathbb{F}-$adapted, as desired. 
	
	\vspace{1mm}
	\underline{\emph{Step \upperRomannumeral{3}}}\emph{: The regularity of the flow $X^{s,x}_t$.} The aim of this part is to prove \ref{pure}-\ref{conx}-\ref{cons} in Definition \ref{sharpala}. We only analyze  the case $\omega\in \Omega_3$, being the other one trivial ($X^{s,x}_t(\omega)=x$, $\omega\in\Omega\setminus \Omega_3$). Conditions \ref{pure}-\ref{conx} are immediate also for $s=T$ because $X^{T,x}_t=x$, so we consider $s\in[0,T)$. Recalling that the series in \eqref{meas} is actually a finite sum, for every $x\in\mathbb{R}^d$ the càdlàg property with respect to $t\in [0,T]$ is evident, because the path $X_{\cdot}^{s,x}(\omega)$ is constructed by combining a finite number of càdlàg trajectories of the flow $Z^{s,x}_t$. Hence \ref{pure} is verified.\\
	 To study the continuity in $x$ in the sense of \ref{conx}, we  take $x\in\mathbb{R}^d$ and a sequence $(x_j)_j\subset \mathbb{R}^d$ such that $x_j\to x$ as $j\to \infty$. Since  $X^{s,x_{j}}_t=Z^{s,x_{j}}_t$ and $X^{s,x}_t=Z^{s,x}_t$ for $t\in[0,\tau_{{n(s)}})$, and  $Z_s\in \mathcal{C}_0$,
	 \[
	 	\lim_{j\to \infty}\sup_{0\le t< \tau_{{n(s)}}}\left|X^{s,x_j}_t-X_t^{s,x}\right|=0.
	 \]
	 If $\tau_{n(s)}<T$, then by \eqref{1step} and the continuity of $f$ in the first argument we have $X^{s,x_j}_{\tau_{{n(s)}}}\to X^{s,x}_{\tau_{{n(s)}}}$ as $j\to \infty$, from which we deduce, by \eqref{2step},
	 \[
	 \lim_{j\to \infty}\sup_{t\in(\tau_{n(s)}, \tau_{n(s)+1})}\left|X^{s,x_j}_t-X_t^{s,x}\right|=0.
	 \]
	  In general, using \eqref{rec}-\eqref{meglio} we can work by induction to obtain \ref{conx}.\\
	Finally we study the càdlàg property in the variable $s\in [0,T]$ according to \ref{cons}.  Firstly, we analyze the right--continuity in $s\in [0,T)$. Fix $M>0$ and take a sequence $(s_j)_j\subset (s,T)$ such that $s_j\to s$ as $j\to \infty$. We assume, without loss of generality, that $\tau_{{n(s_j)}}=\tau_{{n(s)}}$ for all $j$. Since $Z$ is a $(\mathcal{C}_0,\mathcal{C})-$valued càdlàg process,  by construction
	\[
	\lim_{j\to \infty} \sup_{\left|x\right|\le M}\sup_{0\le t < \tau_{n(s)}}\left|X^{s_j,x}_t-X^{s,x}_t\right|=0.
	\] 
	If $\tau_{n(s)}<T$, notice that the set $\big\{Z_{\tau_{n(s)}}^{s_{j},x},Z_{\tau_{n(s)}}^{s,x},  \text{ with } \left|x\right|\le M,\,j\in\mathbb{N}\big\}$ is bounded. Considering that $f$ is uniformly continuous in the first variable on compact sets, by \eqref{1step} we deduce that 
	\begin{equation*}
	 \sup_{\left|x\right|\le M}\left|X^{s_j,x}_{\tau_{n(s)}}\!-\!X^{s,x}_{\tau_{n(s)}}\right|\!\le\!
	\sup_{\left|x\right|\le M}\!\left|Z^{s_j,x}_{\tau_{n(s)}}-Z^{s,x}_{\tau_{n(s)}}\right|
	  +
	     \sup_{\left|x\right|\le M}\left| f \left( Z_{\tau_{n(s)} }^{s_j,x}, \tau_{{n(s)}}, p_{\tau_{n(s)}} \right)\!-\! f \left( Z_{\tau_{n(s)} }^{s,x}, \tau_{{n(s)}}, p_{\tau_{n(s)}} \right)\right|
	  \! \underset{j\to \infty}{\longrightarrow}\!0
	 .
	\end{equation*}
As $x\mapsto Z^{\tau_{n(s)},x}_\cdot$ is a continuous function from $\mathbb{R}^d$ to $\mathcal{D}_0$, it is uniformly continuous on compact sets of $\mathbb{R}^d$. Moreover, the previous equation coupled with
	\[
		 \sup_{\left|x\right|\le M}\left|X^{s_j,x}_{\tau_{n(s)}}\right|<\infty,\,j\in\mathbb{N},\qquad  \sup_{\left|x\right|\le M}\left|X^{s,x}_{\tau_{n(s)}}\right|<\infty\quad \text{(by the continuity in $x$)}
	\] 
	ensures that the set $\big\{X_{\tau_{n(s)}}^{s_{j},x},X_{\tau_{n(s)}}^{s,x},  \text{ with } \left|x\right|\le M,\,j\in\mathbb{N}\big\}$ is bounded. 
Combining these two facts, by \eqref{2step}
	\[
\lim_{j\to \infty} \sup_{\left|x\right|\le M}\sup_{\tau_{n(s)}< t < \tau_{n(s)+1}}\left|X^{s_j,x}_t-X^{s,x}_t\right|=0.
\] 
 Using \eqref{rec}-\eqref{meglio} we argue by induction to  infer the right--continuity in $s$ in the sense of \ref{cons}. \\
 Secondly, we can prove the left--continuity of $X^{\cdot, x}_t$ in $s\in(0,T]$  in a similar way, exploiting the left--continuity of the process $(Z_s)_s$. We only note that, in this case, it is possible that $s=\tau_n$ for some $n\in\mathbb{N}$. Hence given a sequence $(s_j)_j\subset (0,s)$ such that $s_j\to s$ as $j\to \infty$, we might have $\tau_{{n(s_j)}}=s<\tau_{{n(s)}}$ for $j$ large enough. This, however, does not affect the existence of the left--limits because $\tau_{{n(s_j)}}$ are definitively all  equal. 
	
	\vspace{1mm}
	\underline{\emph{Step \upperRomannumeral{4}}}\emph{: The stochastic continuity of the flow $X^{s,x}_t$}. 
	{\color{black} To obtain the stochastic continuity in the sense of Point \ref{constoch} in Definition \ref{sharpala}, it is sufficient to prove that, for every $s\in (0,T]$, there exists an a.s. event $\Omega_s$ where \begin{equation}\label{cont_X_st}
X^{s-,x}_t=X^{s,x}_t,\quad x\in\mathbb{R}^d,\,t\in [0,T].
		\end{equation}
		Indeed, combining this equality with the càdlàg property we have just proved, we deduce that
		\begin{multline*}
		0=\lim_{r\uparrow s}\mathbb{P}\left(\sup_{\left|x\right|\le M}\sup_{0\le t\le  T}\left|X^{r,x}_t-X^{s-,x}_t\right|>\epsilon\right)=
		\lim_{r\uparrow s}\mathbb{P}\left(\sup_{\left|x\right|\le M}\sup_{0\le t\le  T}\left|X^{r,x}_t-X^{s,x}_t\right|>\epsilon\right)\\=
		\lim_{r\downarrow s}\mathbb{P}\left(\sup_{\left|x\right|\le M}\sup_{0\le t\le  T}\left|X^{r,x}_t-X^{s,x}_t\right|>\epsilon\right)
		=
		0,\quad  \epsilon, \, M>0.
		\end{multline*}
The case $s=T$ is the easier one: by construction and \eqref{nftd} we have, in an a.s. event contained in $\Omega_3$ and depending on $T$, 
\[
X^{T-,x}_t=\lim_{r\uparrow T}X^{r,x}_t=\lim_{r\uparrow T}Z^{r,x}_t=
Z^{T-,x}_t=Z^{T,x}_t=x=X^{T,x}_t,\quad t\in[0,T),\,x\in\mathbb{R}^d;
\]
 the final time $t=T$ can be recovered by passing to the limit as $t\to T$, because $X^{T-,x}_\cdot$ and $X^{T,x}_\cdot$ are left--continuous in $T$. As for $s\in [0,T)$, we can argue as at the beginning of this proof to construct an a.s. event $\Omega_s\subset \Omega_3$ such that $\tau_{{n}}(s)\neq s$ for all $n\in\mathbb{N}$, and that  (recall \eqref{nftd}) $Z_{s-}=Z_s$. Then, by \eqref{missala},
 \[
 X^{s-,x}_t=Z^{s-,x}_t=Z^{s,x}_t=X^{s,x}_t,\quad t\in [0,\tau_{n(s)}),\,x\in\mathbb{R}^d, \text{ in $\Omega_s$}:
 \]
  if $\tau_{n(s)}<T$, this equality holds also for $t=\tau_{{n(s)}}$ by  \eqref{1step}. Employing \eqref{2step}-\eqref{rec}-\eqref{meglio} and the left--continuity of $X^{s-,x}_\cdot,\,X^{s,x}_\cdot$ in $T$, we reason by induction to obtain \eqref{cont_X_st}.
}

\vspace{1mm}
\underline{\emph{Step \upperRomannumeral{5}}}\emph{: The stochastic flow $X^{s,x}_t$ satisfies \eqref{SDEJ_common}}.
		Recall that $\tau^s_{{m}}=\tau_m\vee s,\,m\in\mathbb{N}_0$, where $\tau_m$ is given in \eqref{chiaro}.
	We now argue --using Lemma \ref{le_se}-- that the process $X^{s,x}_t$ satisfies \eqref{SDEJ} with $\eta=x$ in $t\in[0,T)$. This is equivalent to showing that, for every $m\in\mathbb{N}_0$, there is an a.s. event $\Omega_{1,m}(s,x)$ such that, for all $t\in[\tau^s_{m},\tau^s_{m+1})$ (note that $[\tau^s_{m},\tau^s_{m+1})$ can also be empty),
	\begin{multline}\label{prova_questo}
		X_{t}^{s,x}= x1_{\left\{\tau_m^s\le s\right\}}+ 
		\left(X^{s, x}_{\tau_{m}^s-}+ f\left(X^{s,x}_{\tau_m^s-},\tau_m^s,p_{\tau^s_m}\right)\right)1_{\left\{\tau^s_m> s\right\}} 
		+\int_{s}^{t}1_{{\{r>\tau_m^s\}}}b\left(r, X_{r}^{s,x} \right) dr \\+ 
		\int_{s}^{t}1_{{\{r>\tau_m^s\}}}\alpha\left(r, X_{r}^{s,x} \right) dW_r+
		\int_{s}^{t}\!\int_{U_0}1_{{\{r>\tau_m^s\}}}g\left(X_{r-}^{s,x},r,z\right)\widetilde{N}_p\left(dr,dz\right).
	\end{multline}
	Indeed, the stochastic integrals appearing in the previous expression can be read as  differences involving truncated processes, see, for instance, \cite[Section 3, Chapter \upperRomannumeral{2}]{IW} and \cite[Property 4.37, Chapter \upperRomannumeral{1}]{js}. More precisely, $\mathbb{P}-$a.s., for every $t\in[s,T)$,
	\begin{equation*}
		\int_{s}^{t}1_{{\{r>\tau_m^s\}}}\alpha\left(r, X_{r}^{s,x} \right) dW_r=\int_{s}^{t}\alpha\left(r, X_{r}^{s,x} \right) dW_r-
		\left(\int_{s}^{\cdot}\alpha\left(r, X_{r}^{s,x} \right) dW_r\right)_{\tau_{{m}}^ s\wedge t}
	\end{equation*}
	and similarly
	 \begin{multline*}
		\int_{s}^{t}\!\int_{U_0}1_{{\{r>\tau_m^s\}}}g\left(X_{r-}^{s,x},r,z\right)\widetilde{N}_p\left(dr,dz\right)=
		\int_{s}^{t}\!\int_{U_0}g\left(X_{r-}^{s,x},r,z\right)\widetilde{N}_p\left(dr,dz\right)\\
		-
		\left(\int_{s}^{\cdot}\!\int_{U_0}g\left(X_{r-}^{s,x},r,z\right)\widetilde{N}_p\left(dr,dz\right)\right)_{\tau^s_m\wedge t}.
	\end{multline*}
	In view of the interlacing construction carried out above (cf. \eqref{meas}), in order to verify \eqref{prova_questo} we search for an a.s. event $\Omega_{1,m}(s,x)\subset\Omega_3$ such that, for all $t\in[\tau_{m}^s, \tau_{m+1}^s) $, 
	\begin{multline*}
		Z^{s,x}_t=x+ \int_{s}^{t}1_{{\{r>\tau_m^s\}}} b\left(r, Z_{r}^{s,x} \right) dr + 
		\int_{s}^{t}1_{{\{r>\tau_m^s\}}}\alpha\left(r, Z_{r}^{s,x} \right) dW_r\\+
		\int_{s}^{t}\!\int_{U_0}1_{{\{r>\tau_m^s\}}}g\left(Z_{r-}^{s,x},r,z\right)\widetilde{N}_p\left(dr,dz\right),\quad \text{ in }\Omega_{1,m}(s,x)\cap \{\tau^s_m\le s\},
	\end{multline*}
	and 
	\begin{multline}\label{true_int}
		Z_t^{\tau^s_m, X^{s,x}_{\tau^s_m}}=X^{s,x}_{\tau^s_m}+
		\int_{s}^{t}1_{{\{r>\tau_m^s\}}}b\left(r, Z_{r}^{\tau_m^s, X^{s,x}_{\tau_m^s}}\right) dr + 
		\int_{s}^{t}1_{{\{r>\tau_m^s\}}}\alpha\left(r, Z_{r}^{\tau_m^s, X^{s,x}_{\tau_m^s}} \right) dW_r\\+
		\int_{s}^{t}\!\int_{U_0}1_{{\{r>\tau_m^s\}}}g\left(Z_{r-}^{\tau^s_m, X^{s,x}_{\tau_m^s}},r,z\right)\widetilde{N}_p\left(dr,dz\right),\quad \text{in } \Omega_{1,m}(s,x)\cap \{\tau^s_m> s\}. 
	\end{multline}
	In $\{\tau^s_{{m}}\le s\}$, the former equation can be rewritten without the indicator functions, namely
	\begin{equation*}
		Z^{s,x}_t=x+ \int_{s}^{t}b\left(r, Z_{r}^{s,x} \right) dr + 
		\int_{s}^{t}\alpha\left(r, Z_{r}^{s,x} \right) dW_r+
		\int_{s}^{t}\!\int_{U_0}g\left(Z_{r-}^{s,x},r,z\right)\widetilde{N}_p\left(dr,dz\right),
	\end{equation*}
	which,  by  \eqref{common_eq},  holds in the whole $\Omega_3\subset \Omega'$. Thus, we only focus on \eqref{true_int}. Note that, in \eqref{true_int}, we can insert $0$ instead of $s$  as lower bound for the integrals because we are working in $\{\tau^s_m> s\}$. 
	
Consider a non--increasing sequence of simple random variables $(\tau_{m,n})_{n}$, with $\tau_{{m,n}}\le  T$,  such that $\tau_{m,n}\downarrow \tau_m^s$ as $n\to \infty$ in an a.s. event $\Omega_{2,m}(s)\subset \Omega_3$: \eqref{true_int} holds if we replace $\tau^s_m$ with $\tau_{m,n}$. More precisely,  we write $\tau_{m,n}=\sum_{k=1}^{N_{n}}a^n_k1_{A^n_k}$ for some $N_{n}\in\mathbb{N}$, $(a^n_k)_k\subset(-\infty,T]$ and some measurable partition $(A^n_k)_k\subset \mathcal{F}$ of $\Omega$,  $k=1,\dots, N_{n}$. Then, by Lemma \ref{le_se}, 
	 there exists an a.s. event where, for every $t\in [0,T]$ and $n\in\mathbb{N}$,
	\begin{align*}
		Z_t^{\tau_{m,n}, X^{s,x}_{\tau_{m,n}}}&=
		\sum_{k=1}^{N_{n}}Z_t^{a^n_k, X^{s,x}_{a^n_k}}1_{A^n_k}
		\notag\\&=\sum_{k=1}^{N_{n}}\Bigg[
		X^{s,x}_{a^n_k}+
		\int_{0}^{t}1_{\{r>a_k^n\}}b\left(r, Z_{r}^{a^n_k, X^{s,x}_{a^n_k}}\right)dr
		+
		\int_{0}^{t}1_{\{r>a_k^n\}}\alpha\left(r, Z_{r}^{a^n_k, X^{s,x}_{a^n_k}} \right)\!dW_r\notag\\&\qquad\quad+
		\int_{0}^{t}\!\int_{U_0}1_{\{r>a_k^n\}}g\left(Z_{r-}^{a^n_k, X^{s,x}_{a^n_k}},r,z\right)\!\widetilde{N}_p\left(dr,dz\right)\Bigg]1_{A^n_k},
	\end{align*}
Note that here we do  not insert $1_{A_n^k}$ inside the stochastic integrals in order not to lose the adaptedness of the integrands.
Invoking \eqref{dep_Zmeas}, the previous equation can be rewritten as follows: 
\begin{align}\label{limit}
	Z_t^{\tau_{m,n}, X^{s,x}_{\tau_{m,n}}}=X^{s,x}_{\tau_{m,n}}+
	Z_{1,t}^{\tau_{m,n},X^{s,x}_{\tau_{{m,n}}}}+
	Z_{2,t}^{\tau_{m,n},X^{s,x}_{\tau_{{m,n}}}}+
	Z_{3,t}^{\tau_{m,n},X^{s,x}_{\tau_{{m,n}}}},\quad t\in [0,T],\,n\in\mathbb{N},
\end{align}
which holds in an a.s. event $\Omega_{3,m}(s,x)\subset \Omega_3$. 
Since $X^{s,x}_t$ is a càdlàg function of $t$ and  $Z$ is a $\mathcal{C}_0-$valued càdlàg process,  
	\begin{equation*}
		\lim_{n\to \infty}X^{s,x}_{\tau_{{m,n}}}=X^{s,x}_{\tau^s_{{m}}}	,\qquad 
		\lim_{n\to \infty}Z_t^{\tau_{m,n}, X^{s,x}_{\tau_{m,n}}}=Z_t^{\tau^s_{{m}}, X^{s,x}_{\tau^s_m}}
		,\quad \text{in }\Omega_{2,m}(s).
	\end{equation*}
Thus, recalling that also the processes $Z_1,\,Z_2$ and $Z_3$ are càdlàg with values in the space $\mathcal{C}_0$, we can pass to the limit in \eqref{limit} as $n\to \infty$ to deduce that, in the a.s. event $\Omega_{1,m}(s,x)=\Omega_{2,m}(s)\cap\Omega_{3,m}(s,x)$,
\begin{equation}\label{xcommon}
Z_t^{\tau^s_{m}, X^{s,x}_{\tau^s_{m}}}=X^{s,x}_{\tau^s_{m}}+Z_{1,t}^{\tau^s_{{m}},X^{s,x}_{\tau^s_{{m}}}}+
Z_{2,t}^{\tau^s_{{m}},X^{s,x}_{\tau^s_{{m}}}}+
Z_{3,t}^{\tau^s_{{m}},X^{s,x}_{\tau^s_{{m}}}}
,\quad t\in [0,T].
\end{equation}	
Therefore \eqref{true_int} is satisfied  in $\Omega_{1,m}(s,x)$ on the entire $[0,T]$, proving that $X^{s,x}$ is a solution to \eqref{prova_questo}. Hence $X^{s,x}$ solves \eqref{SDEJ} with $\eta=x$. 
	
It remains to find an a.s. event $\Omega''$ --not depending on $s,\,x$ and $t$-- where \eqref{SDEJ_common} is satisfied. If we define 
	\[
	\Omega_{1,m}=\bigcap_{s\in [0,T)\cap \mathbb{Q}}\bigcap_{x\in\mathbb{Q}^d}\Omega_{1,m}(s,x),\quad m\in\mathbb{N}_0,
	\]
	then \eqref{xcommon} is simultaneously satisfied in $\Omega_{1,m}$ for every $s\in [0,T)\cap \mathbb{Q}$ and $x\in\mathbb{Q}^d$. In fact, the continuity of the flow $X^{s,x}_t$ in $x$ implies that \eqref{xcommon} holds for every $x\in\mathbb{R}^d$ in $\Omega_{1,m}$, with $s$ being a rational number in $[0,T)$.\\	If $s\in[0,T)\cap (\mathbb{R}\setminus\mathbb{Q})$, then we just consider  $(s_n)_n\subset (s,T)\cap \mathbb{Q}$ such that $s_n\downarrow s$ as $n\to\infty$, and another limiting argument based on the regularity of $Z,\,Z_1,\,Z_2,\,Z_3$ and $X^{s,x}_t$ shows that \eqref{xcommon} holds for this choice of $s$, too. Summarizing, \eqref{xcommon} holds in $\Omega_{1,m}$ for all $s\in[0,T)$ and $x\in\mathbb{R}^d$. \!Therefore the flow $X^{s,x}_t$ \!satisfies \!\eqref{SDEJ_common} \!in  $$\Omega''=\cap_{m=0}^\infty \Omega_{1,m}.$$
	
	\vspace{1mm}
	\underline{\emph{Step \upperRomannumeral{6}}}\emph{: The flow property \eqref{flow_finale}}. Note that, for every $\bar{t}\in [0,T]$, the function $X\colon\Omega\times [0,T]\times \mathbb{R}^d\times [0,\bar{t}]\to \mathbb{R}^d$ defined by $X(\omega,s,x,t)=X^{s,x}_t(\omega)$ is $ \mathcal{F}_{{}}\otimes \mathcal{B}([0,T]\times \mathbb{R}^d\times [0,\bar{t}])-$measurable by \eqref{meas} and Lemma \ref{joint_meas}. As a consequence, for every $s\in[0,T)$, $\eta \in L^0(\mathcal{F}_s)$ and $t\in [s,T]$, the 
	random variable $X_t^{s,\eta(\cdot)}(\cdot)$ is $\mathcal{F}_t-$measurable. Denote by $X^{s,\eta}$ the process defined by $X^{s,\eta}_t(\omega)=X^{s,\eta(\omega)}_t(\omega),\,\omega\in\Omega,\,t\in [s,T]$; by the same arguments as those used to prove \eqref{prova_questo}, with $\eta$ instead of $x$, we deduce that $X^{s,\eta}$ solves \eqref{SDEJ} with initial condition $(s,\eta)$. In particular, for every $0\le s <u<t\le T$ and $x\in\mathbb{R}^d$, thanks to the pathwise uniqueness of \eqref{SDEJ} we infer the existence of an a.s. event $\Omega_{s,u,x}$ such that  
	\[
			X^{s,x}_t(\omega)=X^{u,X^{s,x}_u(\omega)}_t(\omega),\quad t\in (u,T],\, \omega\in\Omega_{s,u,x}.
	\]
	Since $X^{s,x}_t$ satisfies Point \ref{regularity} in Definition \ref{sharpala}, we can proceed as in Corollary \ref{cor_flow} and Lemma \ref{common} to obtain \eqref{flow_finale}, i.e., to establish the previous equation in an a.s. event not depending on $s,\,u$ and $x$.
	
	The proof is now complete.
\end{proof}   

		\section{The Dynamic Programming Principle}
		\label{sec_DPP}
	The aim of this section is to state Theorem \ref{DPP_t}: a dynamic programming principle for controlled SDEs. Its proof is presented in the next Section \ref{sec_proofDPP} and also employs the regularity properties of the sharp stochastic flow constructed in the previous sections, see Theorem \ref{main1}. In particular, the flow property  of the regular controlled solution (see Point \ref{flow_1} of Definition \ref{sharpala}) obtained in Lemma \ref{regular_control} below  will be important in Subsection \ref{sub_firstDPP}, where the first part of the DPP is proved.
	\vspace{1mm}\\
	According to \cite[Chapter 4]{Kalle}, in this section and the next one, we suppose that the measurable space $(U,\mathcal{U})$ is a Polish space, with $\mathcal{U}$ being the Borel $\sigma-$algebra.	
	Moreover, in these parts we use extensively
	\begin{equation} \label{filtr2}
	\mathcal{F}_{t_1,t_2}^{W,N_p}
	=
	\sigma\left(\left\{W_t-W_{t_1},\,N_p\left(\left(t_1,t\right]\times E\right),\,t\in \left[t_1,t_2\right],\,E\in\mathcal{U}\right\}\cup \mathcal{N}\right), \;\;\; 0\le t_1\le t_2\le T,
	 \end{equation}
	where $\mathcal{N}$ is the family of negligible events in $(\Omega,\mathcal{F},\mathbb{P})$: we set $\mathbb{F}^{W,N_p}_{t_1}=(\mathcal{F}^{W,N_p}_{t_1,t})_{t\in[t_1,T]}$ (cf. \eqref{filtr1}). Notice that 
	$\mathbb{F}^{W,N_p} = \mathbb{F}_0^{W,N_p}=(\mathcal{F}^{W,N_p}_{0,t})_{t\in[0,T]}$ is the augmented filtration generated by $W$ and $N_p$. 
	  
	\vskip 2mm
 
 \subsection {Controlled SDEs with general predictable controls}\label{sub_controlledg} 
  Given $l\in\mathbb{N}$, fix a closed convex set $\mathbf{C} \subset {\mathbb {R}}^l$ and a countable  set $\mathbf{Q} \subset \mathbf{C}$ which is dense in $\mathbf{C}$.  We denote by   $\text{proj}_{\mathbf{C}}$    the usual  orthogonal projection
 \begin{gather}\label{projection}
 	{\text{proj}_{\mathbf{C}}}  : \R^l \to \mathbf{C}.   
 \end{gather}  
 We recall that ${\text{proj}_{\mathbf{C}}}$ is a $1-$Lipschitz continuous map, i.e., $|{\text{proj}_{\mathbf{C}}}(x)-{\text{proj}_{\mathbf{C}}}(y)|\le |x-y|,\,x,y\in\mathbb{R}^l$. In the sequel, we write ${\text{proj}_{\mathbf{C}}}(x)$ or ${\text{proj}_{\mathbf{C}}}\, x$ when no confusion may arise. 
 
 We consider  controlled SDEs of the form
 \begin{multline}\label{anchelei} 
	 	X^{s,x, a}_{t} = x+ \int_{s}^{t}b\left(r, X_{r}^{s,x, a},a_r \right) dr + 
	 	\int_{s}^{t}\alpha\left(r,X_{r}^{s,x, a} ,a_r  \right) dW_r\\+
	 	\int_{s}^{t}\!\int_{U_0}g\left(X_{r-}^{s,x, a} ,r,z,a_r\right)\widetilde{N}_p\left(dr,dz\right) 
	 	+ 
	 	\int_{s}^{t}\!\int_{U\setminus U_0 } f\left(X_{r-}^{s,x, a} ,r,z,a_r \right){N}_p\left(dr,dz\right),
	 \end{multline}
where the control $a\colon[0,T]\times \Omega\to \bb$ belongs to the space of $\mathbb{F}^{W,N_p}-$predictable processes $\mathcal{P}_T$,  namely
\[
	\mathcal{P}_T=\left\{a\colon[0,T]\times \Omega\to \bb\,\text{ s.t. $a$ is $\mathbb{F}^{W,N_p}-$predictable} \right\}.
\]
Note that the general controls in $\mathcal{P}_T$ correspond to the progressively measurable controls used in the monograph \cite{Kr} on controlled SDEs driven by Brownian motion.\\
We impose the following assumptions on the diffusion and jump coefficients of \eqref{anchelei}.
\begin{hyp} \label{base22}
The maps $b,\,\alpha,\, g$ and $f$ of \eqref{anchelei} are as those introduced in Hypothesis \ref{base}, but  also depend on an additional variable $\mathbf{y}\in\bb$ representing the control, and  are  jointly measurable in their domains. We require that $f\colon \mathbb{R}^d\times [0,T] \times U\times \bb \to \mathbb{R}^d$ is continuous in the first and last arguments. Moreover, we suppose that $b, \, \alpha, \,g$ satisfy the linear growth and Lipschitz--type conditions in \eqref{lg1}-\eqref{lip1}, uniformly in $\mathbf{y}\in\bb$, and are continuouos in the control variable $\mathbf{{y}}.$
\end{hyp}  
Under Hypothesis \ref{base22}, for every $a\in\mathcal{P_T}$, we denote by  $X^{s,x,a}$  the pathwise unique strong solution
of the controlled SDE \eqref{anchelei}. Such solution has  càdlàg paths, ${\mathbb P}-$a.s.,  for every fixed $s,x,a$;   see, for instance,  \cite[Sections 3.1 and 3.5]{KU04}.   
  With our notation, 
 $X^{s,x,a}$ is $\mathbb{F}^{W, N_p}-$adapted, see also \eqref{strong}. 
 \\
 On the measurable space $([0,T]\times \Omega, \mathcal{B}_T\otimes\mathcal{F})$, where $\mathcal{B}_T$ are the Borelian sets of $[0,T]$, we define the probability measure $\rho(dt,d\omega)$ by $\rho=\frac{1}{T}dt\otimes \mathbb{P}$, i.e., $\rho$ is the normalized product measure on $[0,T] \times \Omega$. Given a sequence $(a_n)_n\subset \mathcal{P}_T$ and $a\in\mathcal{P}_T$, we say that 
\begin{equation}\label{mu-conv}
	a_n\overset{\rho}{\longrightarrow}a \quad \text{ if and only if }\quad \lim_{n\to\infty}\rho(|a_n-a|>\epsilon)=0, \text{ for every } \epsilon>0.
\end{equation} 
The next theorem gives a stability result  for the solution $X^{s,x,a}$ of \eqref{anchelei} with respect to the $\rho-$convergence of controls in \eqref{mu-conv}. 
\begin{theorem}\label{estrap1}
	Fix $s\in[0,T)$ and $x\in\mathbb{R}^d$. Then, for every $(a_n)_n\subset \mathcal{P}_T$ and $a\in\mathcal{P}_T$ such that $a_n\overset{\rho}{ \longrightarrow}a$ in the sense of \eqref{mu-conv}, one has
	\begin{equation}\label{eq_extended}
		\lim_{n\to\infty} 	\mathbb{P}\left(\sup_{0\le t\le  T}\left|X^{s,x,a}_t-X^{s,x,a_n}_t\right|>\epsilon\right)=0,\quad \epsilon >0.
	\end{equation}
\end{theorem}
\begin{proof}
	Let $s\in[0,T)$, $x\in\mathbb{R}^d$, $a\in\mathcal{P}_T$, and consider a sequence $(a_n)_n\subset \mathcal{P}_T$ such that $a_n\overset{\rho}{{\longrightarrow} } a$ according to \eqref{mu-conv}. In order to obtain \eqref{eq_extended}, our objective is to prove by induction that
	\begin{equation}\label{induction_diag}
		\lim_{n\to\infty} 	\mathbb{P}\left(\sup_{0\le t\le  \tau_m\wedge T}\left|X^{s,x,a}_t-X^{s,x,a_n}_t\right|>\epsilon\right)=0,\quad \epsilon >0,\text{ for all }m\in \mathbb{N}.
	\end{equation}
	Here, as in the proof of Theorem \ref{thm_big}, $(\tau_m)_{m\in \mathbb{N}}$ is the increasing sequence of arrival times for the jumps of the Poisson process 	$(P_t)_{t\ge0}$ given by  $P_t=N_p((0,t]\times (U\setminus U_0)),\,t>0$. Recall that $\tau_n\neq T$ for all $n\in \mathbb{N},\,\mathbb{P}-$a.s. 
	\\
	For the case $m=1$, we first show that 
	\begin{equation}\label{m=1}
		\lim_{n\to\infty} 	\mathbb{E}\left[\sup_{0\le t< \tau_1\wedge S_N\wedge T}\left|X^{s,x,a}_t-X^{s,x,a_n}_t\right|^2\right]=0,\quad N\in\mathbb{N},
	\end{equation}
	where, for any  $N\in\mathbb{N}$, we define $$
	S_N=\inf\{t\in[s,T] \text{ such that }X^{s,x,a}_t> N\},\quad \text{ with }\quad \inf\emptyset=\infty.$$
	Note that the $\sup$ in \eqref{m=1} is taken over the (random) interval $[0,\tau_1\wedge S_N\wedge T)$, which is right--open to exclude the first large jump of the solution processes. Throughout the proof, we denote by $C$ a positive constant which is allowed to change from line to line. 
	For every $n\in\mathbb{N}$, from \eqref{anchelei} and the fact that we set, $\mathbb{P}-$a.s.,  $X^{s,x,a}_u$ and $X^{s,x,a_n}_u$ equal to $x$  for $u\le s$, we have 
	\begin{align*}
	&	\sup_{ 0\le  u<  t\wedge \tau_1 \wedge S_N}\Big|X^{s,x, a}_{u}-X^{s,x, a_n}_{u}\Big|^2 = \,C\sup_{ s \le  u\le  t }\bigg(\left|\int_{s}^{u}1_{[0,\tau_1\wedge S_N]}(r)\left(b\left(r, X_{r}^{s,x, a},a_r\right) -b\left(r, X_{r}^{s,x, a_n},a_{n,r}\right)\right) dr\right|^2 \\&\qquad + 
		\left|	\int_{s}^{u}1_{[0,\tau_1\wedge S_N]}(r)\left(\alpha\left(r,X_{r}^{s,x, a} ,a_r  \right) -\alpha\left(r, X_{r}^{s,x, a_n},a_{n,r}\right)\right)dW_r\right|^2\\&\qquad +
		\left|\int_{s}^{u}\!\int_{U_0}1_{[0,\tau_1\wedge S_N]}(r)\left(g\left(X_{r-}^{s,x, a} ,r,z,a_r\right)-g\left(X_{r-}^{s,x, a_n} ,r,z,a_{n,r}\right)\right)\widetilde{N}_p\left(dr,dz\right) 
		\right|^2	\bigg),
	\end{align*}
	which holds for any $t\in [s,T]$, $\mathbb{P}-$a.s.
	Taking expectations, by the Lipschitz--type conditions on the coefficients (see Hypothesis \ref{base22}), the Burkholder--Davis--Gundy inequality and  \cite[Theorem $2.11$]{KU04} we obtain
	\begin{align}\label{similarly}
		\notag	\mathbb{E}	\left[\sup_{ 0 \le  u<  t\wedge \tau_1 \wedge S_N}\Big|X^{s,x, a}_{u}-X^{s,x, a_n}_u\Big|^2\right]
		&\le 
		C\bigg(\mathbb{E}\bigg[\int_{0}^{t}1_{[0,\tau_1\wedge S_N]}(r) \left|X_{r}^{s,x, a}-X_{r}^{s,x, a_n}\right|^2dr\bigg]+
		\mathbb{E}[J^{1,N}_n] 
		\bigg)
		\\&\le 
		C\left(\mathbb{E}\bigg[\int_{0}^{t}\sup_{ 0 \le  u<  r\wedge \tau_1 \wedge S_N} \left|X_{u}^{s,x, a}-X_{u}^{s,x, a_n}\right|^2dr+
		\mathbb{E}[J^{1,N}_n] \bigg]\right),
	\end{align}
	where in the second inequality we also use the fact that $X^{s,x,a}$ and $X^{s,x,a_n}$ are càdlàg, hence their trajectories have at most countable discontinuities. In the previous equation, the quantity $J^{1,N}_n$ is defined as follows:
	\begin{align*}
		J^{1,N}_n=&\int_{s}^{t\wedge \tau_1\wedge S_N}\left(\left|b\left(r, X_{r}^{s,x, a},a_r\right) -b\left(r, X_{r}^{s,x, a},a_{n,r}\right)\right|^2+
		\left|\alpha\left(r,X_{r}^{s,x, a} ,a_r  \right) -\alpha\left(r, X_{r}^{s,x, a},a_{n,r}\right)\right|^2\right) dr  \\&+
		\int_{s}^{t\wedge \tau_1\wedge S_N}\!\left(\int_{U_0}\left|g\left(X_{r-}^{s,x, a} ,r,z,a_r\right)-g\left(X_{r-}^{s,x, a} ,r,z,a_{n,r}\right)\right|^2\nu(dz)\right) dr.
	\end{align*}
	By the linear growth of the coefficients $b,\, \alpha$ and $g$, and the fact that $a_n\overset{\rho}{\longrightarrow}a$ as $n\to\infty$, we can apply the dominated convergence theorem to deduce that $\mathbb{E}[J_n^{1,N}]=\text{o}(1)$ as $n\to \infty$. Thus, from \eqref{similarly}, Fubini's theorem and Gronwall's lemma entail \eqref{m=1}.
	\\
	We now focus on the time $ \tau_1\wedge S_N\wedge T$ and compute, by \eqref{anchelei}, $\mathbb{P}-$a.s.,
	\begin{align}\label{serve_dopo1}
		\notag&	\bigg|
		X^{s,x,a}_{ \tau_1\wedge S_N\wedge T}-X^{s,x,a_n}_{ \tau_1\wedge S_N\wedge T}
		\bigg|\le
		\bigg|	\int_{s}^{T}1_{[0,\tau_1\wedge S_N]}(r)\left(b\left(r, X_{r}^{s,x, a},a_r\right) -b\left(r, X_{r}^{s,x, a_n},a_{n,r}\right)\right) dr\bigg|
		\\\notag&\qquad +
		\bigg|	\int_{s}^{T}1_{[0,\tau_1\wedge S_N]}(r)\left(\alpha\left(r,X_{r}^{s,x, a} ,a_r  \right) -\alpha\left(r, X_{r}^{s,x, a_n},a_{n,r}\right)\right)dW_r\bigg|
		\\\notag&\qquad +
		\bigg|\int_{s}^{T}\!\int_{U_0}1_{[0,\tau_1\wedge S_N]}(r)\left(g\left(X_{r-}^{s,x, a} ,r,z,a_r\right)-g\left(X_{r-}^{s,x, a_n} ,r,z,a_{n,r}\right)\right)\widetilde{N}_p\left(dr,dz\right) 
		\bigg|\\&\qquad +
		1_{\{s<\tau_1 \le S_N\wedge T\}}
		 \bigg|f\left(X^{s,x, a }_{(\tau_1\wedge S_N\wedge T)-},\tau_1,p_{\tau_1},a_{\tau_1\wedge T} \right)
		- f\left(X^{s,x, {a_n} }_{(\tau_1\wedge S_N\wedge T)-},\tau_1,p_{\tau_1},a_{n,\tau_1\wedge T}\right)\bigg|.
	\end{align}
	By \eqref{m=1} and with arguments similar to those employed in \eqref{similarly}, the first three addends in the right--hand side of the previous equation converge to $0$ in $L^2(\Omega)$.
	\\
	As for the last addend in \eqref{serve_dopo1}, we aim to show that it converges to $0$ in probability. To do this, we use the following fact. Let $\phi (y) = y/(y+1),\,y\ge0$, $\eta$ be a  $[0,1]-$valued random variable and $(Y_n)_n$ be a sequence of $d-$dimensional random variables. Then $\lim_{n\to\infty}(\eta Y_n)=0$ in probability if and only if
	\begin{equation}\label{prob}
		\mathbb{E}[\eta\phi(|Y_n|)] \to 0 \quad  \text{as $n \to \infty.$}
	\end{equation}
	We only demonstrate the sufficient condition. Let ${(\eta Y_{n_k})}_k$ be any subsequence of ${(\eta Y_n)}_n$. Since
	\eqref{prob} holds, there exists a further subsequence $(\eta\phi(|Y_{n_{k_j}}|))_j$ which converges to $0$ in an a.s. event $\Omega'$. \\
	We have to prove that  $\lim_{j\to\infty}  |\eta Y_{n_{k_j}}|=0$, $\mathbb{P}-$a.s. This is obvious on the event $\{ \eta =0\}$. On the other hand, for every $\omega\in \{ \eta >0\}\cap \Omega'$, 
	\begin{equation*}
		\lim_{j\to\infty} \phi(|Z_{n_{k_j}}(\omega)| ) = 0,\quad \text{ which implies that }\quad  \lim_{j\to\infty} |Z_{n_{k_j}}(\omega)|=0,
	\end{equation*}
	by the continuity of the inverse map $\phi^{-1}$ and the fact that  $\phi^{-1}(0)=0$. This shows that \eqref{prob} is a sufficient condition for $\lim_{n\to\infty}(\eta Y_n)=0$ in probability, as claimed. 
	\vspace{2mm}\\
	Going back to the last addend in \eqref{serve_dopo1}, we define $\eta= 1_{\{s<\tau_1\le S_N\wedge T\}}$ and consider
	\begin{align}\label{similarly2}
		\notag&\mathbb{E}\left[
		\eta \,\phi \left( \left|f\left(X^{s,x, a }_{(\tau_1\wedge S_N\wedge T)-},\tau_1,p_{\tau_1},a_{\tau_1\wedge T} \right)
		- f\left(X^{s,x, {a_n} }_{(\tau_1\wedge S_N\wedge T)-},\tau_1,p_{\tau_1},a_{n,\tau_1\wedge T}\right)\right| \right) 
		\right]\\
		&\notag
		\qquad=
		\mathbb{E}\bigg[
		\int_{s}^{T}\!\int_{U \setminus U_0}
		1_{[0, \tau_1\wedge S_N]}(r)
		\phi\bigg(\bigg| f\left(X^{s,x, a }_{r-},r,z,a_{r} \right)
		- f\left(X^{s,x, {a_n} }_{r-},r,z,a_{n,r}\right)\bigg|
		\bigg) 	N_p(dr,dz)\bigg]
		\\&
		\qquad =
		\mathbb{E}\bigg[
		\int_{s}^{T}\!\int_{U \setminus U_0}
		1_{[0, \tau_1\wedge S_N]}(r)
		\phi\bigg(\bigg| f\left(X^{s,x, a }_{r-},r,z,a_{r} \right)
		- f\left(X^{s,x, {a_n} }_{r-},r,z,a_{n,r}\right)\bigg|
		\bigg) 	\nu(dz)\,dr\bigg],
	\end{align}
	where the last equality holds because $dt\otimes \nu(dz)$ is the compensator of $N_p(dt,dz)$. Recalling the continuity of $f$ in the first and last arguments, the fact that $\nu(U\setminus U_0)<\infty$ and \eqref{m=1}, the right--most hand side of the previous equation goes to $0$ by the dominated convergence theorem. 
	Consequently, 
	\begin{equation*}
		\lim_{n\to\infty} 	\mathbb{P}\left(\left|X^{s,x,a}_{ \tau_1\wedge S_N\wedge T}-X^{s,x,a_n}_{ \tau_1\wedge S_N\wedge T}\right|>\epsilon\right)=0,\quad \epsilon >0,\,N\in\mathbb{N}.
	\end{equation*}
	Combining the previous equation with \eqref{m=1} we obtain 
	\begin{equation*}
		\lim_{n\to\infty} 	\mathbb{P}\left(\sup_{0\le t\le  \tau_1\wedge S_N\wedge T}\left|X^{s,x,a}_t-X^{s,x,a_n}_t\right|>\epsilon\right)=0,\quad \epsilon >0,\,N\in \mathbb{N}.
	\end{equation*}
	We now prove that this equation holds taking the $\sup$ on the interval $[0,\tau_1\wedge T]$, so that it reduces to \eqref{induction_diag} with $m=1$, i.e., 
	\begin{equation}\label{add2}
		\lim_{n\to\infty} 	\mathbb{P}\left(\sup_{0\le t\le  \tau_1\wedge T}\left|X^{s,x,a}_t-X^{s,x,a_n}_t\right|>\epsilon\right)=0,\quad \epsilon >0.
	\end{equation}  Since convergence in probability implies $\mathbb{P}-$a.s. convergence up to a subsequence, for any $(a_{n_k})_k\subset (a_n)_n$, we use a diagonalization argument to construct a further subsequence $(a_{n_{k_j}})_j\subset (a_{n_k})_k$ such that, for every $\omega$ in an a.s. event $\Omega_0$,
	\[
	\lim_{j\to\infty} 	\sup_{0\le t\le  \tau_1(\omega)\wedge S_N(\omega)\wedge T}\left|X^{s,x,a}_t(\omega)-X^{s,x,a_{n_{k_j}}}_t(\omega)\right|=0,\quad N\in\mathbb{N}.
	\]
	We notice that $\lim_{N\to\infty}S_N= \infty,\,\mathbb{P}-$a.s., as $X^{s,x,a}$ is a c\`adl\`ag process. Hence we can assume -- without loss of generality -- that $S_N\to \infty$ as $N\to \infty$ in $\Omega_{0}$. Therefore, for any $\omega\in\Omega_0$, there exists an $N(\omega)$ sufficiently large such that $S_N(\omega)\wedge T=T$, whence 
	\[
	\lim_{j\to\infty} 	\sup_{0\le t\le  \tau_1(\omega)\wedge T}\left|X^{s,x,a}_t(\omega)-X^{s,x,a_{n_{k_j}}}_t(\omega)\right|=0,\quad \omega\in\Omega_0.
	\]
	This shows that \eqref{add2} is true.

	For the inductive step, we suppose that \eqref{induction_diag} holds for some $m\ge 1$, and aim to prove that
	\begin{equation}\label{inductive_step}
		\lim_{n\to\infty} 	\mathbb{P}\bigg(\sup_{0\le t\le  \tau_{m+1}\wedge S_N\wedge  T}\left|X^{s,x,a}_t-X^{s,x,a_n}_t\right|>\epsilon\bigg)=0,\quad \epsilon >0,\,N\in\mathbb{N}.
	\end{equation} 
	Arguing as we have just done for \eqref{add2}, \eqref{inductive_step} continues to hold computing the $\sup$ on the interval $[0,\tau_{m+1}\wedge T]$. \\ 
	For every $\delta>0$ and $n\in\mathbb{N}$, we define the stopping time $\sigma^{}_n(\delta)$ by
	\[
	\sigma^{}_n(\delta)=\inf\bigg\{
	t\in [s,T] : 	\int_{s}^{t}\!\int_{U\setminus U_0 }\left|f\left(X_{r-}^{s,x, a} ,r,z,a_r \right)- f\left(X_{r-}^{s,x, a_n} ,r,z,a_{n,r} \right)\right|{N}_p\left(dr,dz\right)> \delta
	\bigg\},
	\]
	{with }$\inf\emptyset=\infty$. Thanks to the inductive hypothesis (see \eqref{induction_diag}) and following a dominated convergence argument analogous to the one in \eqref{similarly2},
	\begin{align*}
		&	\mathbb{P}\left(\sigma_n^{}(\delta)< \tau_{{m+1}}\wedge T\right)=
		\mathbb{P}\left(
		\sigma_n^{}(\delta)= \tau_{{i}}\wedge T,\,\text{for some }i=1,\dots,m\right)
		\\&\qquad
		\le \mathbb{P}\bigg(
		\sum_{i=1}^{m}1_{\{s<\tau_i\le T\} }\left|f\left(X^{s,x, a }_{(\tau_i\wedge T)-},\tau_i,p_{\tau_i},a_{\tau_i\wedge T} \right)
		- f\left(X^{s,x, {a_n} }_{(\tau_i\wedge T)-},\tau_i,p_{\tau_i},a_{n,\tau_i\wedge T}\right)\right|
		> \delta\bigg)
		\underset{n\to \infty}{\longrightarrow} 0.
	\end{align*}
	By analogy with the previous step of the proof, for any $N\in\mathbb{N}$, we focus on 
	$$\mathbb{E}\bigg[\sup_{0\le t< \tau_{m+1}\wedge S_N\wedge \sigma_n(\delta)\wedge T}\Big|X^{s,x,a}_t-X^{s,x,a_n}_t\Big|^2\bigg].$$
	In particular, from \eqref{anchelei} and arguments similar to those employed to infer \eqref{similarly}, Gronwall's lemma entails that 
	\begin{equation}\label{mettiamola}
		\mathbb{E}	\bigg[\sup_{ 0 \le  t<  \tau_{m+1}\wedge S_N\wedge\sigma_n(\delta)\wedge T }\Big|X^{s,x, a}_{t}-X^{s,x, a_n}_t\Big|^2\bigg]
		\le
		C\delta^2+\text{o}(1),\quad \text{as }n\to \infty.
	\end{equation}
	Thus, for every $\epsilon>0$, by Markov's inequality 
	\begin{align}\label{decomposition_delta}
		&\notag\mathbb{P}\bigg(
		\sup_{ 0 \le  t<  \tau_{m+1}\wedge S_N\wedge T}\left|
		X_t^{s,x,a}-X_t^{s,x,a_n}
		\right|
		>\epsilon
		\bigg)\\&\qquad\notag
		=	\mathbb{P}\bigg(
		\sup_{ 0 \le  t<    \tau_{m+1}\wedge S_N\wedge \sigma_n(\delta)\wedge T}\left|
		X_t^{s,x,a}-X_t^{s,x,a}
		\right|
		>\epsilon
		,\,
		\sigma_n(\delta)\ge\tau_{m+1}\wedge T
		\bigg)+
		\mathbb{P}\left(
		\sigma_n(\delta)<\tau_{m+1}\wedge T
		\right)
		\\&\qquad
		\le C\frac{\delta ^2}{\epsilon^2} +\text{o}(1),\quad \text{ as }n\to\infty.
	\end{align}
	Since $\delta>0$ is arbitrary, this estimate gives
	\begin{equation}\label{open_right_m}
		\lim_{n\to\infty}\mathbb{P}\bigg(
		\sup_{ 0 \le  t<   \tau_{m+1}\wedge S_N\wedge T}\left|
		X_t^{s,x,a}-X_t^{s,x,a_n}
		\right|
		>\epsilon
		\bigg)=0,\quad \epsilon>0, \,N\in\mathbb{N},
	\end{equation}
	that is, \eqref{inductive_step} holds when the $\sup$ is taken over the right--open interval $[0,\tau_{m+1}\wedge S_N\wedge T)$. \\
	We finally analyze the closed interval $[0,\tau_{m+1}\wedge S_N \wedge T]$. From \eqref{anchelei}, we deduce that (see also \eqref{serve_dopo1})
	\begin{align*}
		&	\Big|X_{\tau_{m+1}\wedge S_N\wedge\sigma_n(\delta)\wedge T}^{s,x,a}-X_{ \tau_{m+1}\wedge S_N\wedge\sigma_n(\delta)\wedge T}^{s,x,a_n}\Big|
	\\&\qquad 	\le
		\bigg|	\int_{s}^{T}1_{[0,{\tau_{m+1}\wedge S_N\wedge\sigma_n(\delta)}]}(r)\left(b\left(r, X_{r}^{s,x, a},a_r\right) -b\left(r, X_{r}^{s,x, a_n},a_{n,r}\right)\right) dr\bigg|
		\\\notag&\qquad \quad +
		\bigg|	\int_{s}^{T}1_{[0,{\tau_{m+1}\wedge S_N\wedge\sigma_n(\delta)}]}(r)\left(\alpha\left(r,X_{r}^{s,x, a} ,a_r  \right) -\alpha\left(r, X_{r}^{s,x, a_n},a_{n,r}\right)\right)dW_r\bigg|
		\\\notag&\qquad \quad +
		\bigg|\int_{s}^{T}\!\int_{U_0}1_{[0,{\tau_{m+1}\wedge S_N\wedge\sigma_n(\delta)}]}(r)\left(g\left(X_{r-}^{s,x, a} ,r,z,a_r\right)-g\left(X_{r-}^{s,x, a_n^{}} ,r,z,a^{}_{n,r}\right)\right)\widetilde{N}_p\left(dr,dz\right) 
		\bigg|\\&\quad \qquad +
		\sum_{i=1}^{m+1}1_{\{s<\tau_i\le  S_N\wedge T\}} \left|f\left(X^{s,x, a }_{(\tau_i\wedge S_N\wedge T)-},\tau_i,p_{\tau_i},a_{\tau_i\wedge T} \right)
		- f\left(X^{s,x, {a_n^{}} }_{{( \tau_i\wedge S_N\wedge T)}-},\tau_i,p_{\tau_i},a^{}_{n,\tau_i\wedge T}\right)\right|,\quad \mathbb{P}-\text{a.s.}
	\end{align*}
	By \eqref{open_right_m} and the dominated convergence theorem, arguing as in \eqref{similarly2} (with\,$1_{[0,\tau_{{m+1}}\wedge S_N]}$ instead of $1_{[0,\tau_1\wedge S_N]}$), the fourth addend in the right--hand side of the previous equation, namely the sum from $1$ to $m+1$, converges to $0$ in probability as $n\to\infty$. As a result,  by Markov's inequality, \eqref{mettiamola} and the same techniques as in \eqref{similarly}, 
	\begin{align*}
		\mathbb{P}\left(\left|X_{ \tau_{m+1}\wedge S_N\wedge\sigma_n(\delta)\wedge T}^{s,x,a}-X_{ \tau_{m+1}\wedge S_N\wedge\sigma_n(\delta)\wedge T}^{s,x,a_n^{}}\right|>4\epsilon\right)\le
		C\frac{\delta^2}{\epsilon^2}
		+\text{o}(1),\quad \text{as }n\to \infty,\quad \epsilon>0,\,N\in\mathbb{N}.
	\end{align*}
	Therefore, arguing with the same decomposition  as in \eqref{decomposition_delta}, we demonstrate that \eqref{inductive_step} holds. This completes the induction argument, which shows that \eqref{induction_diag} is true for all $m\in \mathbb{N}.$\vspace{2mm}
	\\
	Recall that there exists an a.s. event $\Omega_1$ such that $\tau_n(\omega)\uparrow\infty$ as $n\to \infty$, for all $\omega\in\Omega_1$ (see the proof of Theorem \ref{thm_big}). Thus, one can follow the same diagonalization reasoning as in \eqref{add2} to infer \eqref{eq_extended} from \eqref{induction_diag}. The proof is then complete.
\end{proof}

\subsection{Regular controlled stochastic flow with step controls in $\mathcal A$}\label{sub5.2}
 In the general case $a\in\mathcal{P}_T$, it is not clear whether \eqref{anchelei}  generates a regular stochastic flow according to Definition \ref{sharpala}. This issue, which is relevant for proving the DPP, does not seem to be thoroughly addressed in the literature concerning controlled SDEs with jumps, see Remark \ref{flow1no} for more details.
 However, if we restrict the class of controls,  we can apply the results from the previous sections to obtain a useful regular version of the solution to \eqref{anchelei}, see Lemma \ref{regular_control}.

\vspace{1mm}
We introduce step (or simple) control processes with $\qq-$valued jumps and jump times in a countable dense subset of $[0,T]$ starting at $0$. For every $n\in\mathbb{N}$, we then consider the set $S_n=\{t^{(n)}_i,\,i=0,\dots, 2^n\}$ of dyadic points of $[0,T]$ with mesh $T2^{-n}$, and call
$$
S=\cup_{n\in\mathbb{N}}S_n.
$$ For every $t\in [0,T]$, we define  $t_n^-=\max\{s\in S_n : s\le t\}$ and $t_n^+=\min\{s\in S_n : s>t\}$, with  $T_n^+=\infty$.
We denote by $\mathcal{A}_n$  the space of  $\mathbb{F}^{W,N_p}-$adapted càglàd square--integrable $\mathbf{Q}-$valued step processes $a$ of the form
 \begin{equation}\label{form}
a(t,\omega)=\sum_{i=0}^{2^n-1}Z_{i}(\omega)1_{\left(t^{(n)}_{i},t^{(n)}_{i+1}\right]}(t),\quad t\in[0,T],\,\omega\in \Omega,
\end{equation}
for some $Z=(Z_i)_i\subset L^2(\Omega;\mathbf{Q})$, where $Z_i$ is $\mathcal{F}^{W,N_p}_{0,t^{(n)}_i}-$measurable. 
Here $L^2(\Omega;\mathbf{Q})$ is the subset of $L^2(\Omega;\mathbb{R}^l)$ constituted by (classes of) random variables with values in $\mathbf{Q}$.
Therefore every simple process $a\in\mathcal{A}_n$ is identified by $Z$: we are going to write $a\sim Z$. We denote by 
$$\mathcal{Z}_{n}=\big\{(Z_{i})_{i=0,\dots, 2^n-1}\subset L^2(\Omega;\mathbf{Q})\text{ such that $Z_{i}$ is $\mathcal{F}^{W,N_p}_{0,t^{(n)}_i}-$measurable}\}.$$
In the sequel, we set 
\begin{equation} \label{aaa}
	\mathcal{A}=\cup_{n\in\mathbb{N}}\mathcal{A}_n 
	\;\; \text{ and} \;\;   
	\mathcal{Z}=\cup_{n\in\mathbb{N}}\mathcal{Z}_n.
\end{equation}

%
\begin{lemma} \label{regular_control}
 If $a\in\mathcal{A}$, then there exists a regular stochastic  flow $X^{s,x,a}_t$ generated by \eqref{anchelei} according to Definition~\ref{sharpala}.
\end{lemma}
\begin{proof}
Let $n\in \mathbb{N}$ and suppose that $a\in\mathcal{A}_n$ is given by \eqref{form}. Fix $s\in[0,T)$, $\eta \in L^0(\mathcal{F}_s)$ and a vector $\mathbf{{y}}({n})=(y_i)_{ i=0,\dots, 2^n}\in (\bb)^{2^n-1},$  and consider the SDE 
\begin{align}\label{control_SDE}
	\notag	X^{s,\eta, \mathbf{y}( n)}_{t} &= \eta+ \sum_{i=0}^{ 2^n-1}\Bigg(\int_{s}^{t}1_{\{t_{i}<r\le t_{i+1}\}}b\left(r, X_{r}^{s,\eta, \mathbf{y}( n)},y_i \right) dr + 
	\int_s^t1_{\{t_{i}<r\le t_{i+1}\}}\alpha\left(r,X_{r}^{s,\eta, \mathbf{y}( n)} ,y_i  \right) dW_r\\
	&\notag\qquad +
	\int_s^t\!\int_{U_0}1_{\{t_{i}<r\le t_{i+1}\}}g\left(X_{r-}^{s,\eta, \mathbf{y}( n)} ,r,z,y_i \right)\widetilde{N}_p\left(dr,dz\right) 
	\\	& \qquad + 
	\int_{s}^{t}\!\int_{U\setminus U_0 }1_{\{t_{i}<r\le t_{i+1}\}} f\left(X_{r-}^{s,\eta, \mathbf{y}( n)} ,r,z,y_i \right){N}_p\left(dr,dz\right)\Bigg),\quad t\in [s,T].
\end{align} 
	If we define $\tilde{b}\colon [0,T]\times \mathbb{R}^d\times \left(\bb\right)^{2^n}\to \mathbb{R}^d$ by
\[
\tilde{b}\left(t,x,\mathbf{z}\right)
=
\sum_{i=0}^{2^n-1}1_{\{t_{i}<t\le t_{i+1}\}}b\left(t,x,z_i \right),\quad t\in[0,T], \,x\in\mathbb{R}^d,\,\mathbf{z}=(z_0,\dots,z_{2^n-1})\in\left(\bb\right)^{2^n},
\]
with analogous definitions for the coefficients $\tilde{\alpha},\,\tilde{g}$ and $\tilde{f}$, then \eqref{control_SDE} can be rewritten as
\begin{multline*}
	X^{s,\eta, \mathbf{y}( n)}_{t} = \eta+ \int_{s}^{t}\tilde{b}\left(r, X_{r}^{s,\eta, \mathbf{y}( n)},\mathbf{y}( n) \right) dr + 
	\int_s^t\tilde{\alpha}\left(r,X_{r}^{s,\eta, \mathbf{y}( n)} ,\mathbf{y}( n)  \right) dW_r
	\\+
	\int_s^t\!\int_{U_0}\tilde{g}\left(X_{r-}^{s,\eta, \mathbf{y}( n)} ,r,z,\mathbf{y}( n) \right)\widetilde{N}_p\left(dr,dz\right) 
	+ 
	\int_{s}^{t}\!\int_{U\setminus U_0} \tilde{f}\left(X_{r-}^{s,\eta, \mathbf{y}( n)} ,r,z,\mathbf{y}( n) \right){N}_p\left(dr,dz\right),
\end{multline*} 
for $t\in [s,T]$. Since $\tilde{b}(\cdot,\mathbf{{y}}(n)),\, \tilde{\alpha}(\cdot,\mathbf{{y}}(n)),\,\tilde{g}(\cdot,\mathbf{{y}}(n))$ and $\tilde{f}(\cdot,\mathbf{{y}}(n))$ satisfy the requirements of Theorem \ref{main1}, \eqref{control_SDE}  generates a \emph{sharp stochastic flow} $(X^{S_n,\mathbf{y}({n})})^{s,x}$ in the sense of Definition \ref{sharpala}. When $\mathbf{{y}}({n})$ varies in $(\mathbf{Q})^{2^{n}}$,  the corresponding family of SDEs \eqref{control_SDE} generates a common sharp stochastic flow, which we denote by $(X^{S_n})^{s,x,\mathbf{{y}}(n)}$.	Notice that the map
\begin{equation}\label{measQ}
	X^{S_n}\colon \Omega\times [0,T]\times \mathbb{R}^d\times [0,T]\times (\mathbf{Q})^{2^n}\to \mathbb{R}^d \text{ given by } X^{S_n}(\omega,s,x,t,\mathbf{{y}}(n))=(X^{S_n})_t^{s,x,\mathbf{{y}}(n)}
\end{equation}
is  $ \mathcal{F}_{{}}\otimes \mathcal{B}([0,T]\times \mathbb{R}^d\times [0,T])\otimes \mathcal{P}(\mathbf{Q})^{2^n}-$\text{measurable}, where $\mathcal{P}(\mathbf{Q})^{2^n}$ is the power set of $(\mathbf{Q})^{2^n}$.\\
Recalling \eqref{form} and that $Z\in\mathcal{Z}_n$, if we set $[\mathbf{y}( n)](\omega)=Z(\omega)$, then  arguments similar to those in the proof of Lemma \ref{le_se} yield 
\begin{equation}\label{controlla_flow}
	X_t^{s,x,a}(\omega)=(X^{S_{{n}}})_t^{s,x,[\mathbf{y}( n)](\omega)}(\omega),\quad \omega\in\Omega.
\end{equation}
Therefore \eqref{controlla_flow} gives a version of $X_t^{s,x,a}$ which is a regular  stochastic flow in the sense of Definition \ref{sharpala}. In particular, the stochastic continuity in Point \ref{constoch} can be inferred using equalities like \eqref{nftd}-\eqref{cont_X_st} (see also \eqref{mettere} below) and Point \ref{regularity} \ref{cons} in Definition \ref{sharpala}. The proof is now complete.
\end{proof} 
\begin{rem}
	By \eqref{measQ} and \eqref{controlla_flow} in the proof of Lemma \ref{regular_control}, we, in fact, construct a regular stochastic flow $X^\mathcal{A}$ generated by \eqref{anchelei} which is common to all controls $a\in\mathcal{A}$.
\end{rem} 
 \begin{rem} \label{flow1no}  
	 The existence of a regular stochastic flow  is investigated  for several classes of controlled SDEs without jumps.  For instance,  
a regular version of the continuous  solution is established for different controlled McKean--Vlasov SDEs  in \cite{CP,PW} (see also the references therein). \\
As we explain in Introduction, in the case of SDEs with jumps such as \eqref{anchelei}, establishing  the existence of a regular solution is more challenging. 
	 Lemma \ref{regular_control} achieves this  for  jump diffusions controlled by a process $a\in\mathcal{A}$.
	 To the best of our knowledge, Lemma \ref{regular_control} fills a significant gap in the literature on controlled SDEs with jumps. Indeed, even in the case of controlled SDEs with only a small--jumps component  (i.e., \eqref{anchelei} with $f\equiv 0$),   we could not find any proofs 	 regarding the regularity of the generated stochastic flow. 
%
	  In the special case of  controlled jump--diffusions driven by a Wiener process and an independent compound Poisson process  (hence without a small--jumps component), 	 a regular  stochastic flow  has  been  used in the proof of \cite[Proposition 5.4]{touzi}.  However, the details about its validity  are missing.   
\end{rem}

	The following result -- a corollary to Theorem \ref{estrap1} -- shows that it is possible to approximate, uniformly in probability, the solution $X^{s,x,a}_t$ of \eqref{anchelei} for $a\in\mathcal{P}_T$   with solutions of the same equation corresponding to controls in $\mathcal{A}$.
	\begin{corollary}\label{estrap}
		For every $s\in[0,T),\,x\in\mathbb{R}^d$ and $a\in\mathcal{P}_T$, there exists a sequence $(a_n)_n\subset \mathcal{A}$ such that $a_n\overset{\rho}{\longrightarrow} a$  and 
		\begin{equation}\label{eq_corollary}
				\lim_{n\to\infty} 	\mathbb{P}\left(\sup_{0\le t\le  T}\left|X^{s,x,a}_t-X^{s,x,a_n}_t\right|>\epsilon\right)=0,\quad \epsilon >0.
		\end{equation}
	\end{corollary}
\begin{proof}
	Fix $s\in[0,T),\,x\in\mathbb{R}^d$ and $a\in\mathcal{P}_{T}$. Notice that $a\colon([0,T] \times \Omega,\mathcal{P})\to (\mathbb{R}^l,\mathcal{B}(\mathbb{R}^l))$ is a measurable function, where $\mathcal{P}$ is the $\mathbb{F}^{W,N_p}-$predictable $\sigma-$algebra, hence we can find a sequence of simple predictable processes $(a_{1,n})_n$ such that $a_{1,n}(t,\omega)\to a(t,\omega)$ as $n\to\infty$ for every $(t,\omega)\in [0,T]\times \Omega$. In particular, $a_{1,n}\overset{\rho}{\longrightarrow}a$. \\
	Observe that  $(a_{1,n})_n$ is a sequence in the Hilbert space  $\mathcal{L}^2=(L^2([0,T]\times \Omega;\mathbb{R}^l),\norm{\cdot}_{\mathcal{L}^2})$, where $\norm{\cdot}_{\mathcal{L}^2}$ denotes the usual norm, and that the $\norm{\cdot}_{\mathcal{L}^2}-$convergence is stronger than the $\rho-$convergence. Let $\mathcal{E}$ be the space of step processes of the form \eqref{form}, where the jump times are not necessarily in the dyadic points $S$ and the jumps are $\mathbb{R}^l-$valued square--integrable random variables. By \cite[Exercise 2.7.2 and Theorem 2.8.2]{kry}, we deduce that $\mathcal{E}$ is dense in $\mathcal{L}^2$.			 
	Thus,		 we can determine a sequence $(a_{2,n})_n\subset \mathcal{E}$ such that  
	$$
	\norm{a_{1,n}-a_{2,n}}_{\mathcal{L}^2}<\frac{1}{n},\quad n\in\mathbb{N}.$$
	It follows that $a_{2,n}\overset{\rho}{\longrightarrow}a$. Since $a$ takes values in $\mathbf{C}$, it is clear that $\text{proj}_{\mathbf{C}}\,a_{2,n}\overset{\rho}{\longrightarrow}a$, as well.\\
	Recall the countable dense subset $\mathbf{Q}\subset \mathbf{C},$ and write $\mathbf{Q}=\{q_j\}_{j\in\mathbb{N}}$. For every $n\in\mathbb{N}$, we define the function $\Phi_n\colon \mathbf{C}\to\mathbf{Q}$ as follows: given $x\in\mathbf{C}$, $\Phi_n(x)$ is the element $z_j\in \mathbf{Q}\cap B(x,\frac{1}{n\sqrt{T}})$ with the smallest index $j$, where $B(x,r)\subset\R^l$ is the open ball centered in $x$ with radius $r>0$. Notice that $\Phi_n$ is measurable, as
	\[
	\Phi_n^{-1}(z_j)= \left(B\left(z_j,\frac{1}{n\sqrt{T}}\right)\setminus \bigcup_{k=1}^{j-1}
	B\left(z_k,\frac{1}{n\sqrt{T}}\right)\right)
	\cap C,\quad j\in\mathbb{N}.
	\]
	By the definition of $\mathcal{E}$, for every $n\in\mathbb{N}$, there exists an $(\bar{n}(n)+1)-$tuple $\bar t = (t_0, \dots, t_{\bar{n}(n)})$, with $0=t_0<t_1<\dots <t_{\bar{n}(n)}=T$, and $\mathcal{F}^{W,N_p}_{0,t_i}-$measurable, square--integrable, $\mathbf{C}-$valued random variables $\widetilde{Z}_{i,n}$ such that 
	\begin{gather*}
		\text{proj}_{\mathbf{C}}\,a_{2,n}(t,\cdot )= \widetilde{Z}_{i,n},\quad t\in (t_i,t_{i+1}],\,i=0,\dots,\bar{n}(n)-1.
	\end{gather*}
	Considering the processes $(a_{3,n})_n\subset\mathcal{P}_T$ given by $a_{3,n}(0,\cdot)=0$ and 
	\[
	a_{3,n}(t,\cdot )= \Phi_n(\widetilde{Z}_{i,n}),\quad t\in (t_i,t_{i+1}],\,i=0,\dots,\bar{n}(n)-1,\,n\in\mathbb{N},
	\]
	we have 
	\[
	\norm{\text{proj}_{\mathbf{C}}\,a_{2,n}-a_{3,n}}_{\mathcal{L}^2}<\frac{1}{n},\quad n\in\mathbb{N},
	\]
	which implies that $a_{3,n}\overset{\rho}{\longrightarrow}a$.
	\\
	Define $Z_{i,n}=\Phi_n(\widetilde{Z}_{i,n})$ and $M_n=\max_{i=1,\dots,\bar{n}(n)-1} \mathbb{E}[|Z_{i,n}-Z_{i-1,n}|^2]$. Take  an $(\bar{n}(n)+1)-$tuple of  dyadic points $(q_0, q_1,  \dots, q_{\bar{n}(n)})$, with $0=q_0<q_1<\dots<q_{\bar{n}(n)}=T$, such that 
	$t_i <q_i < t_{i+1}$ and $q_i-t_i<({M}_n{n^2(\bar{n}(n)-1)})^{-1}$ for any $i =1, \ldots, \bar n (n)-1$. Then, defining the control  $a_n\in \mathcal{A}$ by
	\begin{gather*} 
		a_n(t,\omega)=
		\sum_{i=0}^{\bar{n}(n)-1} Z_{i,n}(\omega)1_{\left(q_{i},q_{i+1}\right]}(t),\quad t\in[0,T],\,\omega\in \Omega,
	\end{gather*}
	we conclude that  
	\begin{gather*}\label{L2} 
		\norm{a_{n}-a_{3,n}}_{\mathcal{L}^2}^2=
		\sum_{i=1}^{\bar{n}(n)-1}\int_{t_i}^{q_i}  \mathbb{E}\Big[\left|Z_{i,n} - Z_{i-1,n}\right|^2\Big] dr = \sum_{i=1}^{\bar{n}(n)-1}(q_i-t_i) \mathbb{E}\Big[\left|Z_{i,n} - Z_{i-1,n}\right|^2\Big]<\frac{1}{n^2}.
	\end{gather*}
	Consequently,  the sequence $(a_n)_n\subset \mathcal{A}$	satisfies $a_n\overset{\rho}{\longrightarrow}a$, and  Theorem \ref{estrap1} implies that \eqref{eq_corollary} holds.		
	The proof is now complete. 
\end{proof}

	\subsection{The stochastic control problem and the statement of the DPP} \label{sub_DPP}

  To formulate our stochastic control problem, we introduce the functions $h$ and $j$ that satisfy the following assumption.		
\begin{hyp}		We consider a {\sl jointly measurable map $h\colon[0,T] \times \mathbb{R}^d\times \bb\to \mathbb{R}$ such that $h (t, \cdot) : \mathbb{R}^d\times \bb\to \mathbb{R}$  is continuous for any $t \in [0,T]$},
	  and {\sl a continuous map $j\colon \mathbb{R}^d\to \mathbb{R}$.} 
	We also require that {\sl $h$ and $j$ are globally bounded in their domains. } 
\end{hyp}	
	  
	  We define the gain function
		\begin{equation} \label{h22}
		J(s,x,a)=\mathbb{E}\left[\int_{s}^{T}h\left(r,X_r^{s,x,a}, a_r\right)dr + 
		j\left(X_T^{s,x,a}\right)\right],\quad s\in [0,T),\, x\in\mathbb{R}^d,\,a\in\mathcal{P}_T.
		\end{equation}
The value function $v$ associated with $J$ is 
		\[
		v(s,x)=\sup_{a\in\mathcal{P}_T}\,J(s,x,a),\quad s\in [0,T),\,x\in\mathbb{R}^d.
		\] 
	The following three results illustrate some properties of $J$ and $v$ which are essential for the proof of the DPP stated in Theorem \ref{DPP_t}. The first one concerns  a continuity property of $J$ in the control $a\in\mathcal{P}_T$ with respect to the  $\rho-$convergence in \eqref{mu-conv}.
	\begin{lemma}\label{lemma_cont}
		Fix $s\in [0,T)$ and $x\in\mathbb{R}^d$. Then, given $a\in\mathcal{P}_T$ and  a sequence $(a_n)_n\subset \mathcal{P}_T$ such that $a_n\overset{\rho}{\longrightarrow}a$, 
			\begin{equation}\label{cont_J}
			\lim_{n \to \infty } J(s,x,a_{n}) = J(s,x,a).
		\end{equation}
	\end{lemma}
\begin{proof}
	Consider $s\in [0,T)$, $x\in\mathbb{R}^d$ and controls $a,\,(a_n)_n$ in $\mathcal{P}_T$  such that $a_n\overset{\rho}{\longrightarrow}a$. Recall that 
	\[
	J(s,x,a_n)=\mathbb{E}\bigg[\int_{s}^{T}h\left(r,X_r^{s,x,a_n}, a_{n,r}\right)dr + 
	j\left(X_T^{s,x,a_n}\right)\bigg],\quad n\in\mathbb{N}.
	\]
	By  Theorem \ref{estrap1} and Vitali's convergence theorem, $ \mathbb{E}\left[ j\left(X_T^{s,x,a_n}\right)\right]\to  \mathbb{E}\left[ j\left(X_T^{s,x,a}\right)\right]$ as $n\to \infty$. Regarding the other addend, notice that, by \eqref{mu-conv} and \eqref{eq_extended}, for every subsequence $(a_{n_k})_k\subset (a_n)_n$, there exists a further subsequence $(a_{n_{k_j}})_j$ such that 
	\[
	\lim_{j\to \infty}\left|h\left(r,X_r^{s,x,a_{n_{k_j}}}(\omega), a_{n_{k_j}}(r,\omega)\right)
	-  h\left(r,X_r^{s,x,a}(\omega), a(r,\omega)\right) \right|=0,\quad  \text{for $dt\otimes \mathbb{P}-$a.e. }(r,\omega)\in [0,T]\times \Omega. 
	\]
	Thus, the dominated convergence theorem implies that  
	\begin{gather*}
		\lim_{j \to \infty}\int_{s}^{T}  \mathbb{E} \left[h\left(r,X_r^{s,x,a_{n_{k_j}}}, a_{n_{k_j},r}\right)
		-  h\left(r,X_r^{s,x,a}, a_r\right) \right] dr =0.
	\end{gather*} 
This computation gives \eqref{cont_J}, finishing the proof.
\end{proof}
A straightforward consequence of Lemma \ref{lemma_cont} is the next corollary. It shows that, in the definition of  $v$, it is not restrictive to consider only  controls $a\in\mathcal{A}$, i.e.,  step  processes with $\mathbf{Q}-$valued square--integrable jumps and  jump times in $S$  (the set of dyadic points of $[0,T]$). 
		\begin{corollary}\label{l20}
		The following equality holds for every $s\in[0,T)$ and $x\in\mathbb{R}^d$:
		\begin{equation}\label{semplicemente}
			v(s,x)=\sup_{a\in\mathcal{A}}\,J(s,x,a).
		\end{equation}
	\end{corollary}
	\begin{proof}
		Fix $s\in[0,T)$ and $x\in\mathbb{R}^d$. Thanks to Corollary \ref{estrap}, for every $a \in {\mathcal P}_T$ there exists a sequence $(a_n)_n \subset {\mathcal A}$ such that $a_n\overset{\rho}{\longrightarrow} a$. Hence, by Lemma  \ref{lemma_cont}, $	\lim_{n \to \infty } J(s,x,a_{n}) = J(s,x,a),$ whence
	 $$
	v(s,x)\le \sup_{a\in\mathcal{A}}\,J(s,x,a).
	$$
Since $\mathcal{A}\subset \mathcal{P}_T$, the proof is complete.
	\end{proof}
\begin{rem}\label{ref8}
By Corollary 	\ref{l20}, 
 we can consider the value function $v$ as in \eqref{semplicemente}, i.e., $v$ is computed taking the supremum over controls $a\in\mathcal{A}$. Moreover, using the identification $\mathcal{A}_n\ni a\sim Z\in\mathcal{Z}_n$, $n\in\mathbb{N}$, in the sequel we might also write $J(s,x,Z)=J(s,x,a).$
\end{rem}

Thanks to \eqref{controlla_flow} and the characterization of the value function $v$ in Corollary \ref{l20}, we can exploit the regularity of the flow $X^{S_{ n}}$ generated by \eqref{control_SDE}  to deduce the following property of $v$. 
\begin{lemma}\label{lsc}
	The function $v\colon[0,T)\times \mathbb{R}^d\to \mathbb{R}$ is lower semicontinuous. 
\end{lemma}
\begin{proof}
	Fix $n\in\mathbb{N}$, $s\in [0,T)$, $x\in\mathbb{R}^d$ and $a\in\mathcal{A}_n$. By Points \ref{regularity}  \ref{cons} and \ref{constoch} in Definition \ref{sharpala} (see also \eqref{nftd}-\eqref{cont_X_st}), we infer that there exists an a.s. event $\Omega_s$ such that 
	\begin{equation}\label{mettere}
		(X^{S_n})^{s-,x,\mathbf{y}}_t(\omega)=	(X^{S_n})^{s,x,\mathbf{y}}_t(\omega),\quad x\in\mathbb{R}^d,\,\mathbf{y}\in (\mathbf{Q})^{2^n},\,t\in[0,T],\,\omega\in\Omega_s.
	\end{equation}
	Hence by \eqref{controlla_flow} we infer that 
	\[
	X^{s-,x,a}_t(\omega)=\lim_{r\uparrow s}X^{r,x,a}_t(\omega)=X^{s,x,a}_t(\omega),\quad x\in\mathbb{R}^d,\,t\in[0,T],\,\omega\in\Omega_s.
	\]
	We recall that $X^{s-,x,a}_t=\lim_{r\uparrow s}X^{r,x,a}_t$ holds uniformly in $t$ and locally uniformly in $x$. 
	Thus, if we take a sequence $(s_k,x_k)_k\subset [0,T)\times \R^d$ such that $(s_k,x_k)\to (s,x)$ as $k\to \infty$,  the dominated convergence theorem yields -- also using Point \ref{regularity} \ref{conx} in Definition \ref{sharpala} -- $J(s,x,a)=\lim_{k\to\infty}J(s_k,x_k,a)$. Since $a$ is chosen arbitrarily, recalling \eqref{semplicemente} we deduce that 
	\[
	v(s,x)=\sup_{a\in\mathcal{A}}\lim_{k\to \infty}J(s_k,x_k,a)\le \liminf_{k\to\infty}\sup_{a\in\mathcal{A}}J(s_k,x_k,a)=\liminf_{k\to\infty}v\left(s_k,x_k\right),
	\]
	which completes the proof.
\end{proof}

	We conclude this section by stating the \emph{dynamic programming principle} (or \emph{Bellman's principle}) for our stochastic control problem associated with \eqref{anchelei}.  Its proof is presented in Section \ref{sec_proofDPP}.
		\begin{theorem}\label{DPP_t}
			 Consider \eqref{anchelei} and \eqref{h22}. 
			Fix $s\in [0,T)$ and denote by $\mathcal{T}_{s,T}$ the set of $\mathbb{F}^{W,N_p}-$stopping times taking values in $(s,T)$. Then the following holds:
			\begin{align}\label{dpp1}
				\notag 
				v(s,x)&= \sup_{a\in{\mathcal{P}_T}} \inf_{\theta\in \mathcal{T}_{s,T}}
				\mathbb{E}\bigg[\int_{s}^{\theta}h\left(r,X_r^{s,x,a}, a_r\right)dr
				+v\left(\theta, X_\theta^{s,x,a}\right)\bigg]
				\\
				&=
				\sup_{a\in{\mathcal{P}_T}} \sup_{\theta\in \mathcal{T}_{s,T}}
				\mathbb{E}\bigg[\int_{s}^{\theta}h\left(r,X_r^{s,x,a}, a_r\right)dr
				+v\left(\theta, X_\theta^{s,x,a}\right)\bigg],    \quad s\in [0,T),\,x\in\mathbb{R}^d.				
			\end{align}
		\end{theorem}	
		\begin{rem} \label{pham2}
			 Equation \eqref{dpp1} in  Theorem \ref{DPP_t}  gives a stronger version  of the DPP for controlled SDEs with jumps such as \eqref{anchelei}. Indeed, for this class of equations,  the DPP is typically  formulated  as follows: for any $s\in [0,T),\,x\in\mathbb{R}^d$ and stopping time $\theta\in \mathcal{T}_{s,T}$,
			$
			v(s,x)= \sup_{a\in{\mathcal{P}_T}} 			\mathbb{E}\big[\int_{s}^{\theta}h\left(r,X_r^{s,x,a}, a_r\right)dr
			+v\left(\theta, X_\theta^{s,x,a}\right)\big].\\
			$  For   SDEs without jumps,   see  also  \cite[Remark 3.3.3]{Pham} and \cite[Remark 3.2]{PW}.
					\end{rem} 
					\section {Proof of the Dynamic Programming Principle}\label{sec_proofDPP}
					The objective of this section is to rigorously prove the DPP stated in Theorem \ref{DPP_t}. Specifically, to demonstrate the equalities in \eqref{dpp1}, the strategy that we follow consists in proving separately two inequalities. The first one is    
					\begin{equation}\label{less_general}
						v(s,x)\le  \sup_{a\in{\mathcal{P}_T}} \inf_{\theta\in \mathcal{T}_{s,T}}
						\mathbb{E}\bigg[\int_{s}^{\theta}h\left(r,X_r^{s,x,a}, a_r\right)dr
						+v\left(\theta, X_\theta^{s,x,a}\right)\bigg]
					\end{equation}
					and is shown in Subsection \ref{sub_firstDPP} by employing the regular controlled stochastic flow $X^{s,x,a}_t$ constructed in Subsection \ref{sub5.2}, see Lemma \ref{regular_control}.
					The second inequality 
					\begin{equation}\label{opposite_1}
						v(s,x)\ge \sup_{a\in{\mathcal{P}_T}} \sup_{\theta\in \mathcal{T}_{s,T}}
						\mathbb{E}\bigg[\int_{s}^{\theta}h\left(r,X_r^{s,x,a}, a_r\right)dr
						+v\left(\theta, X_\theta^{s,x,a}\right)\bigg].  
					\end{equation}
				is much more difficult to prove, see Subsection \ref{subsect_2}. In particular, to obtain \eqref{opposite_1}, we use a basic and classical measurable selection theorem from \cite{BP}, which we are able to apply by choosing controls in a suitable class of predictable processes  $\mathcal{B}^s\subset \mathcal{P}_T,\,s\in[0,T]$ (see Subsection \ref{sub_mst}). We believe that our approach is of independent interest and may be useful in proving other DPPs for different classes of controlled SDEs. 
					
					\subsection{On $\mathcal{F}^{W,N_p}_{0,t}-$measurable $\qq-$valued  random variables}\label{mtr}
	In this subsection, we first rewrite an $\mathcal{F}^{W,N_p}_{0,t}-$measurable $\qq$-valued random variable as a function of $W$ and $N_p$, for $t\in[0,T]$. Thanks to this, we introduce square--integrable processes in $\mathcal{P}_T$ of the form \eqref{udio} which we will employ in the next subsections to compute conditional expectations.
	Secondly, we consider separable Hilbert spaces (see Lemma \ref{l_21}) that will be crucial for proving the second part \eqref{opposite_1} of the DPP, see Subsections \ref{sub_mst}-\ref{subsect_2}.
	
		Let $\Omega_W$ be the space of continuous functions from $[0,T]$ to $\mathbb{R}^m$ and, for every $t\in [0,T]$, let  $\mathcal{B}^W_t$ be the smallest $\sigma-$algebra on $\Omega_W$ which makes the projections $\pi_s\colon\Omega_W\to \mathbb{R}^m,\,s\in[0,t],$ measurable. Notice that the Brownian motion $W\colon(\Omega,\mathcal{F}^{W,N_p}_{0,t})\to (\Omega_W,\mathcal{B}^W_t)$ is a measurable map.
	We denote by  $\Omega_N$ the set of integer--valued measures  defined on $E= (0,T]\times  U$ with values in ${\mathbb N} \cup { \{ \infty \}}$.  
	As in \cite{IW}, we endow $\Omega_N$ with the minimal $\sigma-$algebra $\mathcal    G$ which makes measurable all the mappings:
	$\mu  \mapsto \mu (A)$, with $A \in \mathcal{B}\left((0,T]\right)\otimes\mathcal{U}$. 
	The Poisson random measure $N_p$ discussed in the previous sections is a measurable map from 
	$\Omega$  \text{into} $\Omega_N$.
	\noindent For every $t\in [0,T]$, we  consider the minimal $\sigma-$algebra $\mathcal{B}_t^N$ on $\Omega_N$ 
	which makes measurable all the mappings
	$$ 
	\mu \mapsto \mu ((0,s] \times A),\;\;\; s \le t,\;\; A \in \mathcal{U};
	$$
	notice that $N_p\colon (\Omega,\mathcal{F}^{W,N_p}_{0,t})\to (\Omega_N, \mathcal{B}^N_t)$ is measurable.
	
	{ Inspired by the proof of \cite[Lemma 3.11]{SZab}, we now 
	clarify how to rewrite a random variable $Y: \big(\Omega, {\mathcal F}_{0,t}^{W,N_p}\big) \to \qq$ in terms of $W$ and $N_p$}, for all $t\in [0,T]$. We introduce the measurable function $\widetilde{T}_t\colon{ \big(\Omega,\mathcal {F}_{0,t}^{W,N_p}\big)}\to\left( \Omega_W\times\Omega_N ,\mathcal{B}_t^W\otimes \mathcal{B}_t^N\right)$ defined by
	\[
	\widetilde{T}_t(\omega)=\left(W(\omega), N_p(\omega)\right),\quad \omega\in \Omega.
	\]
	In the sequel, we write $\left(\Omega_\times, \mathcal{B}^{\times}_t\right)$ instead of $\left( \Omega_W\times\Omega_N ,\mathcal{B}_t^W\otimes \mathcal{B}_t^N\right)$ to keep the notation short. Denoting by $\mathbb{F}^{W,N_p,0}=(\mathcal{F}_t^{W,N_p,0})_{t\in [0,T]}$ the natural filtration generated by $W$ and $N_p$, namely
	\[
	\mathcal{F}^{W,N_p,0}_t=
	\sigma\left(\left\{W_r,\,N_p\left(\left(0,r\right]\times E\right),\,r\in \left[0,t\right],\,E\in\mathcal{U}\right\}\right),\quad t\in [0,T],
	\]
	we observe that
	\begin{equation}\label{tutto}
		\sigma\big(\widetilde{T}_t\big)=\mathcal{F}_{t}^{W,N_p,0}.
	\end{equation}
	As a consequence, 
	\cite[Theorem 1.7]{YZ}  yields the existence of a measurable mapping 
	$$
	F: \left(\Omega_\times,\mathcal{B}_t^\times\right)\to \mathbf{Q}\;\;\text{ such  that }\; Y(\omega) = F \big( \widetilde{T}_t(\omega)\big),\text{ for $\mathbb{P}-$a.s. }\omega\in \Omega.
	$$

	{ \color{black}{Consider an $\mathbb{F}^{W,N_p}-$stopping time $\theta$ with values in $[0,T]$} and denote by $\theta_s=\theta\wedge s$, for every $s\in [0,t]$. Note that, given $(w,\mu) \in \Omega_\times$, one can write, for every $\omega\in\Omega$,
		\begin{align*}
			&w=w(\theta_s\left(\omega\right)\wedge \cdot) + \left[w (\theta_s\left(\omega\right) \vee \cdot)-w(\theta_s\left(\omega\right)) \right]
			, 
			\qquad 
			\mu= \mu \left(\cdot\cap \left(\left(0,\theta_s(\omega)\right]\times U\right)\right)
			+ \mu \left(\cdot\cap \left(\left(\theta_s(\omega),T\right]\times U\right)\right)
			.
		\end{align*}
		It follows that, for every $\omega\in\Omega$,
		\begin{multline*}
			F\big(\widetilde{T}_t(\omega)\big) = F\left(W(\omega),N_p(\omega)\right)=
			F\left(
			W_{\theta_s\left(\omega\right)\wedge \cdot}(\omega) + \left[W _{\theta_s\left(\omega\right) \vee \cdot}-W_{\theta_s\left(\omega\right)} \right](\omega)
			,\right.
			\\\left.
			N_p	\left(\cdot\cap \left(\left(0,\theta_s(\omega)\right]\times U\right)\right)(\omega)
			+ N_p \left(\cdot\cap \left(\left(\theta_s(\omega),T\right]\times U\right)\right)(\omega)
			\right).
		\end{multline*}  
		We observe that the random variable from $(\Omega,\mathcal{F}_{0,t}^{W,N_p})$ to $\left(\Omega_\times, \mathcal{B}_t^\times\right)$ defined by
		\[
		\omega\mapsto \left(W_{\theta_s\left(\omega\right)\wedge \cdot}(\omega), 
		N_p	\left(\cdot\cap \left(\left(0,\theta_s(\omega)\right]\times U\right)\right)(\omega)
		\right)
		\]
		is measurable with respect to $\mathcal{F}^{W,N_p}_{\theta_s}$, the $\sigma-$algebra generated by the stopping time $\theta_s$ relative to the filtration $\mathbb{F}^{W,N_p}$.
		On the other hand, the $\left(\Omega_\times, \mathcal{B}_t^\times\right)-$valued random variable 
		\[
		\omega\mapsto \left(
		\left[W _{\theta_s\left(\omega\right) \vee \cdot}-W_{\theta_s\left(\omega\right)} \right](\omega)
		,
		N_p \left(\cdot\cap \left(\left(\theta_s(\omega),T\right]\times U\right)\right)(\omega)
		\right)
		\]
		is independent of $\mathcal{F}^{W,N_p}_{\theta_s}$. Therefore, { if $Y$ is integrable, }we can compute  its conditional expectation  with respect to $\mathcal{F}^{W,N_p}_{\theta_s}$ to deduce that, for $\mathbb{P}-$a.s. $\omega\in\Omega$,
		\begin{align*}
			\mathbb{E}\!\left[Y\big| {\mathcal F}_{\theta_s}^{W,N_p} \right]\!(\omega)\!
			&= \mathbb{E}\left[F\big(\widetilde{T}_t\big)\big| {\mathcal F}_{\theta_s}^{W,N_p}\right](\omega)
			\\
			&= \mathbb{E}\left[F\left(W_{\theta_s\left(\omega\right)\wedge \cdot}(\omega) \!+ \left[W _{\theta_s\vee \cdot}\!-W_{\theta_s} \right]\!, N_p	\left(\cdot\cap \left(\left(0,\theta_s(\omega)\right]\times\! U\right)\right)\!(\omega)
			+ N_p \left(\cdot\cap \left(\left(\theta_s,T\right]\!\times U\right)\right)\right)\right]\!.
		\end{align*}
		
		In particular, given $\mathcal{A}_n\ni a \sim Z\in \mathcal{Z}_n$, for every $i=0,\dots, 2^n-1$, $n\in\mathbb{N}$, there exists a measurable function
		\[
		F^{(n)}_i: \Big(\Omega_\times,\mathcal{B}_{t_i^{(n)}}^\times\Big)\to \mathbf{Q}\;\;\text{ such  that }\; Z_i(\omega) = F^{(n)}_i \left( \widetilde{T}_{t_i^{(n)}}(\omega)\right),\text{ for $\mathbb{P}-$a.s. }\omega\in \Omega.
		\]
		Moreover,  for every $\mathbb{F}^{W,N_p}-$stopping time $\theta$ we have, $\mathbb{P}-$a.s., 
		\begin{multline*}
			\mathbb{E}\left[Z_i\big| {\mathcal F}_{\theta_{s}}^{W,N_p} \right](\omega)
			=
			\mathbb{E}\Big[F^{(n)}_i\left(W_{\theta_{s}\left(\omega\right)\wedge \cdot}(\omega) + \left[W _{\theta_{s} \vee \cdot}-W_{\theta_{s}} \right],\right. \\
			\left. N_p	\left(\cdot\cap \left(\left(0,\theta_{s}(\omega)\right]\times U\right)\right)(\omega)
			+ N_p \left(\cdot\cap \left(\left(\theta_{s},T\right]\times U\right)\right)\right)\Big],	\quad s\in\big[0,t_i^{(n)}\big].	
		\end{multline*}
		Given $\bar{\omega}\in \Omega$ and $s\in [0,T]$, we denote by $a_{s,n}^{{\theta},\bar{\omega}}$ the following  $\mathbb{F}^{W,N_p}-$adapted $\qq-$valued simple process obtained from~$a$:
		\begin{multline}\label{udio}
			a_{s,n}^{{\theta},\bar{\omega}}(t,\omega)=1_{\left\{t\le s_n^+\right\}}a(t,\bar{\omega})
			+	1_{\left\{t> s_n^+\right\}}
			\sum_{i=0}^{2^n-1}
			F^{(n)}_i\left(W_{\theta_{s_n^+}\left(\bar{\omega}\right)\wedge \cdot}(\bar{\omega}) + \left[W _{\theta_{s_n^+}(\omega) \vee \cdot}-W_{\theta_{s_n^+}(\omega)} \right](\omega),\right.\\\left. N_p	\left(\cdot\cap \left(\left(0,\theta_{s_n^+}(\bar{\omega})\right]\times U\right)\right)(\bar{\omega})
			+ N_p \left(\cdot\cap \left(\left(\theta_{s_n^+}(\omega),T\right]\times U\right)\right)(\omega)\right)1_{\left(t^{(n)}_{i},t^{(n)}_{i+1}\right]}(t),
		\end{multline}
		where $ t\in[0,T],\,\omega\in \Omega.$
		
		\vspace{2mm}
     We conclude this subsection with a lemma that introduces separable Hilbert spaces of square--integrable random variables. These spaces form the foundation for the measurable selection argument that we will develop to prove the second part \eqref{opposite_1} of Theorem \ref{DPP_t} in Subsections \ref{sub_mst}-\ref{subsect_2}. 
	\begin{lemma}\label{l_21}
		For any $0\le s\le t\le T$ and $l \in \mathbb{N}$, the Hilbert space
		\begin{equation}\label{Kal1}
			H= L^2(\Omega,  \mathcal{F}^{W,N_p}_{s,t};\mathbb{R}^l) \;\; \text{is separable}.
		\end{equation}
	\end{lemma}
	\begin{proof}
		The claim is trivial for $s=t$, so we consider $0\le s<t\le T$. Let $\Omega_W^{s,t}  $  be the space of continuous functions from $[s,t]$ to $\mathbb{R}^m$ and, for every $r\in [s,t]$, let  $\mathcal{B}^W_{s,t}$ be the smallest $\sigma-$algebra on $\Omega_W^{s,t}$ which makes the projections $\pi_r - \pi_s \colon\Omega_W^{s,t}\to \mathbb{R}^m,\,r\in[s,t],$ measurable. 
		\\
		Similarly, we denote by  $\Omega_N^{s,t}$ the set of integer--valued measures  defined on $ (s,t]\times  U$ with values in ${\mathbb N} \cup { \{ \infty \}}$.  		We endow $\Omega_N^{s,t}$ with the minimal  
		$\sigma-$algebra $\mathcal{B}_{s,t}^N$  
		which makes measurable all the mappings
		$$ 
		\mu \mapsto \mu ((s,r] \times A),\quad s < r \le t,\;\; A \in \mathcal{U}.
		$$
		We write $\left(\Omega_\times^{s,t}, \mathcal{B}^{\times}_{s,t}\right)$ instead of $\left( \Omega_W^{s,t}\times\Omega_N^{s,t} ,\mathcal{B}_{s,t}^W\otimes \mathcal{B}_{s,t}^N\right)$ to keep the notation short. 
		
				Note  that $H$ defined in \eqref{Kal1} is isomorphic to $K=L^2\left(\Omega_\times^{s,t}, \mathcal{B}^{\times}_{s,t}, \mathbb{Q}^{s,t}; \mathbb{R}^l\right),$ where $\mathbb{Q}^{s,t}$ is the image law of $\mathbb{P}$ under the random variable
		$$
		\omega \mapsto \left(W_{\cdot }(\omega) -  W_{s }(\omega) , N_p\left(\left(s, \cdot \right]\times \cdot\right) (\omega) \right).
		$$
		We know from \cite[Theorem 4.2]{Kalle} that $(\Omega^{s,t}_\times, \mathcal{B}^{\times}_{s,t})$ is metrizable and that it can be considered as a Polish space (this fact has been also remarked in \cite[Section 2]{BLP}). It is not difficult to prove that  $L^2(E, {\mathcal B}, \mu; \mathbb{R}^l )$  is separable when $E$ is a Polish space and $\mu$ is a probability measure defined on the $\sigma-$algebra of Borel sets $\mathcal B$. This shows \eqref{Kal1}. 
	\end{proof}
		 
\subsection {Proof of the first part (\ref{less_general}) of the DPP}		\label{sub_firstDPP}
		
		Here, we use controls  in $\mathcal A$ and the corresponding regular controlled stochastic flows $X^{s,x,a}_t$ constructed in Subsection \ref{sub5.2}, see Lemma \ref{regular_control}. We show that, for every $s\in [0,T)$ and  $x\in\mathbb{R}^d$,
			\begin{equation}\label{less}
				v(s,x)\le \sup_{a\in{\mathcal{A}}} \inf_{\theta\in \mathcal{T}_{s,T}}
				\mathbb{E}\bigg[\int_{s}^{\theta}h\left(r,X_r^{s,x,a}, a_r\right)dr
				+v\left(\theta, X_\theta^{s,x,a}\right)\bigg].
			\end{equation}
		Since $\mathcal{A}\subset \mathcal{P}_T$, the estimate in \eqref{less} will hold replacing $\mathcal{A}$ with $\mathcal{P}_T$, i.e., \eqref{less} implies \eqref{less_general} . 
			Fix $s\in [0,T)$, $x\in\mathbb{R}^d$, $\theta \in\mathcal{T}_{s,T}$ and $a\in\mathcal{A}_{q},$ for some $q\in\mathbb{N}$. By  Lemma \ref{regular_control} and the flow property in Point \ref{flow_1} of Definition \ref{sharpala},
			\begin{align}\label{prim1} 
				\notag J(s,x,a)
				&=\mathbb{E}\bigg[\int_{s}^{\theta}h\left(r,X_r^{s,x,a}, a_r\right) dr \bigg]
				+\mathbb{E}\bigg[\int_{\theta}^{T}h\left(r,X_r^{\theta,X^{s,x,a}_\theta,a}, a_r\right)dr + 
				j\left(X_T^{\theta,X^{s,x,a}_\theta,a}\right)\bigg]
				\\
				&\notag=\mathbb{E}\bigg[\int_{s}^{\theta}h\left(r,X_r^{s,x,a}, a_r\right) dr \bigg]
				\\\notag &\quad \qquad +\sum_{k=0}^{2^q-1}
				\mathbb{E}\bigg[1_{\left({t_{k}^{(q)}, t_{k+1}^{(q)}}\right)}(\theta)\bigg(\int_{\theta_{k}}^{T}h\left(r,X_r^{\theta_{k},X^{s,x,a}_{\theta_{k}},a}, a_r\right)dr + 
				j\left(X_T^{\theta_{k},X^{s,x,a}_{\theta_{k}},a}\right)\bigg)\bigg]
				\\&
					\quad \qquad +\sum_{k=0}^{2^q-1}
				\mathbb{E}\bigg[1_{\left\{\theta= t_{k+1}^{(q)}\right\}}\bigg(\int_{t_{k+1}^{(q)}}^{T}h\bigg(r,X_r^{t_{k+1}^{(q)},X^{s,x,a}_{t_{k+1}^{(q)}},a}, a_r\bigg)dr + 
				j\bigg(X_T^{t_{k+1}^{(q)},X^{s,x,a}_{t_{k+1}^{(q)}},a}\bigg)\bigg)\bigg],
			\end{align}
			where we define $\theta_{k}=\max\{t^{(q)}_{k},\theta\wedge t^{(q)}_{k+1}\},\,k=0,\dots, 2^q-1$. \\
			We  focus on the second and third addends (the sums from 0 to $2^q-1$) in the last member of  \eqref{prim1}. Observe that $X^{s,x,a}_{\theta_{k}}$ is measurable with respect to the $\sigma-$algebra $\mathcal{F}^{W,N_p}_{\theta_{k}}$ generated by the stopping time $\theta_{k}$ relative to the filtration $\mathbb{F}^{W,N_p}$. Therefore, thanks to the arguments in Subsection \ref{mtr} (see, in particular, \eqref{udio}), \eqref{measQ} and \eqref{controlla_flow},  we compute the $\mathcal{F}^{W,N_p}_{\theta_{k}}-$conditional expectation   to deduce that, for every $k=0,\dots, 2^q-1$, for $\mathbb{P}-\text{a.s. }\omega\in\Omega$,
			\begin{multline*}
				\notag	\mathbb{E}\bigg[1_{\left({t_{k}^{(q)}, t_{k+1}^{(q)}}\right)}(\theta)\bigg(\int_{\theta_{k}}^{T}h\left(r,X_r^{\theta_{k},X_{{\theta_{k}}}^{s,x,a},a},  a_r\right)dr + 
				j\left(X_T^{\theta_{k},X_{{\theta_{k}}}^{s,x,a},a}\right)\bigg)\Big| \,\mathcal{F}^{W,N_p}_{\theta_{k}}\bigg]\left(\omega\right) 
				\\ 
				=
				1_{\left({t_{k}^{(q)}, t_{k+1}^{(q)}}\right)}(\theta_{k}(\omega))
				{J}\left(\theta_{k}(\omega),X_{{\theta_{{k}}}}^{s,x,a}(\omega),a_{t_k^{(q)},q}^{{\theta_{k},\omega}}\right)
			 . 
			\end{multline*} 
		Analogously, for the third addend in \eqref{prim1}, we compute the conditional expectation with respect to $\mathcal{F}^{W,N_p}_{t_{k+1}^{(q)}}$ to obtain, for $\mathbb{P}-$a.s. $\omega\in\Omega$,
		\begin{multline*}
				\mathbb{E}\bigg[1_{\left\{\theta= t_{k+1}^{(q)}\right\}}\bigg(\int_{t_{k+1}^{(q)}}^{T}h\bigg(r,X_r^{t_{k+1}^{(q)},X^{s,x,a}_{t_{k+1}^{(q)}},a}, a_r\bigg)dr + 
			j\bigg(X_T^{t_{k+1}^{(q)},X^{s,x,a}_{t_{k+1}^{(q)}},a}\bigg)\bigg) \Big | \mathcal{F}^{W,N_p}_{t_{k+1}^{(q)}}\bigg](\omega)\\=
1_{\left\{ \theta_k(\omega)=t_{k+1}^{(q)}\right\}}
{J}\left(\theta_{k}(\omega),X_{{\theta_{{k}}}}^{s,x,a}(\omega),a_{t^{(q)}_{k+1},q}^{{t^{(q)}_{k+1}{},\omega}}\right).
		\end{multline*}
					Since $a_{t_k^{(q)},q}^{\theta_{k},\omega}, \,a_{t^{(q)}_{k+1},q}^{{t^{(q)}_{k+1}{},\omega}}\in\mathcal{P}_T$, by the definition of the value function $v$ and the tower property of the conditional expectation, from the two  previous equations we deduce that
			\[
			\mathbb{E}\bigg[1_{\left({t_{k}^{(q)}, t_{k+1}^{(q)}}\right]}(\theta)\bigg(\int_{\theta_{k}}^{T}h\left(r,X_r^{\theta_{k},X_{{\theta_{k}}}^{s,x,a},a},  a_r\right)dr + 
			j\left(X_T^{\theta_{k},X_{{\theta_{k}}}^{s,x,a},a}\right)\bigg)\bigg]
			\le 
			\mathbb{E}\left[1_{\left({t_{k}^{(q)}, t_{k+1}^{(q)}}\right]}(\theta)
			v(\theta_{},X_{{\theta_{}}}^{s,x,a})\right].
			\]
			Going back to \eqref{prim1}, we  conclude that
			\[
			J(s,x,a)\le \mathbb{E}\bigg[\int_{s}^{\theta}h\left(r,X_r^{s,x,a}, a_r\right) dr \bigg]+\mathbb{E}\left[v(\theta,X_{{\theta}}^{s,x,a})\right].
			\]
			Since $a\in{\mathcal{A}}$ and $\theta\in \mathcal{T}_{s,T}$ are arbitrary, by \eqref{semplicemente} we obtain \eqref{less}.
			 

		\subsection {A basic measurable selection theorem involving  controls in  $ {\mathcal B}^s$} \label{sub_mst}
	For the proof of the second part \eqref{opposite_1} of the DPP in Theorem \ref{DPP_t}, we employ a basic and classical measurable selection theorem from \cite{BP},  see Theorem \ref{BP_Sim}. To rigorously apply it, we first introduce a suitable class of predictable  step controls which we will denote by ${\mathcal{B}}^s$ ($s\in[0,T]$).
	  
		  \vspace{1mm}	Fix $n\in\mathbb{N}$ and
		  recall the set $S_n=\{t^{(n)}_i,\,i=0,\dots, 2^n\}$ of dyadic points of $[0,T]$ with mesh $T2^{-n}$.
		  Note that, as $U$ is a Polish space, according to \eqref{Kal1}, 
		  \[
		  \text{the space $L_i^{(n)}=L^2\Big(\Omega,  \mathcal{F}^{W,N_p}_{0,t^{(n)}_i};\mathbb{R}^l\Big)$ is separable, $i=0,\dots,\,2^n-1$.}
		  \]
		  It is then possible to consider a Hilbert basis $(e^{(n)}_{i,m})_{m\in\mathbb{N}}$ for $L_i^{(n)}$, which we use to define
	  		  \begin{equation}\label{ref1}
		  	\text{the subspace $L_{i,M}^{(n)}$ of $L_i^{(n)}$ generated by $e^{(n)}_{i,1},\dots,e_{i,M}^{(n)}$, for every $M\in\mathbb{N}$. }
		  \end{equation}
	  Since $\text{proj}_{\bb}\colon \R^l\to\mathbf{C}$ in \eqref{projection} is $1-$Lipschitz continuous, we can introduce the map $\text{proj}_\bb'\colon L_i^{(n)}  \to   L^2\Big(\Omega,  \mathcal{F}^{W,N_p}_{0,t^{(n)}_i}; \bb  \Big)  $ 
	  	  given by
	 \begin{equation}\label{proj_abuse}
	  \text{proj}_\bb'(X) (\omega) = \pj(X(\omega)),\quad X\in  L_i^{(n)},\,\omega\in\Omega.
	 \end{equation}
	  Note that $\text{proj}_\bb'$ is the orthogonal
projection of $ L_i^{(n)} $	 onto its  closed convex subset $L^2\Big(\Omega,  \mathcal{F}^{W,N_p}_{0,t^{(n)}_i};\bb \Big)$ consisting of all random variables with values in $\bb$.  
	  With a slight abuse of notation, we will simply denote $\text{proj}_\bb'$ by $\pj$.
		  Letting $s\in [0,T]$, for every $i=0,\dots,2^n-1$ such that $t_i^{(n)}\ge s$, we define
	  \begin{equation}\label{Hn}
	  H_{s,i}^{(n)}=L^2\Big(\Omega,  \mathcal{F}^{W,N_p}_{s,t^{(n)}_i};\mathbb{R}^l\Big).
  \end{equation}
Lemma \ref{l_21} implies that also $H^{(n)}_{s,i}$ is separable, hence it admits a Hilbert basis $(e^{(n)}_{s,i,m})_{m\in\mathbb{N}}$. We denote by 
\[
\text{$H_{s,i,M}^{(n)}$ the subspace of $H_{s,i}^{(n)}$ generated by $e^{(n)}_{s,i,1},\dots,e_{s,i,M}^{(n)}$, for every $M\in\mathbb{N}$.}
\]
As in \eqref{proj_abuse}, the map  $\text{proj}_\bb\colon H_{s,i}^{(n)}  \to   L^2\Big(\Omega,  \mathcal{F}^{W,N_p}_{s,t^{(n)}_i}; \bb  \Big)$ is continuous and onto,  and $L^2\Big(\Omega,  \mathcal{F}^{W,N_p}_{s,t^{(n)}_i}; \bb  \Big)$ is a closed and convex subset of $H_{s,i}^{(n)}$.\\
	We now define, for all $M\in\mathbb{N}$, the set 
		\[
		{\mathcal{Z}}_{n,M}^s=\big\{\big(Z_{i}\big)_{i=0,\dots, 2^n-1}\text{ s.t. $Z_i\in H_{s,i,M}^{(n)}$ [resp., $L_{i,M}^{(n)}$], for all $i$ : $t_i^{(n)}\ge s$ [resp., $t_i^{(n)}< s$]}\}.
		\]
		The finitely generated step controls corresponding to $\pj \big ( \mathcal{Z}_{n,M}^s\big)$ (i.e., controls of the form \eqref{form} with 
		$
		Z\in \pj \big ( \mathcal{Z}_{n,M}^s
		\big)  
		$)
		are denoted by $\mathcal{B}_{n,M}^s$. We put 
		\begin{equation} \label{cvi}
			{\mathcal{B}}^s=\cup_{n,M\in\mathbb{N}}\mathcal{B}^s_{n,M},\qquad \mathcal{Z}_{\mathcal{B}}^s=\cup_{n,M\in\mathbb{N}}\,    \big(\mathcal{Z}^s_{n,M} \big).
		\end{equation}
		Consequently,  step controls in ${\mathcal{B}}^s$ are associated with $
		\pj (\mathcal{Z}_{\mathcal{B}}^s).
		$
		Define the function
		\[
		\tilde{v}(s,x)=\sup_{a\in{\mathcal{B}}^s}\,J(s,x,a),\quad s\in [0,T),\,x\in\mathbb{R}^d.
		\]
By the continuity of $\pj$,  $\pj (L_{i}^{(n)})=\overline{\cup_M  \pj  L^{(n) }_{i,M} }^2$ and 
		$\pj H_{s,i}^{(n)}= \overline{\cup_M  \pj      H^{(n)}_{s,i,M} }^2$,    	where $\overline{\hspace{.1cm}\cdot\hspace{.1cm}}^2$ denotes the closure relative to the $L^2-$norm. Then,   if we call   
		\begin{equation*}
		\mathcal{Z}_n^s=\big\{(Z_{i})_{i=0,\dots, 2^n-1}\subset L^2(\Omega;{\mathbb R}^l )\text{ s.t. $Z_{_i}$ is $\mathcal{F}^{W,N_p}_{0,t^{(n)}_i}$ [resp., $\mathcal{F}^{W,N_p}_{s,t^{(n)}_i}$]--measurable for $t^{(n)}_i <s$ [resp., $t_i^{(n)}\ge s$]}\}
		\end{equation*}
		and denote by $\mathcal{B}_n^s$ the step controls corresponding to $\pj (\mathcal{Z}_n^s)$,  Lemma \ref{lemma_cont} entails that, similarly to Corollary \ref{l20}, 
		\begin{equation}\label{norestr1}
			\tilde{v}(s,x)=\sup_{a\in\cup_{n\in\mathbb{N}}\mathcal{B}^s_n }J(s,x,a) =  \sup_{y\in\cup_{n\in\mathbb{N}}\mathcal{Z}^s_n }J(s,x, \pj y),\quad s\in[0,T),\,x\in\mathbb{R}^d,
		\end{equation}  
		since we can consider (cf. Remark \ref{ref8})
		\begin{equation}\label{ritorna} 
			J(s, \cdot, \pj(\cdot)) : {\mathbb R}^d \times \cup_{n\in\mathbb{N}}\mathcal{Z}^s_n     \to {\mathbb R}.
		\end{equation}
	The next lemma clarifies that  we can restrict our analysis to controls in $\mathcal{B}^s$, which involve finite sequences of random variables in $\mathcal{Z}_\mathcal{B}^s$ independent of $\mathcal{F}^{W,N_p}_{0,s}$ after time $s$, when considering $v(s,\cdot)$ (cf.  \cite[Remark 5.2]{touzi} and \cite[Remark 3.1]{PW}).

		\begin{lemma}\label{ind_c_l}
			The following equality holds for every $s\in[0,T)$ and $x\in\mathbb{R}^d$:
			\begin{equation}\label{indep_con}
				v(s,x)=\tilde{v}(s,x).
			\end{equation}  
		\end{lemma} 
		\begin{proof}
			We only focus on the inequality $v(s,x)\le \tilde{v}(s,x)$, as the other one is trivial because ${ \mathcal{B}}^s\subset \mathcal{P}_T$.\\
			Fix $s\in [0,T)$, $x\in\mathbb{R}^d$ and a  $\mathbf{Q}-$valued step control $a\in \mathcal{A}_{{n}}$, with  $n\in\mathbb{N}$. 
		Recalling \eqref{udio}, we define
	\begin{multline*}
			a_{s,n}^{{s},\bar{\omega}}(t,\omega)=1_{\left\{t\le s_n^+\right\}}a(t,\bar{\omega})
			+	1_{\left\{t> s_n^+\right\}}
			\sum_{i=0}^{2^n-1}
			F^{(n)}_i\left(W_{{s}\wedge \cdot}(\bar{\omega}) + \left[W _{{s} \vee \cdot}-W_{{s}} \right](\omega),\right.\\\left. N_p	\left(\cdot\cap \left(\left(0,{s}\right]\times U\right)\right)(\bar{\omega})
			+ N_p \left(\cdot\cap \left(\left({s},T\right]\times U\right)\right)(\omega)\right)1_{\left(t^{(n)}_{i},t^{(n)}_{i+1}\right]}(t),
		\end{multline*} 
		where  $ t\in[0,T]$ and $\omega,\, \bar \omega \in \Omega$. 
		Observe that, for every $\mathbf{y}(n)\in (\mathbf{Q})^{2^n}$ and $r\in(s,T]$, by \cite[Theorem 117]{situ} and the construction carried out in the proof of Theorem \ref{thm_big}, the random variable  $(X^{S_n})^{{s},x,\mathbf{y}(n)}_r$ is independent from $\mathcal{F}^{W,N_p}_{0,s}$.
		Then, by the tower property of the conditional expectation, \eqref{measQ} and \eqref{controlla_flow},
		\[
			J(s,x,a)=
			\mathbb{E}\bigg[\mathbb{E}\bigg[\int_{s}^{T}h\left(r,X_r^{s,x,a}, a_r\right)dr + 
			j\left(X_T^{s,x,a}\right)\Big | \mathcal{F}^{W,N_p}_{0,s}\bigg]\bigg]
			=
			\mathbb{E}\left[ 
			J(s,x, a_{s,n}^{s,\cdot})
			\right]\le \tilde{v}(s,x).
			\]
	Here the last inequality is due to \eqref{norestr1} and the fact that $a_{s,n}^{s,\omega}\in \mathcal{B}_{{n}}^s$, for every $\omega\in\Omega$. Taking the supremum over $a\in{\mathcal{A}}$,  by \eqref{semplicemente} we deduce \eqref{indep_con} and complete the proof.
		\end{proof}
		
	We now move to a basic  \emph{measurable selection theorem} that will be fundamental in the proof of \eqref{opposite_1} presented in Subsection \ref{subsect_2}. In particular, we employ a simplified version of \cite[Theorem 2]{BP} (see Theorem \ref{BP_Sim} below), an important result which is also considered, in a more general form, in \cite[Proposition 7.50(b)]{BS}. Note that, according to \cite[Remark 2, Page 909]{BP}, we can work with functions defined in separable absolute Borel sets (or Borel spaces) instead of Polish spaces. Recall that a topological space $X$ is said to be a Borel space if there exists a Polish space $Z$ such that $X$ is homeomorphic to a member of the Borel $\sigma-$algebra of $Z$.  We also recall that, given two Borel spaces $X$ and $Y$, a function $g: X \to Y$ is \emph{universally measurable} if it is $(\mathcal G,\mathcal{B}(Y))-$measurable, where $\mathcal  G$  is the intersection of the completions of the Borel $\sigma-$algebra $\mathcal{B}(X)$ with respect to all the  Borel probability measures on $X. $ 
		\begin{theorem}[\cite{BP}]\label{BP_Sim}
			Let $X$ and $Y$ be Borel spaces and let  $f: X \times Y \to \mathbb{R}$ be a Borel measurable bounded function.
			Then, for any $\epsilon >0$, there exists a universally measurable function 
			$
			\varphi_{\epsilon} : X \to Y
			$
			such that
			\begin{gather*}
				f(x, \varphi_{\epsilon}(x)) \le  \inf_{y \in Y} f(x,y)+\epsilon,\quad x\in X.
			\end{gather*}
		\end{theorem}
\noindent
We refer to \cite{karoui} for more recent and very general  measurable selection theorems.

	We conclude this part with two remarks which will enable us to apply Theorem \ref{BP_Sim} in Subsection \ref{subsect_2} thanks to the finitely generated step controls in $\mathcal{B}^{\bar{\theta}},\,\bar{\theta}\in [0,T]$.
	\begin{rem}\label{seq}
		For every ${\bar \theta}\in [0,T]$, the set ${\mathcal{Z}}_\mathcal{B}^{\bar \theta}$ (see \eqref{cvi}) can be identified with a Borel subset of the separable space 
		\[  
		{\prod}_{\begin{subarray}{l}{i,n \,:\, t^{(n)}_i<{\bar \theta}}\\
				{M\in\mathbb{N}}\end{subarray}}\left[L^{(n)}_{i,M} \right]
		\,\,\times\,\, 
		{\prod}_{\begin{subarray}{l}{i,n \,:\, {\bar \theta}\le t^{(n)}_i<T}\\
				{M\in\mathbb{N}}\end{subarray}}\left[H^{(n)}_{{\bar \theta}, i,M}\right].
		\]  
		Therefore, by \cite[Proposition 7.12]{BS}, ${\mathcal{Z}}_\mathcal{B}^{\bar \theta}$ endowed with the trace topology is a Borel space.  Moreover, for every $n,M\in\mathbb{N}$ and $i=0,\dots,2^{n}-1$ such that $t_i^{(n)}\ge {\bar \theta}$ [resp., $t_i^{(n)}<{\bar \theta}$], the projection map $\pi_{{\bar \theta},i,M}^{(n)}\colon{\mathcal{Z}}_\mathcal{B}^{\bar \theta}\to  H^{(n)}_{{\bar \theta}, i,M}$ [resp., 
		$\pi_{{\bar \theta},i,M}^{(n)}\colon{\mathcal{Z}}_\mathcal{B}^{\bar \theta}\to  L^{(n)}_{i,M}$] is continuous, hence Borel measurable. 	
	\end{rem}
	  
	\begin{rem} \label{finite} 	
		Fix  a separable Hilbert space $H=L^2(\Omega, \mathcal{G}; \mathbb{R}^k)$ with  basis $\mathbf{e}=(e_m)_m$. We can choose a representative for every element $e_m,\,m\in\mathbb{N}$, of the basis $\mathbf{e}$. In this way,  it makes sense to consider $e_m(\omega),\,\omega\in\Omega$. Given $M\in\mathbb{N},$ we denote by $F_M=\text{span}\left\{e_1,\dots,e_M\right\}$;  for every $y\in F_M$, there exists a unique representative $\bar{y}$ of $y$ such that $\bar{y}(\omega)=\sum_{m=1}^{M}\langle y,e_m\rangle_2e_m(\omega),\,\omega\in\Omega$. 
		We then define a  map $T\colon F_M \times (\Omega, \mathcal{G}) \to \R^k$ by 
		$$ 
		T(y, \omega)= \bar{y}(\omega),\quad y\in F_M,\, \omega\in\Omega.
		$$
		Note that $T$ is well defined, and depends on the choice of the representatives for $(e_m)_m$. 
		Furthermore, $T$ is  \emph{measurable with respect to the product $\sigma-$algebra}. Indeed, this is a consequence of the fact that $T(y,\cdot)$ is $\mathcal{G}-$measurable for every $y\in F_M$, and that $T(\cdot,\omega)\colon F_M\to \mathbb{R}^k$ is continuous for every $\omega\in\Omega$.  This observation will be important in \eqref{richiama} (see also \eqref{eccolo}).   
	\end{rem}
	
	\subsection {Proof of the second part (\ref{opposite_1}) of the  DPP}\label{subsect_2}
	 Here, using the class of predictable controls $\mathcal{B}^s$ introduced in the previous subsection, we show that, for every $s\in [0,T)$ and  $x\in\mathbb{R}^d$,
			\begin{equation}\label{opposite}
				v(s,x)\ge \sup_{a\in{\mathcal{A}}} \sup_{\theta\in \mathcal{T}_{s,T}}
				\mathbb{E}\bigg[\int_{s}^{\theta}h\left(r,X_r^{s,x,a}, a_r\right)dr
				+v\left(\theta, X_\theta^{s,x,a}\right)\bigg].  
			\end{equation}
	 	Note that \eqref{opposite}  is equivalent to proving \eqref{opposite_1}, which is the same inequality with $\mathcal{P}_T$ instead of ${\mathcal{A}}$. Indeed, for every $s\in[0,T),$ $x\in\mathbb{R}^d$,  $\theta\in\mathcal{T}_{s,T}$ and $a\in\mathcal{P}_T$, also employing Corollary \ref{estrap}, we can find a sequence $(a_n)_n\subset{\mathcal{A}}$ such that, for $\mathbb{P}-$a.s. $\omega\in\Omega$,  $\lim_{n\to\infty} a_n(\cdot,\omega)= a(\cdot,\omega)$ a.e. in $[0,T]$ and $\lim_{n\to\infty}\sup_{0\le t\le T}|X^{s,x,a_n}_t-X^{s,x,a}_t|=0$.  Hence, by the dominated convergence theorem, 
		\[
			\lim_{n\to \infty}
				\mathbb{E}\bigg[\int_{s}^{\theta}h\left(r,X_r^{s,x,a_n}, a_{n,r}\right)dr
			\bigg]
			=
				\mathbb{E}\bigg[\int_{s}^{\theta}h\left(r,X_r^{s,x,a}, a_{r}\right)dr
			\bigg].
		\]
	Furthermore, 	by Lemma \ref{lsc} and Fatou's lemma,
	\[
		\liminf_{n\to \infty}\mathbb{E}\left[v(\theta,X_\theta^{s,x,a_n})\right]
		\ge
		\mathbb{E}\left[\liminf_{n\to \infty}v(\theta,X_\theta^{s,x,a_n})\right]
		\ge
		\mathbb{E}\left[v(\theta,X_\theta^{s,x,a})\right].
	\]
Combining the last two equations shows that \eqref{opposite} implies \eqref{opposite_1}, as desired.

\vspace{2mm}
Fix $s\in[0,T)$, $x\in\mathbb{R}^d$ and  $a\in{ \mathcal{A}}$. First, we assume that $\theta=\bar{\theta}\in  (s,T)\cap S$. 	As mentioned in Subsection \ref{sub_mst}, our idea is to apply a measurable selection argument, specifically Theorem \ref{BP_Sim}. Before doing so,  some  preparation is required.\\
	 		Consider $n,M\in\mathbb{N}$ and $i=0,\dots,2^{n}-1$ such that $t_i^{(n)}\ge {\bar \theta}$. Recall the space 
			\[
				H^{(n)}_{{\bar \theta}, i,M}\subset H_{{\bar \theta},i}^{(n)} \text{ generated by } e^{(n)}_{{\bar \theta},i,m},\, m=1,\dots,M,
			\] 
			see \eqref{Hn} and the subsequent equation. We write $e_m$ for $e^{(n)}_{{\bar \theta},i,m}$ to keep notation simple 
  and    fix a representative for every function $e_{m},\,m=1,\dots, M$. 
  By Remark \ref{finite} we  can consider the measurable map
  \begin{equation}\label{richiama} 
  	g_{t_i^{(n)}}\colon  H^{(n)}_{{\bar \theta}, i,M} \times \big(\Omega, \mathcal{F}^{W,N_p}_{{\bar \theta},t_i^{(n)}}\big) \to \R^l,\qquad g_{t_i^{(n)}}(y, \omega) = \bar{y}(\omega),
  \end{equation}
  where $ \bar{y}(\omega)=  \sum_{m=1}^{M}\langle y,e_m\rangle_{2}\,  e_m(\omega)$.   
  
			Denote by $\mathcal X = \mathcal{D}([s, \bar{\theta}]; \mathbb{R}^d)$ the usual space of càdlàg functions endowed with the Skorokhod topology: since  $\mathcal X$ is a Polish space, it is a Borel space too.
			%
			Recalling \eqref{cvi}, we introduce the function $f\colon \mathcal X\times \mathcal{Z}_\mathcal{B}^{\bar{\theta}}\to \mathbb{R}$ defined by, for  every $\xi \in \mathcal X$ and $y \in {\mathcal{Z}}_\mathcal{B}^{\bar{\theta}}$, 
	 		\begin{gather*} 
				f(\xi,y) = - J(\bar{\theta}, \pi_{\bar{\theta}}(\xi), \pj y),\quad\text{ where \,\,$\pi_{\bar{\theta}}(\xi) = \xi (\bar{\theta})$}. 
			\end{gather*}
			The map $f$ is Borel measurable and bounded. Thus, by Remark \ref{seq}, we  can apply Theorem \ref{BP_Sim}, which yields,  for any $\epsilon >0$,   the existence of a universally measurable function  $ c_{\epsilon} \colon \mathcal X \to {\mathcal{Z}}_\mathcal{B}^{\bar{\theta}}$ such that   
			\begin{gather}\label{sel_eq}
				J(\bar{\theta}, \pi_{\bar{\theta}}(\xi),  \pj c_{\epsilon}(\xi)) 
				\ge  \sup_{y \in {\mathcal{Z}}_\mathcal{B}^{\bar{\theta}}} \, J(\bar{\theta}, \pi_{\bar{\theta}}(\xi), \pj y) - \epsilon 
				= v(\bar{\theta}, \pi_{\bar{\theta}}(\xi)) - \epsilon,\quad  \xi \in \mathcal X.
			\end{gather}
			Note that  the last equality in \eqref{sel_eq} is due to Lemma \ref{ind_c_l}, see also \eqref{ritorna}.
			Consequently, 
			\begin{gather}\label{sisi}
				\mathbb{E} [  J(\bar{\theta}, X^{s,x,a}_{{\bar{\theta}}} ,  \pj c_{\epsilon}(X^{s,x,a}_{{\cdot }} ))  ]
				\ge \mathbb{E}  [v(\bar{\theta}, X^{s,x,a}_{{\bar{\theta}}})]  - \epsilon.
			\end{gather}
			 For every $\omega\in\Omega,$ we identify   
			 $$ {\mathcal{B}}^{\bar{\theta}}\ni\tilde{c}_\epsilon(\omega)\sim \pj c_\epsilon(X^{s,x,a}_\cdot(\omega))\in \pj ({\mathcal{Z}}_{\mathcal{B}}^{\bar{\theta}}). 
			 $$

						At this point, we modify the control $a\in{\mathcal{A}}$ after time $\bar{\theta}$ using the processes $\tilde{c}_\epsilon(\cdot)$, with the aim of invoking \eqref{sisi} and the flow property in Point \ref{flow_1} of Definition \ref{sharpala}. However, since $\tilde{c}_\epsilon(\omega)$ is not $\mathbf{Q}-$valued and belongs to $\mathcal{B}^{\bar{\theta}}_{n,M}$ for some integers $n,M$ \emph{depending on} $\omega$, we need to consider suitable approximated controls.\\
						More specifically, for every $\omega\in\Omega$ and $n,M\in\mathbb{N}$, define $c_{\epsilon,M}^{(n)}(\omega)\in\mathcal{Z}^{\bar{\theta}}_{n,M}$ by   (recall Remark \ref{seq})
			\[
			\big(c_{\epsilon,M}^{(n)}(\omega)\big)_i=\pi^{(n)}_{{\bar \theta},i,M}\left(c_\epsilon(X^{s,x,a}_{{\cdot }}(\omega))\right),\quad i=0,\dots,2^{n}-1;
			\]
			{we identify $\quad \mathcal{B}^{\bar{\theta}}_{n,M}\ni c_{\epsilon,M,n}(\omega)\sim \pj c_{\epsilon,M}^{(n)}(\omega)\in \pj(\mathcal{Z}^{\bar{\theta}}_{n,M})$. } \\
			Recalling \eqref{richiama}, for every $i=0,\dots,2^n-1$ such that $t_i^{(n)}\ge \bar{\theta}$, we define the $\mathcal{F}^{W,N_p}_{0,t_i^{(n)}}-$measurable random variable
			\[
			\tilde{c}_{\epsilon,M,i}^{(n)} (\omega)\coloneqq \pj\!\left(\big(c_{\epsilon,M}^{(n)}(\omega)\big)_i(\omega)\right)
			=\pj\!\left(g_{t_i^{(n)}}\left(\pi^{(n)}_{{\bar \theta},i,M}\left(c_\epsilon(X^{s,x,a}_{{\cdot }}(\omega))\right)\right),\omega\right),\quad \omega\in\Omega.
			\]
			As for the indexes $i=0,\dots, 2^{n}-1$ such that $t_i^{(n)}< \bar{\theta}$, we set  $\tilde{c}_{\epsilon,M,i}^{(n)} =0$ in $\Omega$. \\
			As in the proof of Corollary \ref{estrap}, since $\qq$ is countable we write $\qq=\{q_j\}_{j\in \mathbb{N}}$. For every $K\in\mathbb{N}$, we denote by $\Phi_K\colon \bb \to {\qq}$ the (measurable) function such that, given $x\in \bb$, 
			$\Phi_K(x)$  is  the first   element in $\{ q_1, ..., q_K \} $ having the smallest distance from $x$. Recalling that $\qq$ is dense in $\bb$, $\phi_K(x) \to x$ as $K \to \infty$.
			We define
			\begin{equation}\label{eccolo} 
				\tilde{c}^{(n)}_{\epsilon,M,i,K}(\omega)\coloneqq \Phi_K\left(\tilde{c}_{\epsilon,M,i}^{(n)} (\omega)\right)
				\left[=
				\Phi_K\left(\pj\left(g_{t_i^{(n)}}\left(\pi^{(n)}_{{\bar \theta},i,M}\left(c_\epsilon(X^{s,x,a}_{{\cdot }}(\omega))\right),\omega\right)\right)\right),\quad t_i^{(n)}\ge \bar{\theta}\right];
			\end{equation}
			the map $\Phi_K$ ensures that the $\mathcal{F}^{W,N_p}_{0,t_i^{(n)}}-$measurable random variables $\tilde{c}^{(n)}_{\epsilon,M,i,K}$ in \eqref{eccolo} are square--integrable and ${\qq}-$valued, which allows to identify   $\quad  \mathcal{Z}_{n}\ni\big(\tilde{c}^{(n)}_{\epsilon,M,i,K}\big)_i\sim \tilde{c}_{\epsilon,M,n,K}\in\mathcal{A}_{n}$ (see \eqref{aaa}). 
		 \vspace{1mm}\\
			Consider $n,M,K\in\mathbb{N}$ and suppose that $\bar{\theta}\in\Pi_n$.  Defining the step process  ${a}_{\epsilon,M,n,K}\in\mathcal{A}$ by 
			\[
					a_{\epsilon,M,n,K}(t,\omega)=1_{\{t\le \bar{\theta}\}}a(t,\omega)+
1_{\{t> \bar{\theta}\}}			\tilde{c}_{\epsilon,M,n,K}(t,\omega),\quad t\in[0,T],\,\omega\in \Omega,
			\]
		by	the flow property of $X^{s,x,({a}_{\epsilon,M,n,K})}$ in Point \ref{flow_1} of Definition \ref{sharpala} (see also \eqref{controlla_flow} in the proof of Lemma \ref{regular_control}),
			\begin{multline}\label{tr}
				v(s,x)\ge \mathbb{E}\bigg[\int_{s}^{\bar{\theta}}h\left(r,X_r^{s,x,a}, a_r\right)dr \bigg]
				\\+\mathbb{E}\bigg[\int_{\bar{\theta}}^{T}h\left(r,X_r^{\bar{\theta},X^{s,x,a}_{\bar{\theta}},(\tilde{c}_{\epsilon,M,n,K})}, {(\tilde c_{\epsilon,M,n,K})}_r\right)dr + 
				j\left(X_T^{\bar{\theta},X^{s,x,a}_{\bar{\theta}},(\tilde{c}_{\epsilon,M,n,K})}\right)\bigg].
			\end{multline}

			\noindent	 By \eqref{richiama} and  \eqref{eccolo}, for every $i=0,\dots,2^n-1$ such that $t_i^{(n)}\ge \bar{\theta}$,  $\text{for }\mathbb{P}-\text{a.s. }\omega\in\Omega,$ 
			\begin{multline*}
				\mathbb{E}\left[\tilde{c}^{(n)}_{\epsilon,M,i,K}\big | \mathcal{F}^{W,N_p}_{0,\bar\theta}\right]\left(\omega\right)
				=\mathbb{E}\left[\Phi_K	\left(\pj
			\left(g_{t_i^{(n)}}\left(\pi^{(n)}_{{\bar \theta},i,M}\left(c_\epsilon(\xi)\right),\cdot\right)\right)\right)
				\right]_{\xi = X^{s,x,a}_{{\cdot }}(\omega)}
				\\=
				\mathbb{E}\left[\Phi_K\big(\pj\big(c^{(n)}_{\epsilon,M}(\omega)\big)_i(\cdot)\big)
				\right]. 
			\end{multline*} 
			Therefore, $\mathbb{P}-$a.s., identifying $ \mathcal{Z}_{n}\ni\Phi_K\big(\pj{c}^{(n)}_{\epsilon,M}(\omega)\big)_i\sim {c}_{\epsilon,M,n,K}(\omega)\in\mathcal{A}_{n}$,
			\begin{equation*}
				\mathbb{E}\bigg[\int_{\bar{\theta}}^{T}h\left(r,X_r^{\bar{\theta},X^{s,x,a}_{\bar{\theta}},(\tilde{c}_{\epsilon,M,n,K})}, (\tilde{c}_{\epsilon,M,n,K})_r\right)dr + 
				j\left(X_T^{\bar{\theta},X^{s,x,a}_{\bar{\theta}},(\tilde{c}_{\epsilon,M,n,K})}\right)  \Big| {\mathcal F}^{W,N_p}_{0,\bar{\theta}}\bigg]
				=J\left(\bar{\theta},X^{s,x,a}_{\bar{\theta}}, c_{\epsilon,M,n,K}\right).
			\end{equation*}  
			Going back to \eqref{tr}, by the law of total expectation  we can write
			\begin{equation}\label{dais1} 
				v(s,x)\ge \mathbb{E}\bigg[\int_{s}^{\bar{\theta}}h\left(r,X_r^{s,x,a}, a_r\right)dr \bigg]+
				\mathbb{E}\left[J\left(\bar{\theta},X^{s,x,a}_{\bar{\theta}}(\cdot),c_{\epsilon,M,n,K}(\cdot)\right)\right].
			\end{equation}
			Observing that, for every $\omega\in\Omega$,  
			\[
			\lim_{K\to \infty} \Phi_K\big(\pj\big(c_{\epsilon,M}^{(n)}(\omega)\big)_i(\omega')\big)
			=
		\pj\!	\big(c_{\epsilon,M}^{(n)}(\omega)\big)_i(\omega')
			,\quad i=0,\dots,2^n-1,\,\omega'\in\Omega,
			\]
			by Lemma \ref{lemma_cont} we can pass to the limit as $K\to \infty$ in \eqref{dais1} to obtain, by dominated convergence, 
			\begin{equation}\label{dais}
				v(s,x)\ge \mathbb{E}\bigg[\int_{s}^{\bar{\theta}}h\left(r,X_r^{s,x,a}, a_r\right)dr \bigg]+
				\mathbb{E}\left[J\left(\bar{\theta},X^{s,x,a}_{\bar{\theta}}(\cdot),c_{\epsilon,M,n}(\cdot)\right)\right].
			\end{equation}
			Notice  that, for every $\omega\in \Omega$, when $n$ and $M$ are sufficiently large then ${\mathcal{B}}^{\bar{\theta}}\ni \tilde{c}_\epsilon(\omega)=c_{\epsilon,M,n}(\omega)\in\mathcal{B}_{n,M}^{\bar{\theta}} $.  Hence, choosing $M=n$, we have
			\begin{equation*}
				J \left(\bar{\theta}, X^{t,x,a}_{\bar{\theta}} (\omega), \tilde{c}_\epsilon(\omega) \right)
				=
				\lim_{n\to\infty }J\left(\bar{\theta},X^{s,x,a}_{\bar{\theta}}(\omega), c_{\epsilon,n,n}(\omega)\right),\quad \omega\in \Omega,
			\end{equation*}
			which by dominated convergence implies 
			\begin{equation}\label{maddai}
				\mathbb{E}\left[J \left(\bar{\theta}, X^{t,x,a}_{\bar{\theta}}(\cdot), \tilde{c}_\epsilon\left(\cdot\right)\right)\right]=\lim_{n\to\infty }
				\mathbb{E}\left[J\left(\bar{\theta},X^{s,x,a}_{\bar{\theta}}(\cdot),c_{\epsilon,n,n}(\cdot)\right)\right].
			\end{equation}
			Combining \eqref{dais}-\eqref{maddai}  with \eqref{sisi}  we conclude that 
			\begin{equation*}
				v(s,x)\ge \mathbb{E}\bigg[\int_{s}^{\bar{\theta}}h\left(r,X_r^{s,x,a}, a_r\right)dr \bigg]+
				\mathbb{E}\left[J \left(\bar{\theta}, X^{t,x,a}_{\bar{\theta}}, \tilde{c}_\epsilon \right)\right]
				\ge 
				\mathbb{E}\bigg[\int_{s}^{\bar{\theta}}h\left(r,X_r^{s,x,a}, a_r\right)dr +v(\bar{\theta}, X^{t,x,a}_{{\bar{\theta}}})\bigg] -\epsilon.
			\end{equation*}

			Suppose now that $\theta$ is a simple, $\mathbb{F}^{W,N_p}-$stopping time with values in $(s,T)\cap S$.  Then we can write $\theta=\sum_{k=1}^{N}\theta_k1_{A_k}$, where $\theta_k\in (s,T)\cap S$ and $(A_k)_k$ is a partition of $\Omega$ such that $A_k\in\mathcal{F}_{0,\theta_k}^{W,N_p},$ $k=1,\dots, N.$ We can invoke the measurable selection theorem (see Theorem \ref{BP_Sim}) $N-$times to deduce, for any $\epsilon>0$, the existence of a universally measurable function $c_{\epsilon,k} \colon \mathcal D\left([s,\theta_k];\mathbb{R}^d\right) \to {\mathcal{Z}}_{\mathcal{B}}^{{\theta}_k}$  such that (cf. \eqref{sel_eq})
			\[
			J({\theta}_k, X^{s,x,a}_{{{\theta}_k}}(\omega) , \pj  c_{\epsilon,k}(X^{s,x,a}_{{\cdot }}(\omega) ))
			\ge v\left({\theta}_k, X^{s,x,a}_{{{\theta}_k}}(\omega)\right)  - \epsilon,\quad k=1,\dots, N,\,\omega\in\Omega.
			\]
			If we identify, for every $k=1,\dots, N$,  $ {\mathcal{B}}^{{\theta}_k}\ni\tilde{c}_{\epsilon,k}(\omega)\sim\pj c_{\epsilon,k}(X^{s,x,a}_\cdot(\omega))\in \pj({\mathcal{Z}}_\mathcal{B}^{{\theta}_k}),\,\omega\in\Omega$, then conditioning with respect to $\mathcal{F}^{W,N_p}_{0,\theta_k}$ we can follow the previous arguments to obtain
			\begin{multline}\label{wed}
				v(s,x)\ge 
				\mathbb{E}\left[\int_{s}^{{\theta}}h\left(r,X_r^{s,x,a}, a_r\right)dr\right]+
				\sum_{k=1}^{N}\mathbb{E}\left[1_{A_k}J\left(\theta_k, X^{s,x,a}_{\theta_k}(\cdot),\tilde{c}_{\epsilon,k}(\cdot)\right)\right]\\
				\ge	\mathbb{E}\left[\int_{s}^{{\theta}}h\left(r,X_r^{s,x,a}, a_r\right)dr +v({\theta}, X^{s,x,a}_{{{\theta}}})\right] -\epsilon
				. 
			\end{multline}

			Finally, we show that \eqref{opposite} holds for all $\theta\in \mathcal{T}_{s,T}$.  We consider a sequence $(\tilde{\theta}_n)_{n\in\mathbb{N}}\subset \mathcal{T}_{s,T}$ of simple $\mathbb{F}^{W,N_p}-$stopping times with values in $(s,T)\cap S$ such that $\tilde{\theta}_n \downarrow \theta$ as $n\to \infty,\,\mathbb{P}-$a.s.
			An application of Lemma \ref{lsc} yields, for $\mathbb{P}-$a.s. $\omega\in\Omega$,
			\begin{equation*}
				v\left(\theta(\omega),X^{s,x,a}_{{\theta}}(\omega)\right)\le \liminf_{n\to \infty}v\left(\tilde{\theta}_n(\omega), X^{s,x,a}_{\tilde{\theta}_n}(\omega)\right).
			\end{equation*}
			Hence by \eqref{wed} and Fatou's lemma -- which can be applied because $v$ is bounded -- we deduce that
			\begin{multline*}
				v(s,x)\ge 
				\liminf_{n\to \infty}\left(\mathbb{E}\left[\int_{s}^{{\tilde{\theta}}_n}h\left(r,X_r^{s,x,a}, a_r\right)dr +v\left({\tilde{\theta}_n}, X^{s,x,a}_{{{\tilde{\theta}_n}}}\right)\right] \right)-\epsilon\\
				\ge
				\mathbb{E}\left[\int_{s}^{{\theta}}h\left(r,X_r^{s,x,a}, a_r\right)dr +v({\theta}, X^{s,x,a}_{{{\theta}}})\right] -\epsilon.
			\end{multline*}	
			Since $\epsilon>0$, $a\in\mathcal{A}$ and $\theta\in \mathcal{T}_{s,T}$ are chosen arbitrarily, the previous equation entails \eqref{opposite}, which in turn gives \eqref{opposite_1}.
			
			Combining \eqref{less_general} obtained in Subsection \ref{sub_firstDPP} with \eqref{opposite_1}, the proof of Theorem \ref{DPP_t} is  complete.
		
\vspace{-.4cm}
\begin{appendices}
	\section{On the $\sigma-$algebra $\mathcal{C}$}\label{ap_sigma}
Recall that $\mathcal{D}([0,T];\mathbb{R}^d)$ is the set of $\mathbb{R}^d-$valued, càdlàg functions defined on $[0,T]$, and that we denote by $\mathcal{D}_0=(\mathcal{D}([0,T];\mathbb{R}^d),\norm{\cdot}_0)$ and by $\mathcal{D}_S=(\mathcal{D}([0,T];\mathbb{R}^d),J_1)$. In particular,  the Skorokhod topology $J_1$ is generated by the metric $d_S$, see Remark \ref{skor}. \\	
Notice that $C(\mathbb{R}^d;\mathcal{D}_0)\!\subset\! C(\mathbb{R}^d;\mathcal{D}_S)$, because $d_S(x,y)\!\le \!\norm{x-y}_0$ for every $x,y\in\mathcal{D}([0,T];\mathbb{R}^d)$. Since $(C(\mathbb{R}^d;\mathcal{D}_S), d^{lu}_S)$ is separable (see \cite{Kh}), where 
$$
d^{lu}_{S}\left(f,g\right)= \sum_{N=1}^{\infty}{2^{-N}}\frac{\sup_{\left|x\right|\le N} d_S(f\left(x\right),g\left(x\right))}{1+\sup_{\left|x\right|\le N} d_S(f\left(x\right),g\left(x\right))},\quad f,g\in C\left(\mathbb{R}^d;\mathcal{D}_S\right),
$$
it follows that $(C(\mathbb{R}^d;\mathcal{D}_0), d^{lu}_S)$ is separable, too. We denote $\mathcal{C}_S$ the corresponding Borel $\sigma-$algebra.
\begin{lemma}\label{l=}
	The following equality between $\sigma-$algebras holds:
	\begin{equation}\label{=}
		\mathcal{C}=\mathcal{C}_S.
	\end{equation}
\end{lemma}
\begin{proof}
	First we prove the inclusion $\mathcal{C}\subset \mathcal{C}_S$. Fix $x\in\mathbb{R}^d$ and consider the projection $\pi_x\colon C(\mathbb{R}^d;\mathcal{D}_0)\to \mathcal{D}_0$ defined by $\pi_x(f)=f(x),\,f\in C(\mathbb{R}^d;\mathcal{D}_0)$. Notice that 
	$\pi_x\colon (C(\mathbb{R}^d;\mathcal{D}_0), d^{lu}_S)\to \mathcal{D}_S$ is continuous, because if $f_n,f\in C(\mathbb{R}^d;\mathcal{D}_0),\,n\in\mathbb{N}$, such that $\lim_{n\to \infty}d^{lu}_S(f_n,f)=0$, then 
	$$\lim_{n\to \infty}d_S(f_n(y),f(y))=0,\quad  \text{for every $y\in\mathbb{R}^d$.}
	$$
	As $\mathcal{D}$ is the Borel $\sigma-$algebra generated by $J_1$, we infer that $\pi_x$ is $\mathcal{C}_S/\mathcal{D}-$measurable, whence $\mathcal{C}\subset \mathcal{C}_S$.
	
	Secondly, we focus on the inclusion $\mathcal{C}_S\subset \mathcal{C}$. Since $(C(\mathbb{R}^d;\mathcal{D}_0), d^{lu}_S)$ is separable, it suffices to show that 
	\[
	B(f_0,R)=\left\{f\in C(\mathbb{R}^d;\mathcal{D}_0): d^{lu}_S\left(f_0,f\right)\le R\right\}\in \mathcal{C},\quad f_0\in C(\mathbb{R}^d;\mathcal{D}_0),\,R>0.
	\]
	Hence we fix $f_0\in C(\mathbb{R}^d;\mathcal{D}_0)$ and consider the map $C(\mathbb{R}^d;\mathcal{D}_0) \ni f \mapsto   d^{lu}_S\left(f_0,f\right)\in\mathbb{R}$: we argue that it is $\mathcal{C}-$measurable. Indeed, this  is a consequence of the fact that, for every $m\in\mathbb{N}$, the map
	\[
	C(\mathbb{R}^d;\mathcal{D}_0) \ni f 
	\mapsto\frac{\sup_{\left|x\right|\le m} d_S(f_0\left(x\right),f\left(x\right))}{1+\sup_{\left|x\right|\le m} d_S(f_0\left(x\right),f\left(x\right))}\in \mathbb{R}
	\quad \text{is $\mathcal{C}-$measurable.}
	\]
	Note that the supremum can be computed over $x\in\mathbb{Q}^d$ because $f_0,\,f\colon \mathbb{R}^d\to \mathcal{D}_S$ are continuous. Since $x/(1+x),\,x\in\mathbb{R},$  is measurable, we only need to verify that 
	\begin{equation}\label{h}
		h_m\colon C(\mathbb{R}^d;\mathcal{D}_0)\to \mathbb{R} \quad \text{defined by}\quad h_m(f)=\sup_{\left|x\right|\le m,\,x\in\mathbb{Q}^d} d_S(f_0\left(x\right),f\left(x\right))\quad \text{is $\mathcal{C}-$measurable}.
	\end{equation}
	For every $c\in \mathbb{R}$ and $r>0$, denoting by $B(c,r)\subset \mathbb{R}$ the closed ball of radius $R$ and center $c$ in $\mathbb{R}$,
	\begin{align*}
		h_m^{-1}\left(B_{}(c,r)\right)&=
		\left\{f\in C(\mathbb{R}^d;\mathcal{D}_0) : \sup_{\left|x\right|\le m,\,x\in\mathbb{Q}^d} d_S(f_0\left(x\right),f\left(x\right))\in  B_{}(c,r)\right\}
		\\&
		=	\bigcap_{\left|x\right|\le m,\,x\in\mathbb{Q}^d}
		\pi_x^{-1}\left(d^{-1}_S(f_0(x),\cdot)(B_{}(c,r))\right)\in\mathcal{C},
	\end{align*}
	where in the last step we use  that $d_S(f_0(x),\cdot)\colon \mathcal{D}_S\to \mathbb{R}$ is continuous (hence measurable) and that $\mathcal{D}$ is the Borel $\sigma-$algebra generated by $J_1$. 
	Thus, \eqref{h} is satisfied, whence $\mathcal{C}_S\subset \mathcal{C}$.
	
	The double inclusion proves \eqref{=}, completing the proof. 
\end{proof}

\section{Proof of Theorem \ref{thm_BZ}}\label{ap_BZ}
In this section we use the same notation and work under the same hypotheses as in Theorem \ref{thm_BZ}. First of all, by Jensen's inequality, which can be invoked since $\theta$ is concave,  we have
\begin{equation}\label{prop_theta}
	\sup\left\{\sum_{j=1}^{n}\theta\left(p_j\right), \,p_j\ge 0\text{ such that }\sum_{j=1}^{n}p_j=1\right\}=n\theta\left(\frac{1}{n}\right),\quad n\in\mathbb{N}.
\end{equation} 
 Before presenting the proof of Theorem \ref{thm_BZ}, we introduce an approximation scheme. For every $n\in\mathbb{N}$, we denote by $S_n$ the  set of dyadic points
 \[
 	S_n=\left\{
 	Tj2^{-n},\,j=0,1,\dots,2^n
 	\right\}.
 \] 
 For any $t\in[0,T]$, let $t_n^-=\max\{s\in S_n : s\le t\}$ and $t_n^+=\min\{s\in S_n : s>t\}$, where we set $T_n^+=\infty$. Next, for every $n\in\mathbb{N}$ we define the function $g_n\colon \Xi\times S_{n+1}\to S_n$ by 
\[
	g_n\left(\omega, t\right)=\begin{cases}
		t_n^-,&\Delta\left(X_t\left(\omega\right), X_{t_n^-}\left(\omega\right)\right)=\Delta\left(t_{n}^-,t,t_n^+\right)\left(\omega\right)\\
		t_n^+,&\text{otherwise}
	\end{cases},\quad \omega\in \Xi,\,t\in S_{n+1}.
\]
In our framework, the map  $\omega\mapsto \Delta(X_{t_n^\pm}(\omega),X_t(\omega))$ is $\mathcal{G}-$measurable. It follows that also the function $\omega\mapsto \Delta(X_t(\omega), X_{g_n\left(w,t\right)}\left(\omega\right))=\Delta(t_n^{-},t,t_n^+)(\omega)$ is $\mathcal{G}-$measurable, because it is the minimum of two random variables. Finally, for all $m,n\in \mathbb{N}$ such that $n\ge m$, define $f_{m,n}=g_{m}\circ\dots\circ g_{n}$. Note that $f_{m,n}$ maps $\Xi\times S_{n+1}$ into $S_m$. The following lemma is a fundamental tool in the proof of Theorem \ref{thm_BZ}.
\begin{lemma}\label{l1.3}
	Let $m,n\in\mathbb{N}$ be such that $n\ge m$ and $\omega\in \Xi$. Then the map $g_n(\omega,\cdot)$ is non--decreasing in $S_{n+1}$ and, restricted to $S_n$, is the identity. As a consequence, the map $f_{m,n}(\omega,\cdot)$ is non--decreasing in $S_{n+1}$.
	
	Furthermore, there exists a family of increasing, càdlàg step functions $f_n\colon \Xi\times [0,T]\to S_n,\,n\in\mathbb{N}$, such that 
	\begin{equation}\label{comp_f}
		f_{m,n}(\omega,\cdot)\circ f_{n+1}(\omega,\cdot)=f_m(\omega,\cdot),\quad n\ge m,\,\omega\in\Xi,
	\end{equation}
and that 
\begin{equation}\label{dist_T}
	\left|t-f_n\left(\omega,t\right)\right|\le T2^{-n},\quad \omega\in \Xi,\,t\in\left[0,T\right].
\end{equation}
\end{lemma}
\begin{proof}
	Fix $m,n\in\mathbb{N}$ such that $n\ge m$ and $\omega\in \Xi$. In the sequel, we do not explicitly write the dependence of $g_n$ and $f_{m,n}$ on $\omega$ to keep notation simple. By definition, $S_n\subset S_{n+1}$ and $t=t_n^-$ for every $t\in S_n$. This implies that $g_n(t)=t$. Consider now $s, t\in S_{n+1}$ such that $s<t$: since $s_n^+\le t_n^-$, the function $g_n$ is non--decreasing in $S_{n+1}$, as desired. Therefore the same property  holds for $f_{m,n}=g_m\circ\dots\circ g_n$, too.
	
	For the second part of the statement, consider $n\in\mathbb{N},\,\omega\in \Xi,\,s\in S_n$, and for every  integer $k\ge n$ define $T(k,n,s)(\omega)=\min\{t\in S_{k+1} : f_{n,k}(\omega,t)=s\}.$ Since $g_{k+1}$ restricted to $S_{k+1}$ is the identity, the sequence $(T(k,n,s)(\omega))_{k\ge n}$ is non--increasing, hence there exists $T(n,s)(\omega)=\lim_{k\to\infty}T(k,n,s)(\omega)$. The monotonicity of the map $f_{n,k}$ proved in the previous point yields 
	\begin{equation}\label{T_space}
		T\left(n,s\right)\left(\omega\right)\in \left[s-T2^{-n},s\right],\quad s\in S_n\setminus\{0\},
	\end{equation}
	and 
	\begin{equation}\label{T_incr}
		T\left(n,s\right)\left(\omega\right)\le T\left(n,s_n^+\right)\left(\omega\right),\quad \text{setting }T\left(n,T_n^+\right)=T.
	\end{equation}
	Thanks to \eqref{T_incr}, we can construct the function $f_n\colon \Xi\times [0,T]\to S_n$ defining, for all $\omega\in\Xi,$
	\[
		f_n\left(\omega, t\right)=\begin{cases}
			s, & t\in \left[T\left(n,s\right)\left(\omega\right), T\left(n, s_n^+\right)\left(\omega\right)\right),\,S_n\ni s<T,\\
			T, & t\in\left[T\left(n,T\right)\left(\omega\right),T\right].
		\end{cases}
	\]
	It is immediate to notice that $f_n$ is a càdlàg, increasing step function, while  \eqref{dist_T} is guaranteed by \eqref{T_space}.\\
	 It only remains to prove the composition property in \eqref{comp_f}. Consider $\omega\in\Xi$, two integers $n\ge m$ and $s\in S_{n+1}.$ Note that, by definition (omitting $\omega$ as before), $$T\left(k,n+1,s\right)\ge T\left(k,n,g_n\left(s\right)\right)\ge \dots \ge T\left(k,m,f_{m,n}\left(s\right)\right),\quad k>n, 
	$$ 
	 hence passing to the limit as $k\to \infty$,
	\begin{equation}\label{mon_J}
		T\left(n+1,s\right)\ge T\left(n,g_n\left(s\right)\right)\ge 	\dots \ge T\left(m,f_{m,n}\left(s\right)\right).
	\end{equation}
Take $t\in [0,T]$ and denote by $\tilde{s}\in S_{n+1}$ the unique element in $S_{n+1}$ such that $t\in [T(n+1,\tilde{s}),T(n+1,\tilde{s}_{n+1}^+)[$ (this interval is closed when $\tilde{s}=T$). Analogously, let $\bar{s}\in S_m$ be such that $t\in [T(m,\bar{s}),T(m,\bar{s}_{m}^+)[$, again closing the interval when $\bar{s}=T$. Since, by \eqref{mon_J}, $t\ge T(n+1,\tilde{s})\ge T(m,f_{m,n}(\tilde{s}))$, from \eqref{T_incr} we infer that $f_{m,n}(\tilde{s})\le\bar{s}$, whence 
\begin{equation}\label{in_1}
	f_{m,n}\left(f_{n+1}\left(t\right)\right)=f_{m,n}(\tilde{s})\le \bar{s}=f_m\left(t\right).
\end{equation}
On the other hand, we argue by cases. Firstly, note that $T(m,\bar{s})\le T(n+1,\bar{s})$ by \eqref{mon_J}. Secondly, we observe that either $t\ge T(n+1,\bar{s})$ or $t\in [T(m,\bar{s}), T(n+1,\bar{s}))$. In the former case, \eqref{T_incr} implies that $\bar{s}\le f_{m,n}(\tilde{s})$. In the latter (where in particular $\bar{s}\neq 0$),  there exists $u\in [T(m,\bar{s}), T(n+1,\bar{s})) \cap S_{k+1} $, for some integer $k>n$, such that 
\[
	\bar{s}=f_{m,k}\left(u\right)=f_{m,n}\left(f_{n+1,k}\left(u\right)\right).
\]
	Since $u<T(n+1,\bar{s})$, $f_{n+1,k}\left(u\right)\in (\bar{s}-T2^{-m},\bar{s})\cap S_{n+1}$. Hence we can define  $\bar{u}=\min\{v\in (\bar{s}-T2^{-m},\bar{s})\cap S_{n+1} : f_{m,n}(v)=\bar{s}\}$, and it is easy to argue by contradiction that $T(n+1, \bar{u})\le T(m,\bar{s}).$ Consequently,
\begin{equation}\label{in_2}
	f_{m,n}\left(f_{n+1}\left(t\right)\right)\ge f_{m,n}\left(\bar{u}\right)=\bar{s}=f_m\left(t\right).
\end{equation}
Combining \eqref{in_1} and \eqref{in_2} we obtain \eqref{comp_f}, completing the proof.
\end{proof}
We are now ready to prove Theorem \ref{thm_BZ}.
\begin{myproof}{Theorem \ref{thm_BZ}\!}
For every $n\in\mathbb{N}$, consider the finite partition $\{A^{(n)}_t,\,t\in S_{n+1}\setminus S_n\}$ of $\Xi$, defined by
\begin{align*}
	A^{(n)}_t&=\left\{\omega\in\Xi : \Delta\left(X_t\left(\omega\right), X_{g_n\left(\omega, t\right)}\left(\omega\right)\right)=\max_{u\in S_{n+1}}\Delta\left(X_u\left(\omega\right), X_{g_n\left(\omega, u\right)}\left(\omega\right)\right)\right\}
	\\&=
	\left(\Delta\left(X_t, X_{g_n\left(\cdot, t\right)}\right)-
	\max_{u\in S_{n+1}}\Delta\left(X_u, X_{g_n\left(\cdot, u\right)}\right)\right)^{-1}\left\{0\right\},\quad t\in S_{n+1}\setminus S_n.
\end{align*}
Noticing that  $\Xi\ni\omega\mapsto\max_{u\in S_{n+1}}\Delta\left(X_u\left(\omega\right), X_{g_n\left(\omega, u\right)}\left(\omega\right)\right)$ is $\mathcal{G}-$measurable because it is the maximum of $2^{n+1}+1$ random variables, we deduce that $A^{(n)}_t\in \mathcal{G}$ for every $t\in S_{n+1}\setminus S_n$. Then, by the hypothesis in~\eqref{class_BZ},
\begin{multline}\label{1.2.2.}
	\mathbb{E}_\mathbb{Q}\left[\max_{u\in S_{n+1}}\Delta\left(X_u, X_{g_n\left(\cdot, u\right)}\right)\right]
	=
	\sum_{t\in S_{n+1}\setminus S_n} \mathbb{E}_\mathbb{Q}\left[\Delta\left(t_n^-,t,t_n^+\right)1_{A^{(n)}_t}\right]
	\\
	\le 
	\delta\left(T2^{-n}\right)
	\sum_{t\in S_{n+1}\setminus S_n}\theta\left(\mathbb{Q}\left(A^{(n)}_t\right)\right)
	\le 
		\delta\left(T2^{-n}\right)
	2^n\theta\left(2^{-n}\right),
\end{multline}
where in the last step we use \eqref{prop_theta} and the fact that the cardinality of $S_{n+1}\setminus S_n$ is $2^n$. Now, since for every $\omega\in\Xi$ and $m,n\in\mathbb{N}$ such that $n\ge m$, (omitting $\omega$ to save space)
\begin{align*}
	\Delta\left(X_t,X_{f_{m,n}\left(t\right)}\right)&\le 
	\Delta\left(X_t,X_{g_{n}\left(t\right)}\right)
	+\Delta\left(X_{g_n\left(t\right)},X_{g_{n-1}\left(g_n\left(t\right)\right)}\right)+\dots
	+\Delta\left(X_{f_{m+1,n}\left(t\right)},X_{f_{m,n}\left(t\right)}\right)\\&
	\le
	\sum_{j=m}^{n}
	\max_{u\in S_{j+1}}\Delta\left(X_u, X_{g_{j}\left( u\right)}\right),\quad t\in S_{n+1},
\end{align*}
we obtain
\begin{equation}\label{up_boun_c}
	\sup_{n\ge m}\max_{t\in S_{n+1}}\Delta\left(X_t\left(\omega\right),X_{f_{m,n}\left(\omega,t\right)}\left(\omega\right)\right)
	\le 
	\sum_{n=m}^{\infty}
	\max_{u\in S_{n+1}}\Delta\left(X_u\left(\omega\right), X_{g_{n}\left( \omega, u\right)}\left(\omega\right)\right)\eqqcolon\mathbf{\upperRomannumeral{1}}_m\left(\omega\right). 
\end{equation}
Consider $m$ so big that $2^{1-m}\le T$ and denote by $\widebar{T}=T\vee 1$. From the integral condition in \eqref{cond_inte}, \eqref{1.2.2.} and recalling that the functions $\delta, \,\theta$ are non--decreasing, we invoke the dominated convergence theorem to conclude that
\begin{multline*}
	\mathbb{E}_\mathbb{Q}\left[\mathbf{\upperRomannumeral{1}}_m\right]
	\le 
\sum_{n=m}^{\infty}\delta\left(T2^{-n}\right)2^n\theta\left(2^{-n}\right)\le
\widebar{T}\sum_{n=m}^{\infty}\delta\left(\widebar{T}2^{-n}\right)\frac{2^n}{\widebar{T}}\theta\left(\widebar{T}2^{-n}\right) 	
\le 2\widebar{T}
\int_{m-1}^{\infty}\delta\left(\widebar{T}2^{-x}\right)\frac{2^x}{\widebar{T}}\theta\left(\widebar{T}2^{-x}\right) dx
\\=
\frac{2\widebar{T}}{\log 2}\int_{0}^{T}1_{\left\{y\le \widebar{T}2^{1-m}\right\}}y^{-2}\delta\left(y\right)\theta\left(y\right) dy\underset{m\to \infty}{\longrightarrow} 0.
\end{multline*}
Therefore, there exists a subsequence $(\mathbf{\upperRomannumeral{1}}_{m_k})_k$ such that $\lim_{k\to\infty}\mathbf{\upperRomannumeral{1}}_{m_k}=0$ almost uniformly in $\Omega$. This means that there exists a sequence of $\mathcal{G}-$measurable sets $(\Xi_N)_N$, with $\mathbb{Q}(\Xi_N)<N^{-1}$ and $\Xi_{N+1}\subset \Xi_N$, such that 
\[
	\lim_{k\to \infty}\sup_{\omega\in\Xi\setminus \Xi_N}\mathbf{\upperRomannumeral{1}}_{m_k}\left(\omega\right)=0.
\]
Using a diagonalization argument, we can find a subsequence $(\mathbf{\upperRomannumeral{1}}_{m_{k_p}})_p$ with the following property:
\[
	\mathbf{\upperRomannumeral{1}}_{m_{k_q}}\left(\omega\right)<2^{-q},\quad \omega\in\Xi\setminus \Xi_p,\,q\ge p,  \text{ for every }p\in\mathbb{N}.
\]
Let us define the almost sure event $\Xi_0=\Xi\setminus\cap_{N=1}^\infty\Xi_N$: note that for any $\omega\in\Xi_0$, $\omega\in \Xi\setminus\Xi_N$ for all $N$ sufficiently large. Going back to \eqref{up_boun_c}, the previous relation gives
\[
	\sup_{n\ge m_{k_p}}\max_{t\in S_{n+1}}\Delta\left(X_t\left(\omega\right),X_{f_{m_{k_p},n}\left(\omega,t\right)}\left(\omega\right)\right)\le \mathbf{\upperRomannumeral{1}}_{m_{k_p}}\left(\omega\right)<2^{-p},
	\quad \omega\in \Xi_0,\,p\ge p\left(\omega\right)\in \mathbb{N},
\]
which is equivalent to writing, for every $\omega\in \Xi\setminus \Xi_{p(\omega)}$, 
\[
	\Delta\left(X_t\left(\omega\right),X_{f_{m_{k_p},n}\left(\omega,t\right)}\left(\omega\right)\right)
	<2^{-p},\quad t\in S_{n+1},\,n\ge m_{k_{p}},\,p\ge p\left(\omega\right).
\]
The composition property in Lemma \ref{l1.3} (see \eqref{comp_f}) yields
\begin{equation}\label{ok_cauchy}
	\Delta\left(X_{f_{n+1}\left(\omega,t\right)}\left(\omega\right),X_{f_{m_{k_p}}\left(\omega,t\right)}\left(\omega\right)\right)<2^{-p},\quad t\in\left[0,T\right],\,n\ge m_{k_p},\,p\ge p\left(\omega\right).
\end{equation}
Note that $X_{f_n(\omega,\cdot)}\left(\omega\right)\colon[0,T]\to E$ is càdlàg for every $n\in\mathbb{N}$ and $\omega\in \Xi$ because $f_n(\omega,\cdot)$ is a càdlàg, step function. Thus, \eqref{ok_cauchy} shows that  the sequence $(X_{f_{m_{k_p}}(\omega, \cdot)}(\omega))_p$ is Cauchy in the metric space of the $E-$valued  càdlàg functions defined in $[0,T]$  endowed with the uniform distance. This space is complete since $E$ is complete, hence there exists a càdlàg function $\widetilde{X}(\omega)\colon[0,T]\to E$ such that \[\lim_{p\to\infty}\sup_{0\le t \le T}\Delta\left(\widetilde{X}\left(\omega\right)\left(t\right),X_{f_{m_{k_p}}(\omega,t)}(\omega)\right)=0,\quad \omega\in \Xi_0.
\]
Finally, we define the function  $Z\colon\Xi \times [0,T]\to E$ by
\[
	Z_t\left(\omega\right)=\begin{cases}
		\widetilde{X}\left(w\right)(t),& \omega\in \Xi_0,\\
		0,&\text{otherwise}.
	\end{cases}
\]
By construction $Z$ is càdlàg. Moreover, notice that $f_n(\omega,t)\in\{t_n^-,t_n^+\}$ by \eqref{dist_T}, for every $\omega\in\Xi,\,n\in\mathbb{N}$ and $t\in [0,T]$. Recall that $X$ is continuous in probability, hence for every $t\in[0,T]$ there exists a full probability set $\Xi_t$ and a subsequence $(m_{k_{p_q}})_q$ depending on $t$ such that, denoting by $\tilde{q}=\tilde{q}(q)=m_{k_{p_q}}$, 
\[
	\lim_{q\to \infty}\Delta\left(X_t\left(\omega\right),X_{t_{\tilde{q}}^-}\left(\omega\right)\right)=
	\lim_{q\to \infty}\Delta\left(X_t\left(\omega\right),X_{t_{\tilde{q}}^+}\left(\omega\right)\right)
	=0,\quad \omega\in\Xi_t.
\] 
Then, for all $\omega\in \Xi_0\cap \Xi_t$, 
\begin{equation*}
	\Delta\!\left(X_t\left(\omega\right)\!,Z_t\left(\omega\right)\right)\!
	=\!\lim_{q\to \infty}\!	\Delta\Big(\!X_t\left(\omega\right),X_{f_{m_{k_{p_q}}}\left(\omega,t\right)}\left(\omega\right)\!\!\Big)
	\le 
	\lim_{q\to \infty}
	\Delta\left(X_t\left(\omega\right),X_{t_{\tilde{q}}^-}\left(\omega\right)\right)
	\!+\!
	\Delta\left(X_t\left(\omega\right),X_{t_{\tilde{q}}^+}\left(\omega\right)\right)
	=0.
\end{equation*}
In conclusion, $Z$ is the càdlàg version of $X$ that we are looking for as $\mathbb{Q}(\Xi_0\cap \Xi_t)=1$. The proof is now complete.
\end{myproof}
 
\section{Proofs of Lemma \ref{joint_meas}-\ref{le_se}}\label{ap_lemmi}
\begin{myproof}{Lemma \ref{joint_meas}}
	In order not to complicate the notation, we carry out the proof in the case $\bar{s}={\bar{t}}=T$, hence $s,t\in [0,T]$. The assumption $\bar{t}=T$ is not restrictive because the $\mathbb{R}^d-$valued random variables $Z^{s,x}_t$ are $\mathcal{F}_{\bar{t}}-$measurable for every $t\in [0,{\bar{t}}]$, so that all the following passages can be easily adapted to treat a general ${\bar{t}}\in [0,T]$.  As for $\bar{s}\in[0,T]$, this case can be treated by  considering only dyadic points in $[0,T]$ up to $\bar{s}$ in the procedure that we are about to explain.
	
	Consider the function $g_1\colon (\Omega\times [0,T],\mathcal{F}_T\otimes \mathcal{B}([0,T]))\to (\mathcal{C}_0,\mathcal{C})$ defined by $g_1(\omega,s)=Z_s(\omega)$. Let $S_n=\{Tj2^{-n}, j=0,\dots,2^n\}$ be the set of dyadic points and denote by $s_n^+=\min\{t\in S_n : t\ge s\},\,s\in [0,T]$. Then, we define $g_{1,n}(\omega,s)=g_1(\omega, s_n^+)$. For every $A\in\mathcal{B}(\mathbb{R}^d)$, $n\in\mathbb{N}$, $x\in\mathbb{R}^d$ and $t\in [0,T]$, we have
	\begin{multline*}
		g_{1,n}^{-1}\left(\pi_x^{-1}\left(\pi_t^{-1}\left(A\right)\right)\right)
		=\left(\left(Z^{0,x}_t\right)^{-1}\left(A\right)\times\left\{0\right\} \right)\\\bigcup\bigg(\bigcup_{s\in S_n\setminus \{0\}}\left(\left(Z^{s,x}_t\right)^{-1}\left(A\right)\times \left(s-T2^{-n},s\right]\right)\bigg)
		\in  \mathcal{F}_T\otimes\mathcal{B}\left(\left[0,T\right]\right).
	\end{multline*}
	Recalling that $\mathcal{C}$ is the $\sigma-$algebra generated by $\pi_x\colon \mathcal{C}_0\to(\mathcal{D}_0,\mathcal{D}),\,x\in\mathbb{R}^d$, and that $\mathcal{D}$ is the $\sigma-$algebra generated by $\pi_t\colon\mathcal{D}_0\to \mathbb{R}^d,\,t\in [0,T]$, the previous computation shows that $g_{1,n}$ is $\mathcal{F}_T\otimes\mathcal{B}([0,T])/\mathcal{C}-$measurable. 
	The càdlàg property of $g_1(\omega,\cdot)$ ensured by Theorem \ref{cadlag} yields $\lim_{n\to \infty}g_{1,n}=g_1$ pointwise in $\Omega\times [0,T]$, hence $g_1$ is measurable, as well. Next, consider $g_2\colon (\mathcal{C}_0\times \mathbb{R}^d,\mathcal{C}\otimes \mathcal{B}(\mathbb{R}^d))\to (\mathcal{D}_0,\mathcal{D})$ given by $g_2(f,x)=\pi_x(f)=f(x)$. For every $n\in\mathbb{N},$ let $\Pi_n=\{2^{-n}z,  z\in \mathbb{Z}^d\}$ be the set of lattice points in $\mathbb{R}^d$ with mesh $2^{-n}$.  Denoting by $\bar{x}_n=2^{-n}[2^nx]\in\Pi_n,\,x\in\mathbb{R}^d$, we define $g_{2,n}(f,x)=g_2(f,\bar{x}_n)$. Note that $g_{2,n}$ is $\mathcal{C}\otimes \mathcal{B}(\mathbb{R}^d)/\mathcal{D}-$measurable for all $n\in\mathbb{N}.$  Indeed, for every $A\in \mathcal{B}( \mathbb{R}^d)$ and $t\in [0,T]$,
	\begin{align*}
		g_{2,n}^{-1}\left(\pi_t^{-1}\left(A\right)\right)
		&=\left\{(f,x)\in \mathcal{C}_0\times \mathbb{R}^d : [g_{2,n}(f,x)](t)\in A\right\}
		\\&
		=
		\bigcup_{x\in\Pi_n}
		\left(\pi_x^{-1}\left(\pi_t^{-1}(A)\right)
		\times 	
		\left(\left[x_1+2^{-n}\right)\times\dots\times\left[x_d+2^{-n}\right) \right)\right)\in \mathcal{C}\otimes \mathcal{B}\left(\mathbb{R}^d\right).
	\end{align*}
	Since, by continuity,  $\lim_{n\to \infty}g_{2,n}=g_2$ pointwise in $\mathcal{C}_0\times \mathbb{R}^d$, we conclude that $g_2$ is measurable. Finally, we introduce the map $g_3\colon (\mathcal{D}_0\times  [0,T], \mathcal{D}\otimes \mathcal{B}([0,T]))\to \mathbb{R}^d$ defined by $g_3(f,t)=f(t)$: arguing  as we have done for $g_1$, we infer that $g_3$ is measurable. At this point, we read the function $Z$ in the statement of this lemma as the following composition, where $\text{Id}\colon [0,T] \to [0,T]$ and $\text{Id}_d\colon \mathbb{R}^d\to \mathbb{R}^d$ are the identity maps:
	\[
	Z=g_3\circ (g_2, \text{Id})\circ (g_1, \text{Id}_d, \text{Id}).
	\]
	The previous argument allows then to deduce that $Z$ is $\mathcal{F}_{T}\otimes \mathcal{B}([0,T]\times \mathbb{R}^d\times [0,T])-$measurable.
\end{myproof}

\begin{myproof}{Lemma \ref{le_se}}
	{\color{black} Since \eqref{serve_1} can be obtained from \eqref{SDEJ} by setting $f=0$, the existence of a pathwise unique solution of \eqref{serve_1} can be argued as in Remark \ref{nonbanale}. Thus, we only focus on showing that the process $Z^{s,\eta}$ solves \eqref{serve_1}.}
	
	All the assertions of the lemma are trivially satisfied when $s=T$. Thus, we fix $s\in[0,T)$ and take $\eta\in L^0(\mathcal{F}_s)$. By construction, $Z^{s,\eta}_t=\eta$ for every $t\in [0,s],\,\mathbb{P}-$a.s., so we  only focus on $t\in [s,T]$. First of all, notice that $Z^{s,\eta}_t$ is $\mathcal{F}_t-$measurable by Lemma \ref{joint_meas}. Next, consider a sequence of simple, $\mathcal{F}_s-$measurable, $\mathbb{R}^d-$valued random variables $(\eta_n)_n$ such that $\eta_n\to \eta$ as $n\to\infty$, $\mathbb{P}-$a.s. Specifically, for every $n\in\mathbb{N}$, let  $\eta_n=\sum_{k=1}^{N_n}a_k^n1_{A_k^n}$, for some $N_n\in\mathbb{N}$,  $(a^n_k)_k\subset \mathbb{R}^d$ and some partition $(A_k^n)_k\subset \mathcal{F}_s,\,k=1,\dots, N_n$. By \eqref{common_eq}, using \cite[Section 3, Chapter \upperRomannumeral{2}]{IW}  and 
	\cite[Property 4.37, Chapter \upperRomannumeral{1}]{js}	we have, $\mathbb{P}-$a.s., for any $n\in\mathbb{N},$
	\begin{align}\label{limit_semplice}
		\notag&Z_t^{s, \eta_n}=
		\sum_{k=1}^{N_{n}}Z_t^{s,a^n_k}1_{A^n_k}
		\\&\notag=\sum_{k=1}^{N_{n}}\Bigg[
		{a^n_k}+
		\int_{s}^{t}b\left(r, Z_{r}^{s,a^n_k}\right) dr + 
		\int_{s}^{t}\alpha\left(r, Z_{r}^{s,a^n_k} \right) dW_r+
		\int_{s}^{t}\!\int_{U_0}g\left(Z_{r-}^{s,a^n_k},r,z\right)\widetilde{N}_p\left(dr,dz\right)\Bigg]1_{A^n_k}\\&
		=\eta_n+ 
		\int_{s}^{t}b\left(r, Z_{r}^{s,\eta_n}\right) dr + 
		\int_{s}^{t}\alpha\left(r, Z_{r}^{s,\eta_n} \right) dW_r+
		\int_{s}^{t}\!\int_{U_0}g\left(Z_{r-}^{s,\eta_n},r,z\right)\widetilde{N}_p\left(dr,dz\right)
		,\quad t\in[s,T].
	\end{align}
	In order to recover \eqref{serve_1}, we want to take  limits in \eqref{limit_semplice} as $n\to\infty$. From the continuity of $Z^{s,x}_t$ in $x$ (see \ref{conx} in Definition \ref{sharpala}), we infer that $\lim_{n\to \infty} Z^{s,\eta_n}_t=Z^{s,\eta}_t$ uniformly in $t\in [s,T],\,\mathbb{P}-$a.s.  Next, by dominated convergence and \eqref{lip1}, 
	\begin{equation}\label{b1}
	\lim_{n\to\infty}\int_{s}^{t}b\left(r, Z_{r}^{s,\eta_n}\right) dr
	=
	\int_{s}^{t}b\left(r, Z_{r}^{s,\eta}\right) dr,\quad t\in [s,T],\,\text{$\mathbb{P}-$a.s.}
	\end{equation}
	The convergence of the stochastic integrals in \eqref{limit_semplice} is studied via a localization procedure. As for the integral with respect to the Brownian motion, for every $\epsilon>0$ we define
	\begin{equation*}
		\sigma_n(\epsilon)=\inf\left\{
		u\in [s,T] :  \int_{s}^{u}
		\left|\alpha\left(r, Z_{r}^{s,\eta_n} \right)
		-\alpha\left(r, Z_{r}^{s,\eta} \right)\right|^2
		dr \ge \epsilon
		\right\},\quad \text{with }\inf\emptyset=\infty.
	\end{equation*}
	Since
	$
	\lim_{n\to\infty}\int_{s}^{T}\left|\alpha\left(r, Z_{r}^{s,\eta_n} \right)
	-\alpha\left(r, Z_{r}^{s,\eta} \right)\right|^2dr = 0,\, \text{$\mathbb{P}-$a.s.},
	$
	we have $\sigma_n(\epsilon)\to \infty$ as $n\to \infty$, $\mathbb{P}-$a.s. In particular, $\lim_{n\to \infty}\mathbb{P}\left(\sigma_n\left(\epsilon\right)\le  T\right)=0$.  Hence by Markov's inequality, for every $\delta>0$, { for some $c>0$,} 
	\begin{gather*}
	\mathbb{P}\Big(\sup_{t\in [s,T]}\left|\int_{s}^{t} 
		\left(\alpha\left(r, Z_{r}^{s,\eta_n} \right)-\alpha\left(r, Z_{r}^{s,\eta} \right)\right)
		dW_r\right|\ge \delta\Big)
		\\ \qquad \le \frac{{ c}}{\delta^2}
		\mathbb{E}\bigg[\int_{s}^{T}1_{\left[s,\sigma_n(\epsilon)\right]}(r)\left|\alpha\left(r, Z_{r}^{s,\eta_n} \right)-\alpha\left(r, Z_{r}^{s,\eta} \right)\right|^2dr\bigg] 
		\!\!+\mathbb{P}\left(\sigma_n\left(\epsilon\right)\le  T\right)		 
		\le \frac{c\epsilon}{\delta^2}+\mathbb{P}\left(\sigma_n\left(\epsilon\right)\le  T\right)
		,\quad \epsilon>0,
	\end{gather*}
	which proves that $\lim_{n\to\infty}\int_{s}^{t}\alpha\left(r, Z_{r}^{s,\eta_n} \right)dW_r=\int_{s}^{t}\alpha\left(r, Z_{r}^{s,\eta} \right)dW_r$ uniformly on $[s,T]$ in probability. This in turn yields the existence of a subsequence such that
	\begin{equation}\label{b2}
		\lim_{k\to\infty}\int_{s}^{t}\alpha\left(r, Z_{r}^{s,\eta_{n_k}} \right)dW_r=\int_{s}^{t}\alpha\left(r, Z_{r}^{s,\eta} \right)dW_r,\quad \text{uniformly in } t\in\left[s,T\right],\,\mathbb{P}-\text{a.s.}
	\end{equation}
	The convergence of the integral with respect to $\widetilde{N}_p$ in \eqref{limit_semplice} can be treated analogously. More precisely, by~\eqref{lip1},
	$
		\lim_{n\to \infty} 
		\int_{s}^{T}\!\int_{U_0}\left|g\left(Z_{r-}^{s,\eta_n},r,z\right)-g\left(Z_{r-}^{s,\eta},r,z\right)
		\right|^2\nu(dz)\, dr=0,\, \mathbb{P}-\text{a.s.}
	$
	If we introduce  the stopping times
	\[
	\widetilde{\sigma}_n(\epsilon)=\inf\left\{
	u\in [s,T] :  \int_{s}^{u}\!\int_{U_0}\left|g\left(Z_{r-}^{s,\eta_n},r,z\right)-g\left(Z_{r-}^{s,\eta},r,z\right)
	\right| ^2 \nu(dz)\,dr \ge \epsilon
	\right\},\quad \epsilon>0,
	\]
	then we can proceed as before to deduce that, $\mathbb{P}-$a.s., 
	\begin{equation}\label{b3}
		\lim_{k\to\infty}\int_{s}^{t}\!\int_{U_0}g\left(Z_{r-}^{s,\eta_{n_k}},r,z\right)\widetilde{N}_p(dr,dz)
		=\int_{s}^{t}\!\int_{U_0}g\left(Z_{r-}^{s,\eta},r,z\right)\widetilde{N}_p(dr,dz),\quad\text{uniformly in } t\in\left[s,T\right]. 
	\end{equation}
	Thus, combining \eqref{b1}-\eqref{b2}-\eqref{b3}, we can pass to the limit in \eqref{limit_semplice} along a suitable subsequence to get~\eqref{serve_1}. 
	
	The equalities in \eqref{dep_Zmeas} are obtained using \eqref{b1}-\eqref{b2}-\eqref{b3} and recalling the continuity of $Z^{s,x}_{i}$ in the space variable $x$.
	The lemma is now completely proved.
\end{myproof}
\end{appendices} 
 
\vskip 3mm  \noindent {\bf Acknowledgments.}  The authors wish to
thank  M. Fuhrman for useful
comments on a preliminary version of the DPP.

\end{document}